\documentclass[12pt]{amsart}
\usepackage[dvipsnames]{color}
\usepackage{amsfonts,amssymb,amsmath,amscd,amstext}
 \usepackage{esint}
\usepackage[colorlinks=true,linkcolor=blue,citecolor=blue]{hyperref}
\usepackage[utf8]{inputenc}
\usepackage{graphicx}
\usepackage{tikz}
\usepackage{esint}
\usepackage{changes}
\usepackage{comment}
\usepackage{cleveref}
\usepackage{mathdots}
\usepackage[a4paper,top=2.5cm,bottom=2.5cm,left=2cm,right=2cm]{geometry}
\usepackage[style=alphabetic,maxalphanames=5,maxnames=5]{biblatex}
\newcommand{\R}{\mathbb R}

\newcommand{\eps}{\epsilon}

\renewcommand{\H}{\mathbb H}

\newcommand{\C}{\mathbb C}

\newcommand{\diam}{{\rm diam}}

\newcommand{\average}{{\mathchoice {\kern1ex\vcenter{\hrule height.4pt
width 6pt depth0pt} \kern-9.7pt} {\kern1ex\vcenter{\hrule height.4pt width
4.3pt depth0pt}
\kern-7pt} {} {} }}

\newcommand{\dist}{\mbox{dist}}
\newcommand{\N}{\mathbb N}

\renewcommand{\rho}{\varrho}
\renewcommand{\epsilon}{\varepsilon}

\renewcommand{\theta}{\vartheta}
\numberwithin{equation}{section}
\theoremstyle{plain}
\newtheorem{theorem}{Theorem}[section]
\newtheorem{corollary}[theorem]{Corollary}
\newtheorem{lemma}[theorem]{Lemma}
\newtheorem{example}[theorem]{Example}
\newtheorem{proposition}[theorem]{Proposition}
\theoremstyle{definition}

\newtheorem{definition}[theorem]{Definition}
\newtheorem{definitionT}[theorem]{Definition and Theorem}

\theoremstyle{remark}
\newtheorem{remark}[theorem]{Remark}
\begin{document}

\title[Uniform rectifiability in Heisenberg groups]
{On low-dimensional uniform rectifiability\\ in Heisenberg groups 
}
\author{Katrin Fässler}
\address{Department of Mathematics and Statistics\\ University of Jyv\"askyl\"a \\ P.O. Box 35 (MaD),
FI-40014 University of Jyv\"askyl\"a, Finland}

\email{katrin.s.fassler@jyu.fi}
\author{Andrea Pinamonti}
\address{Dipartimento di Matematica, Università di Trento\\
Via Sommarive, 14, 38123 Povo TN, Italy}

\email{andrea.pinamonti@unitn.it}
\author{Kilian Zambanini}
\address{Dipartimento di Matematica, Università di Trento\\
Via Sommarive, 14, 38123 Povo TN, Italy}
\email{kilian.zambanini@unitn.it}

\thanks{This work was initiated during a visit of K.Z. to the University of Jyväskylä, which was partially financed by the Research Council of Finland, grant 
352649. K.F. was partially supported by the Academy of Finland (now: Research Council of Finland)
grant no.\ 321696. A.P. and K.Z. are members of the Istituto Nazionale di Alta Matematica (INdAM), Gruppo Nazionale per l'Analisi Matematica, la Probabilità e le loro Applicazioni (GNAMPA), and both are supported by the University of Trento, the MIUR-PRIN 2022 Project \emph{Regularity problems in sub-Riemannian structures}  Project code: 2022F4F2LH and the INdAM-GNAMPA 2025 Project \emph{Structure of sub-Riemannian hypersurfaces in Heisenberg groups}, CUP ES324001950001.}

\date{\today}

\begin{abstract} Refining an earlier result due to Hahlomaa, we provide a new Carleson-type condition for $k$-regular sets in the Heisenberg group $\mathbb{H}^n$ to have big pieces of Lipschitz images of subsets of $\mathbb{R}^k$ for $1\leq k\leq n$. Our approach passes via the corona decompositions by normed spaces, recently introduced by Bate, Hyde, and Schul. Along the way, we prove implications between several notions of quantitative rectifiability for low-dimensional sets in $\mathbb{H}^n$.
%and we discuss connections 
\end{abstract}
\maketitle

%\tableofcontents

\section{Introduction}\label{s:Intro}

This paper concerns uniformly $k$-rectifiable sets in sub-Riemannian Heisenberg groups $\mathbb{H}^n$ for $1\leq k\leq n$. Quantitative multiscale methods for the study of (uniform) rectifiability in $\mathbb{H}^n$ have been mainly developed for sets of dimension $1$ (e.g., \cite{MR2371434,Juillet,LiSchul2,
LiSchul1,chouli,CLZ19,MR4299821,Li,MR4669029,
%didonato
}) or codimension $1$ (e.g., \cite{MR3682744,MR3815462,chousionis,2019arXiv190406904R,MR4127898,MR4460594,
%MR4514458,
CLY1,2022arXiv220703013C,CLY2,MR4605201,2025arXiv251026934H}). A notable exception is Hahlomaa's preprint \cite{Hahlomaa}, which gave the initial impetus for the present work. On the other hand, in abstract metric spaces, a recent breakthrough by Bate, Hyde, and Schul \cite{Bate} laid the foundations for a general theory of uniform rectifiability by Lipschitz images from $\mathbb{R}^k$, in spirit of David and Semmes' work in Euclidean spaces \cite{David1,David2}.
%, in spirit of the theory by David and Semmes \cite{}

\subsection*{A new sufficient condition for BPLI} We provide a new sufficient Carleson-type condition for a set in $\mathbb{H}^n$ to have \emph{big pieces of Lipschitz images (BPLI)} from $\mathbb{R}^k$, for $1\leq k\leq n$, which implies the result in \cite{Hahlomaa} as a corollary. Along the way, we introduce \emph{corona decompositions by horizontal planes} and relate them to other notions of corona decompositions  in the literature \cite{FO,didonato,Bate}. At the same time, we streamline the proof in Hahlomaa's paper by applying parts of the general metric space theory \cite{Bate}. This supplements the results in  \cite{Bate} with a family of concrete examples, namely $k$-regular sets in the non-Euclidean spaces $(\mathbb{H}^n,d)$ for $1\leq k \leq n$, where characterizations of the BPLI property % in the sense of Definition \ref{def_BPLI} 
can be meaningfully studied. Every Lipschitz function $f:A\subset \mathbb{R}^k \to (\mathbb{H}^n,d)$ is locally Lipschitz with respect to the Euclidean distance on $\mathbb{H}^n$, but only Euclidean Lipschitz maps which are \emph{horizontal} with respect to the sub-Riemannian structure on $\mathbb{H}^n$ are Lipschitz in the Heisenberg metric $d$ \cite{MR2659687}. 
In this setting, we find new sufficient \emph{Carleson measure} conditions for BPLI that
rely on the structure of the ambient $\mathbb{H}^n$ and  are not covered by the characterizations of BPLI in \cite{Bate}.

\begin{theorem}\label{t:ImprovedHahlomaa}
Let $n\in \mathbb{N}$ and $k\in \{1,\ldots,n\}$. If $E\subset \mathbb{H}^n$ is a $k$-regular set with the properties $E\in \mathrm{GLem}(\beta_{1,\pi,A(2n,k)},2)$ and $E\in \mathrm{WGL}(\beta_{\infty,\mathcal{V}_k})$, then $E$ has big pieces of Lipschitz images of subsets of $\mathbb{R}^k$. 
\end{theorem}

This also holds with ``Lipschitz'' replaced by ``bi-Lipschitz''.
The restriction to $k\leq n$ is natural since  $(\mathbb{H}^n,d)$ is purely $k$-unrectifiable in Federer's sense for $k>n$ \cite{MR1800768,MR2105335}. In the range $k>n$,
uniform $k$-rectifiability should be built on intrinsic Lipschitz graphs \cite{MR2313532} instead of metric Lipschitz images from $\mathbb{R}^k$, and this is a direction which we do not pursue here.

The concepts in Theorem \ref{t:ImprovedHahlomaa} will be properly introduced in Definitions \ref{d:k-reg}, \ref{def_BPLI}, \ref{d:Coefficient functions}, \ref{d:GL}, and \ref{d:WGL}. For now, we simply remark that the \emph{weak geometric lemma} $\mathrm{WGL}(\beta_{\infty,\mathcal{V}_k})$ is a way of quantifying that, for each precision level $\varepsilon >0$, the set $E$ stays $\varepsilon$-close -- in a scale-invariant way -- to horizontal $k$-planes in all balls centered on $E$, except perhaps for a family of balls satisfying a Carleson packing condition. It is known in Euclidean spaces that a weak geometric lemma alone does not even imply rectifiability \cite[\S 20]{David1} since it does not sufficiently control how much the best-approximating planes at a point can spin around as one considers smaller and smaller scales. In the proof of Theorem \ref{t:ImprovedHahlomaa}, the \emph{geometric lemma} $\mathrm{GLem}(\beta_{1,\pi,A(2n,k)},2)$  allow us to control the amount of spinning for certain approximating horizontal $k$-planes. Remarkably, unlike the coefficients $\beta_{\infty,\mathcal{V}_k}$, the $\beta_{1,\pi,A(2n,k)}(x,r)$-numbers, for $x\in E$ and $r>0$, only measure how well the \emph{projection} $\pi(E\cap B(x,r))\subset \mathbb{R}^{2n}$ is approximated by \emph{arbitrary} affine $k$-planes in $\mathbb{R}^{2n}$, but this turns out to be sufficient in combination with the $\mathrm{WGL}(\beta_{\infty,\mathcal{V}_k})$ assumption.

\subsection*{Connections with geometric lemmas in the literature.} Our proof of Theorem \ref{t:ImprovedHahlomaa} is modeled on Hahlomaa's paper \cite{Hahlomaa}, which contains a version of the result under a stronger assumption. We recover this as a corollary  of Theorem \ref{t:ImprovedHahlomaa}:
\begin{corollary}[Hahlomaa]\label{c:HahlomaaWithHorizontal}
         Let $n\in \mathbb{N}$ and $k\in \{1,\ldots,n\}$. If $E\subset \mathbb{H}^n$ is a $k$-regular set such that $E\in \mathrm{GLem}(\beta_{1,\mathcal{V}_k},2)$, then $E$ has big pieces of Lipschitz images of $\mathbb{R}^k$. 
    \end{corollary}
In Section \ref{s:Juillet} we present an example of a set $E\subset \mathbb{H}^1$ to demonstrate that the assumption in  Theorem \ref{t:ImprovedHahlomaa}  is strictly weaker than the one in Corollary \ref{c:HahlomaaWithHorizontal}. The set $E$ is the rectifiable curve $\Gamma$ constructed by Juillet in \cite{Juillet}. Juillet used this example to prove that the 
sufficient condition stated in \cite{MR2371434} for a subset of $\mathbb{H}^1$ to be contained in a rectifiable curve in the spirit of the \emph{analyst's traveling salesman theorem} is not a necessary condition. To apply this example in our context, we first verify that the curve $\Gamma$ from \cite{Juillet} is in fact $1$-regular. It is then easy to see from \cite{Juillet} that $\Gamma \notin \mathrm{GLem}(\beta_{\infty,\mathcal{V}_1},2)$, that is, $\Gamma$ violates a geometric lemma with exponent $2$ for $L^{\infty}$-based horizontal  $\beta$-numbers; analogous examples could be constructed for any exponent $p<4$, cf. \cite{Li}. However, this does not a priori rule out the possibility that $\Gamma$ could satisfy a corresponding geometric lemma with \emph{$L^1$-based} horizontal $\beta$-numbers as in Corollary \ref{c:HahlomaaWithHorizontal}, at least not by standard estimates such as \cite[(5.4)]{David1}. The main work in Section \ref{s:Juillet} consists in verifying that the curve $\Gamma$ is so badly approximable by horizontal lines that it even fails the $L^1$-based geometric lemma $\mathrm{GLem}(\beta_{1,\mathcal{V}_1},2)$. On the other hand, it follows from a (quantification) of Li's work \cite{Li}  that $\Gamma$ satisfies the assumptions of our Theorem \ref{t:ImprovedHahlomaa}.
This example also shows that a characterization of BPLI by means of a geometric lemma $\mathrm{GLem}(\beta_{1,\mathcal V_k},2)$ with exponent $2$, analogous to the Euclidean one \cite{David1}, cannot hold in $\H^n$.

\medskip

In fact, Theorem \ref{t:ImprovedHahlomaa} was partially motivated by the quest to provide sufficient conditions for the BPLI property in $\mathbb{H}^n$ that can be checked assuming one of the geometric lemmas in the existing literature. We present two such applications; the first one is related to the \emph{stratified $\hat\beta$-numbers} in \cite{Li}, the second one to the \emph{projection $\iota$-numbers} in \cite{FV2}.

\begin{corollary}\label{c:HahlomaaWithStratifASs}
    Let $n\in \mathbb{N}$ and $k\in \{1,\ldots,n\}$. If $E\subset \mathbb{H}^n$ is a $k$-regular set such that $E\in \mathrm{GLem}(\widehat{\beta}_{1,\mathcal{V}_k},4)$, then $E$ has big pieces of Lipschitz images of $\mathbb{R}^k$. 
\end{corollary}

The curve constructed in \cite{Juillet} does not meet the assumption in Corollary \ref{c:HahlomaaWithHorizontal}, while it satisfies the assumption in Corollary \ref{c:HahlomaaWithStratifASs}.
\begin{theorem}\label{t:JuilletCurveInNotIn}
    There is a $1$-regular curve $\Gamma\subset \mathbb{H}^1$ with  $\Gamma\notin \mathrm{GLem}(\beta_{1,\mathcal{V}_1},2)$ yet $\Gamma \in \mathrm{GLem}(\widehat{\beta}_{1,\mathcal{V}_1},4)$.
\end{theorem}
%; see Theorems \ref{gammaglem} and \ref{t:JuilletCurveGLem}, respectively.
The $\widehat{\beta}_{1,\mathcal{V}_k}$-numbers are an $L^1$-based version of the $\hat \beta$-numbers introduced in
\cite{Li} for subsets of arbitrary Carnot groups. In this paper, Li proved the following \emph{traveling salesman theorem}. A set $E$ in a Carnot group $\mathbb{G}$ of step $s$ and homogeneous dimension $Q$ lies on a finite length rectifiable curve if and only if
\begin{displaymath}
 \mathrm{diam}(E) +   \int_0^{\infty} \int_{\mathbb{G}}\hat \beta_E(x,r)^{2s}\,dx\frac{dr}{r^Q}<\infty.
\end{displaymath}
The exponent ``$4$'' in Corollary \ref{c:HahlomaaWithStratifASs} corresponds to $2s$ in case of 
$\mathbb{G}=\mathbb{H}^n$. Corollary \ref{c:HahlomaaWithStratifASs} can be interpreted as partial generalization of  a quantitative version of Li's result \cite{Li} in $\mathbb{H}^n$ from $k=1$ to $1\leq k\leq n$. 
Namely, a $1$-regular set  $E\subset \mathbb{H}^n$ with  $E\in \mathrm{GLem}(\widehat{\beta}_{\infty,\mathcal{V}_1},4)$ satisfies  $E\in \mathrm{GLem}(\widehat{\beta}_{1,\mathcal{V}_1},4)$ and consequently BPLI by
Corollary \ref{c:HahlomaaWithStratifASs}. It is known that every $1$-regular BPLI set in a complete, doubling, quasiconvex metric space is contained in a connected $1$-regular set, see \cite[Corollary 4.6]{FV1}. In conclusion, every $1$-regular set $E\in \mathbb{H}^n$ with
$E\in \mathrm{GLem}(\widehat{\beta}_{\infty,\mathcal{V}_1},4)$ can be covered by a $1$-regular connected set $\Gamma$.
This fact could also be deduced from \cite{Li}, using the argument in \cite[Section 3]{FV1}, but Corollary \ref{c:HahlomaaWithStratifASs} offers an alternative approach that provides information also for $k>1$.

Finally, we mention an application of Theorem \ref{t:ImprovedHahlomaa} to the $\iota_{1,\mathcal{V}_k}$-numbers introduced by the first author together with Violo in \cite{FV2}. Roughly speaking, these coefficients are small for a $k$-regular set $E$ at point $x\in E$ and scale $r>0$, if there exists a projection $P_V$ onto a horizontal $k$-plane $V$ such that $P_V$ is, in $L^1$-norm, close to an isometry on $E\cap B(x,r)$; see Definition \ref{d:Coefficient functions}.
In particular, if $\iota_{1,\mathcal{V}_k}(x,r)$ is very small for  $E$, the metric on  $E\cap B(x,r)$ is almost isometric to the Euclidean distance on $\mathbb{R}^k$ under the map $P_V$.

\begin{corollary}\label{c:HahlomaaWithIotaASs}
     Let $n\in \mathbb{N}$ and $k\in \{1,\ldots,n\}$. If $E\subset \mathbb{H}^n$ is a $k$-regular set such that $E\in \mathrm{GLem}(\iota_{1,\mathcal{V}_k},1)$, then $E$ has big pieces of Lipschitz images of $\mathbb{R}^k$. 
\end{corollary}

Corollary \ref{c:HahlomaaWithIotaASs} for $k=1$ also follows from  \cite[Thm 1.4]{FV1} since the $\iota_{1,\mathcal{V}_1}$-numbers are bounded from below by the $L^1$-based Gromov-Hausdorff numbers studied in \cite{FV1}. While applicable in more general metric spaces, the approach in \cite{FV1} passes via Menger curvature and is restricted to $1$-dimensional sets, whereas Theorem \ref{t:ImprovedHahlomaa} provides an approach via corona decompositions that yields Corollary  \ref{c:HahlomaaWithIotaASs} for all $1\leq k\leq n$. The latter assumes a geometric lemma with 
exponent ``$1$'' instead of ``$4$'' as in Corollary \ref{c:HahlomaaWithStratifASs}. The formal reason is visible in 
Proposition \ref{p:FromIotaToStratifAssumpt}, but in light of \cite{MR4489627,FV1,FV2}, we believe that ``$1$'' is in fact optimal.

\subsection*{A few words about the proofs} Our proof of Theorem \ref{t:ImprovedHahlomaa} follows Hahlomaa's proof of 
Corollary \ref{c:HahlomaaWithHorizontal}, which is in turn modeled on  David and Semmes' work in Euclidean spaces \cite[(C3) $\Rightarrow$ (C4) $\Rightarrow$ (C6)]{David1}. We focus here on describing the main new elements in our approach.

Although not explicitly stated in this form, Hahlomaa deduces
the BPLI property for $k$-regular sets  $E\subset \mathbb{H}^n$ from  $\mathrm{GLem}(\beta_{1,\mathcal{V}_k},2)$ by passing via an intermediate  \emph{corona decomposition}. Roughly speaking, this condition means that a system of dyadic cubes on $E$ can be divided into a family of bad cubes, of which there are not too many, and the remaining good cubes can be further partitioned into a forest $\mathcal{F}$ of -- again not too many -- trees, so that the set $E$ has good approximation properties from the perspective of each tree  $\mathcal{S}\in \mathcal{F}$. We define \emph{corona decompositions by horizontal planes (P-C)} (Definition \ref{d:P-C}) and adapt Hahlomaa's proof to verify, for $1\leq k\leq n$, that a $k$-regular set $E$ in $\mathbb{H}^n$ with $\mathrm{GLem}(\beta_{1,\pi,A(2n,k)},2)$ and $\mathrm{WGL}(\beta_{\infty,\mathcal{V}_k})$ satisfies (P-C) (Theorem \ref{t:FromWeakAssToCoronaByNormed}). This corona decomposition is formally stronger than the one implicitly contained in \cite{Hahlomaa},
and we verify the crucial approximation property in Lemma \ref{l:HahlLem4.3}.
The reason why our weaker assumption in terms of the $\beta_{1,\pi,A(2n,k)}$-numbers suffices for (P-C) becomes visible in Lemma \ref{l:HLem6.1} and how this lemma is used to control the trees in a certain subfamily of $\mathcal{F}$. These are trees  for which a large proportion of minimal cubes has children whose best approximating horizontal planes are significantly rotated compared to best approximating plane of the top cube $Q(\mathcal{S})$. The amount of rotation is related to properties of the image of $E$ under the projection $\pi:\mathbb{H}^n \to \mathbb{R}^{2n}$, $\pi(z,t)=z$. This observation allows us to control the  rotation using $E\in \mathrm{GLem}(\beta_{1,\pi,A(2n,k)},2)$. On the other hand, this assumption (or the stronger condition $E\in \mathrm{GLem}(\beta_{1,\pi,\mathcal{V}_k},2)$) \emph{alone} does not imply BPLI, as can be seen  by considering a $1$-regular subset of the vertical $t$-axis, which is mapped to a single point under the projection $\pi$ (Example \ref{ex:ProjBetaNotEnough}). For this reason we need $\mathrm{WGL}(\beta_{\infty,\mathcal{V}_k})$ to establish that $E$ is a priori sufficiently horizontal.

Having deduced from  $E\in \mathrm{GLem}(\beta_{1,\pi,A(2n,k)},2)$ and $E\in \mathrm{WGL}(\beta_{\infty,\mathcal{V}_k})$ that $E$ satisfies the corona decomposition (P-C), we could conclude Theorem \ref{t:ImprovedHahlomaa} by following the approach in \cite{Hahlomaa}. However, to highlight that this last step is independent of  the specific ambient structure of $\mathbb{H}^n$ (unlike the first part of the proof), we apply the result in \cite{Bate} instead. More precisely, we  observe in Lemma \ref{l:P-CtoN-C} that (P-C) implies a \emph{corona decomposition by normed spaces}, as introduced in \cite{Bate}, and Theorem \ref{t:ImprovedHahlomaa}  then follows from \cite[Theorem B]{Bate}. 

\subsection*{Connections with other corona decompositions.}
In Lemmas \ref{l:PCimpliesILGC}  -- \ref{lem_equivcor} we prove:
\begin{theorem}\label{t:CorEquiv}
A $k$-regular set $E\subset \mathbb{H}^n$ ($1\leq k\leq n$) satisfies (P-C) if and only if it admits a \emph{corona decomposition by intrinsic Lipschitz graphs (ILG-C)}.
\end{theorem}
The (ILG-C) property is a variant of the corona decomposition studied by David and Semmes \cite{David1}, where from the perspective of each tree $\mathcal{S}$, the set $E$ is well approximated by an \emph{intrinsic Lipschitz graph} with small intrinsic Lipschitz constant (in the sense of Franchi, Serapioni, and Serra Cassano \cite{MR3587666}). In contrast, (P-C), and the related corona decompositions by normed spaces, are defined via mapping properties rather than approximation by sets. Through the equivalence of (P-C) and (ILG-C), our approach to Theorem \ref{t:ImprovedHahlomaa} via (P-C) 
 yields several corollaries of independent interest. These can be seen as counterparts for classical results in the Euclidean theory of uniform rectifiability \cite{David1}.

\begin{corollary}
[Corollary of Theorems \ref{t:CorEquiv} and \ref{t:FromWeakAssToCoronaByNormed}]     \label{t:FromWeakAssToILG-C} 
    Let $n\in \mathbb{N}$ and $k\in \{1,\ldots,n\}$. If $E\subset \mathbb{H}^n$ is a $k$-regular set such that $E\in \mathrm{GLem}(\beta_{1,\pi,A(2n,k)},2)$ and $E\in \mathrm{WGL}(\beta_{\infty,\mathcal{V}_k})$, then $E$ has a corona decomposition by intrinsic Lipschitz graphs.
\end{corollary}

\begin{corollary}[Corollary of Theorem \ref{t:CorEquiv} and \cite{Bate}]\label{t:FromCoronaToBPLI}
    Let $n\in \mathbb{N}$ and $k\in \{1,\ldots,n\}$. If $E\subset \mathbb{H}^n$ is a $k$-regular set satisfying a corona decomposition by intrinsic Lipschitz graphs, then $E$ has big pieces of Lipschitz images of $\R^k$.
\end{corollary}

The (ILG-C) property can be seen as a set-theoretic version of corona decompositions that were studied for intrinsic Lipschitz functions on $1$-dimensional horizontal subgroups in $\mathbb{H}^n$ \cite{FO,didonato}. 
In \cite[\S 3.1.1]{FO}, the first author and Orponen proved for $n=1$ that such intrinsic Lipschitz functions  admit corona decompositions by intrinsic Lipschitz functions with small Lipschitz constant, and applied this result to prove that certain singular integral operators are $L^2$-bounded on regular curves in $\mathbb{H}^1$. 
We hope that Corollary \ref{t:FromWeakAssToILG-C} could find applications for the study of $k$-dimensional singular integral operators in $\mathbb{H}^n$ for all $1\leq k\leq n$, and more generally, that the present paper will serve as a step towards developing a theory of quantitative (or \emph{uniform}) rectifiability for low-dimensional sets in Heisenberg groups. We believe that this is a natural setting to apply and advance the results in \cite{Bate}, as there is already a  rich theory of \emph{qualitative} $k$-dimensional rectifiability in $\mathbb{H}^n$ for $1\leq k\leq n$, see for instance \cite{MR2789472,MR4405432}.

\medskip

\textbf{Structure of the paper.} Section \ref{s:Prelim} contains standard notation and preliminaries about Heisenberg groups. Section \ref{s:QuantRectif} provides an overview of known and new notions of quantitative rectifiability for low-dimensional sets in Heisenberg groups. In particular, we introduce various versions of geometric lemmas and corona decompositions, and prove implications between them. In Section \ref{s:Juillet}, 
we strengthen Juillet's result by showing that the curve $\Gamma$ constructed in \cite{Juillet} is $1$-regular with $\Gamma\notin \mathrm{GLem}(\beta_{1,\mathcal{V}_1},2)$. In Section
\ref{s:SuffCond}, building on  Section \ref{s:QuantRectif},
we explain how Theorem \ref{t:ImprovedHahlomaa} follows from 
an intermediate result about a corona decomposition by horizontal planes, and we deduce the various corollaries. 
Finally, in Section \ref{s:ProofCorona}, we complete the proof of Theorem \ref{t:ImprovedHahlomaa}  by deducing the existence of the corona decomposition from the stated assumptions. 

\section{Preliminaries}\label{s:Prelim}
\subsection{Notation} 
We write $a\lesssim b$ to denote the existence of an absolute constant $C>0$ with $a\leq Cb$. When writing $a\lesssim_\lambda b$, we allow $C$ to depend on $\lambda$. We write $a\sim b$ if $a\lesssim b$ and $b\lesssim a$.
%Average integrals are denoted by $\fint$.
\subsection{Heisenberg groups}
We recall the definition of the Heisenberg group, and introduce objects related to its horizontal structure  relevant for the study of low-dimensional subsets.
\begin{definition}[Heisenberg group] The \textit{$n$-th Heisenberg group} $\H^n$ is defined as the set $\R^{2n+1}$ equipped with the group product
  \begin{equation}\label{grouplaw} (z,t)\cdot(z',t')=(z+z', t+t'+\omega(z,z'))\quad\text{ for }z,z'\in\R^{2n},t,t'\in\R.\end{equation}
  where \begin{equation}\label{hform}\omega(z,z'):=\frac{1}{2}\sum_{i=1}^n (z_i z_{n+i}'-z_{n+i}z_i')\quad\text{ for }z,z'\in \R^{2n}.\end{equation}
%\textcolor{magenta}{-TBC-} 
\end{definition}
\noindent We equip $\mathbb{H}^{n}$ with left invariant vector fields 
\[X_{i}=\partial_{z_{i}}-\tfrac{1}{2}z_{n+i}\partial_{t},\quad 1\leq i\leq n, \quad X_{i} = \partial_{z_{i}}+\tfrac{1}{2}z_{i-n}\partial_{t}, \quad n+1\leq i\leq 2n, \quad T=\partial_{t}.\]
Here $\partial_{z_{i}}$ and $\partial_{t}$ denote the coordinate vectors in $\mathbb{R}^{2n+1}$, which may be interpreted as operators on differentiable functions. If $[\cdot, \cdot]$ denotes the Lie bracket of vector fields, then $[X_{i}, X_{n+i}]=T$, while all the other commutation relations are trivial. Thus $\mathbb{H}^{n}$ is a Carnot group of step $2$.

\begin{definition}\label{d:horiz}
A vector in $\mathbb{R}^{2n+1}$ is \emph{horizontal} at $p \in \mathbb{R}^{2n+1}$ if it is a linear combination of the vectors $X_{i}(p)$, $1\leq i\leq 2n$.
An absolutely continuous curve $\gamma$ in the Heisenberg group is \emph{horizontal} if, at almost every point $t$, the derivative $\gamma'(t)$ is horizontal at $\gamma(t)$.
\end{definition}
For every absolutely continuous curve $\gamma_I: [a,b]\to \mathbb{R}^{2n}$, one can find $\gamma_{2n+1}:[a,b]\to \mathbb{R}$ so that $\gamma:=(\gamma_I,\gamma_{2n+1}):[a,b]\to \mathbb{H}^n$ is horizontal. If  $\gamma_{2n+1}(a)$ is prescribed, the  component $\gamma_{2n+1}$
is uniquely determined by the horizontality condition, and $\gamma$ is called a \emph{horizontal lift} of $\gamma_I$.
Specifically, for $n=1$, if the curve $\gamma_I$ is closed and $\gamma$ is one of its horizontal lifts, then the total change in height $\gamma_3(b)-\gamma_3(a)$ equals the signed area enclosed by $\gamma_I$; see for instance \cite{CDPT}.
\begin{definition}\label{d;SubspaceDef}
A vector subspace $V\subset\R^{2n}$ is said to be \textit{isotropic} if $\omega(z,z')=0$ for all $z,z'\in V$, where $\omega:\R^{2n}\times\R^{2n}\to\R$ is the form defined in \eqref{hform}. 

For $1\leq k\leq n$, 
we denote by $\mathcal V_k^0$ the family of $k$-dimensional \emph{horizontal subgroups} of $\mathbb{H}^n$. Every $V\in \mathcal{V}_k^0$ is of the form $V=V_I \times \{0\}$ for a  $k$-dimensional isotropic subspace $V_I$  of~$\mathbb{R}^{2n}$.

We call a subset $V$ of $\mathbb{H}^n$ an \emph{(affine) $k$-dimensional horizontal plane} if it can be written as $x\cdot V_0$ for some $x\in \mathbb{H}^n$ and $V_0\in \mathcal V_k^0$. We denote by $\mathcal V_k$ the collection of all affine horizontal $k$-dimensional planes of $\mathbb{H}^n$.

The family of arbitrary affine $k$-dimensional planes in $\mathbb{R}^{2n}$ will be denoted by $A(2n,k)$.
\end{definition}

\begin{definition}[Projection onto horizontal subgroups]
%\textcolor{magenta}{- TBC-}
Let $V=V_I\times\{0\}\in \mathcal V_k^0$ be a $k$-dimensional horizontal subgroup of $\H^n$, where $V_I$ is a $k$-dimensional isotropic subspace of $\R^{2n}$. The \textit{horizontal projection} onto $V$ is defined by 
\[P_V:\H^n\to V,\quad P_V(z,t):=(\pi_{V_I}(z),0),\]
where $\pi_{V_I}:\R^{2n}\to V_I$ stands for the standard Euclidean orthogonal projection onto $V_I$.
\end{definition}

From the structure of the group law \eqref{grouplaw}, it easily follows that horizontal projections are group homomorphisms. For a more thorough discussion of these mappings and their role in geometric measure theory on $\mathbb{H}^n$, we refer the reader to \cite{MR2789472,MR2955184}.

\begin{definition}[Projection onto affine horizontal planes]\label{d:AffHeisProj}
    Let $V= x\cdot V_0 \in \mathcal{V}_k$ be a $k$-dimensional horizontal plane with $x\in \mathbb{H}^n$ and  $V_0\in  \mathcal{V}_k^0$. Then the \emph{horizontal projection} onto $V$ is defined by
    \begin{displaymath}
        P_V:\mathbb{H}^n \to V,\quad P_V(y):= x \cdot P_{V_0}(x^{-1}\cdot y).
    \end{displaymath}
\end{definition}
The affine horizontal projection $P_V$ is well-defined 
 since it does not depend on the specific choice of the point $x\in V$; see \cite[\S 2]{Hahlomaa} or \cite[\S 4.1.2]{FV2} for this and other properties of~$P_V$.

It will be convenient to introduce also other kind of projections, which we may call \textit{Euclidean coordinate projections}: if $p=(z,t)\in\H^n$, we define $\pi:\H^{n}\to\R^{2n}$ and $\pi_t:\H^n\to\R$ as
\[\pi(p)=\pi(z,t):=z,\qquad \pi_t(p)=\pi_t(z,t):=t.\]
We will denote by $|\cdot|_{\mathbb{R}^m}$, or simply by $|\cdot|$, the Euclidean norm on $\mathbb{R}^m$. The Euclidean distance is $d_{\mathrm{Eucl}}(z,z')=|z-z'|$ for every $z,z'\in\R^m$.
In the Heisenberg group $\mathbb{H}^n$, we let $d(x,y):=\|y^{-1}\cdot x\|$ to be the \emph{Korányi distance}  induced by the homogeneous \emph{Korányi norm}
\begin{equation}\label{eq:KoranyiNorm}
\|(z,t)\|:=\sqrt[4]{|z|_{\R^{2n}}^4+16\,t^2}.
\end{equation}
For any $V\in \mathcal V_k$, the horizontal projection $P_V$ is 1-Lipschitz with respect to the Korányi
distance. In addition, the space $(V,d)$ is isometric to $(\R^k, |\cdot|_{\R^k})$. For $s\geq 0,$ we will  denote by $\mathcal H^s$ the $s$-dimensional Hausdorff measure associated to the Korányi distance $d$, obtained via classical Carathéodory construction. Hausdorff measures with respect to $d_{\mathrm{Eucl}}$ will be denoted $\mathcal{H}^s_{\mathrm{Eucl}}$.
The closed ball in $(\mathbb{H}^n,d)$ with center $x\in \mathbb{H}^n$ and radius $r>0$ is denoted by $B(x,r)$, while $B^m(z,r)$ or $B_{\mathrm{Eucl}}(z,r)$ denotes a Euclidean ball in $\mathbb{R}^m$. The diameter of a set $A\subset \mathbb{H}^n$ with respect to the Korányi distance is denoted by $\mathrm{diam}(A)$.
\begin{definition}\label{d:rot}
    A linear map $\varphi:\H^n\to\H^n$ is called a \textit{rotation} if for all $z,z'\in\R^{2n}$, $t,t'\in\R$, 
    \begin{align*}\pi_t(\varphi(z,t))&=t,\\ \omega\left(\pi(\varphi(z,t)),\pi(\varphi(z',t'))\right)&=\omega(z,z'),\\ |\pi(\varphi(z,t))-\pi(\varphi(z',t'))|&=|z-z'|.\end{align*} 
\end{definition}
Any rotation is a group homomorphism and an isometry with respect to $d$ of $\H^n$. Moreover, for any $V,W\in\mathcal V_k$ there exists a rotation $\varphi$ and a point $p\in\H^n$ such that $W=p\cdot\varphi(V)$. Rotations are exactly those maps which can be expressed in coordinates as $\varphi(z,t)=(Az,t)$ with $A\in SU(n)$. For more details, see, e.g.,  \cite[Section 2]{MR2955184}\vspace{0.25 cm}\par
If $V\in \mathcal{V}_k^0$ is a $k$-dimensional horizontal subgroup of $\H^n$, the \textit{complementary orthogonal subgroup} of $V$ inside $\H^n$ is defined as the (Euclidean)  $(2n+1-k)$-dimensional orthogonal complement of  $V$ inside $\R^{2n+1}$, which is a normal subgroup of $\H^n$. If $W$ denotes the complementary orthogonal subgroup of $V$, then we define $P_W:\H^n\to W$ as $P_W(y):=(P_V(y))^{-1}\cdot y$.
The next definition is due to Franchi, Serapioni, and Serra Cassano \cite{MR2313532,MR3587666}.

\begin{definition}[Intrinsic Lipschitz graph]\label{d:kDIMiLG}
 %\textcolor{magenta}{-TBC-}   
 Let $V\in \mathcal{V}_k^0$ and let $W$ be the corresponding complementary orthogonal vertical subgroup. A map $\phi:V\to W$ is called \textit{intrinsic Lipschitz} if there exists $L>0$ such that
 \begin{equation}\label{intrinsicLip}\|P_W\left(\Phi(v')^{-1}\Phi(v)\right)\|\leq L\,\|P_V\left(\Phi(v')^{-1}\Phi(v)\right)\|\quad\text{ for every }v,v'\in V,\end{equation}
 where $\Phi: V\to\H^n$ is the graph map associated to $\phi$, defined by $\Phi(v):=v\cdot\phi(v)$ for all $v\in V$. 
 The \textit{intrinsic Lipschitz constant} of $\phi$ is defined as the best constant $L$ such that \eqref{intrinsicLip} holds.

 A set $\Gamma \subset \mathbb{H}^n$ is a \emph{$k$-dimensional intrinsic Lipschitz graph}, for $1\leq k\leq n$, if there exists $V\in \mathcal{V}_k^0$ and 
 and an intrinsic Lipschitz function $\phi:V \to W$ such that $\Gamma = \Phi(V)$.

\end{definition}

If $V$ is a horizontal subgroup with orthogonal vertical subgroup $W$, and $\phi:V \to W$ is intrinsic Lipschitz as in Definition \ref{d:kDIMiLG}, then the associated graph map $\Phi:V \to \mathbb{H}^n$ is metrically Lipschitz since we have
for all $v,v'\in V$ that
\begin{align}\label{eq:GraphMapMetricLip}
d(\Phi(v),\Phi(v'))&\leq \|P_V\left(\Phi(v')^{-1}\cdot \Phi(v)\right)\|+\|P_W\left(\Phi(v')^{-1}\cdot \Phi(v)\right)\|\notag\ \leq (1+L)\|P_V\left(\Phi(v')^{-1}\cdot \Phi(v)\right)\|\notag\\
&= (1+L) \|(v')^{-1} \cdot v\| = (1+L) d(v,v').
\end{align}

We will also need precise quantitative information in the opposite direction: knowing that a graph map $\Phi(v)=v\cdot \phi(v)$ is metrically Lipschitz with constant close to $1$, we can deduce that the function $\phi:V \to W$ is intrinsic Lipschitz with small constant. 
 \begin{proposition}\label{prop_lipgraph} 
        Let $V$ be a $k$-dimensional horizontal subgroup of $\H^n$, with complementary orthogonal subgroup $W$. Let $\Phi:V\to\H^n$ denote the graph map associated to a function $\phi:V\to W$. Suppose that $\Phi$ is $(1+\eta)$-Lipschitz with respect to the Kor\'{a}nyi metric for some $0<\eta\leq 1$. Then $\phi$ is intrinsic Lipschitz and its intrinsic Lipschitz constant is at most $6\sqrt[4]{\eta}$.
    \end{proposition}

    \begin{proof}
%If $k$ denotes the dimension of $V$, we can suppose without loss of generality that $V=\{(v,0,0)\in\H^n:\,v\in\R^k\}$ and $W:=\{(0,w,t)\in\H^n:\, w\in\R^{2n-k}, t\in\R\}$ since rotations are isometries of the group. 
%Since $V$ is horizontal, we can isometrically and algebraically identify $V\equiv\R^k$. 
%We will also use the following notation: for $p=(v,w,t)\in\H^n\equiv \R^k\times\R^{2n-k}\times\R$, we will write $w=p_I$ and $t=p_t$. In particular, $\phi(v)=(0,\phi_I(v),\phi_t(v))$ and $\phi(v')=(0,\phi_I(v'),\phi_t(v'))$. %Finally, let $P_V$ and $P_W$ denote the projections associated with the splitting $\H^n=V\cdot W$.
Consider $v,v'\in V$. Our assumption tells us that \begin{equation}\label{lipassum}\|\Phi(v')^{-1}\cdot\Phi(v)\|\leq (1+\eta)\|(v')^{-1}\cdot v\|=(1+\eta)|\pi(v)-\pi(v')|_{\R^{2n}},\end{equation}
and our goal is to prove that
\begin{equation}\label{lipgoal}\|P_W(\Phi(v')^{-1}\cdot \Phi(v))\|\leq 6\sqrt[4]{\eta}\,\|P_V(\Phi(v')^{-1}\cdot \Phi(v))\|=6\sqrt[4]{\eta}\,|\pi(v)-\pi(v')|_{\R^{2n}}.\end{equation}
From \eqref{lipassum}, we deduce that 
\[\begin{split}\sqrt{|\pi(v)-\pi(v')|^2_{\R^{2n}}+|\pi(\phi(v))-\pi(\phi(v'))|^2_{\R^{2n}}}&=|\pi(\Phi(v')^{-1}\cdot\Phi(v))|_{\R^{2n}}\leq \|\Phi(v')^{-1}\cdot\Phi(v)\|\\ &\leq (1+\eta)|\pi(v)-\pi(v')|_{\R^{2n}},\end{split}\]
which implies, squaring both sides and collecting terms,
\begin{equation}\label{1ineq}|\pi(\phi(v))-\pi(\phi(v'))|_{\R^{2n}}\leq 2\sqrt{\eta}\,|\pi(v)-\pi(v')|_{\R^{2n}}.\end{equation}
Again, by \eqref{lipassum}, we get that
\[\sqrt[4]{|\pi(v)-\pi(v')|_{\R^{2n}}^4+16|\pi_t(\Phi(v')^{-1}\cdot \Phi(v))|^2}\leq \|\Phi(v')^{-1}\cdot\Phi(v)\|\leq (1+\eta)|\pi(v)-\pi(v')|_{\R^{2n}}.\]
Similarly as before, this implies that 
\begin{equation}\label{2ineq}
|\pi_t(\Phi(v')^{-1}\cdot \Phi(v))|\leq \sqrt\eta|\pi(v)-\pi(v')|_{\R^{2n}}^2.
\end{equation}
We now write $V=V_I\times \{0\}$ and consequently $W=V_I^\perp\times\R$, where $V_I$ is an isotropic subspace of $\R^{2n}$. %Let $\pi_{V_I}$ and $\pi_{V_I^\perp}$ denote the Euclidean orthogonal projections (from $\R^{2n})$ onto $V_I$ and $V_I^\perp$, respectively.
%Notice that, for $z\in\R^{2n}$ and $t\in\R$, it holds
%\[P_V(z,t)=(\pi_{V_I}(z),0)\]
%and
Using the definitions of $P_V$ and $P_W$, we can explicitly write \[\begin{split}P_W(z,t)=(z-\pi_{V_I}(z),t-\omega(\pi_{V_1}(z),z))&=(\pi_{V_I^\perp}(z), t-\omega(\pi_{V_I}(z),\pi_{V_I^\perp} (z)))\\ &=(\pi_{V_I^\perp}(z),t)\cdot(0,-\omega(\pi_{V_I}(z),\pi_{V_I^\perp} (z))),\end{split}\]
where $\omega:\R^{2n}\times\R^{2n}\to\R$ is the form appearing in the group law \eqref{grouplaw}. Using this expression for $P_W$, we get by the triangular inequality
\[\begin{split}\|P_W(\Phi(v')^{-1}\cdot \Phi(v))\|\leq&\left\|\big(\pi(\phi(v))-\pi(\phi(v')),\pi_t(\Phi(v')^{-1}\cdot \Phi(v))\big)\right\|\,+\\&+\left\|\big(0,-\omega(\pi(v)-\pi(v'),\pi(\phi(v))-\pi(\phi(v')))\big)\right\|.\end{split}\]
Now we estimate the two terms on the right-hand side.
\[\left\|\big(\pi(\phi(v))-\pi(\phi(v')),\pi_t(\Phi(v')^{-1}\cdot \Phi(v))\big)\right\|\leq |\pi(\phi(v))-\pi(\phi(v'))|_{\R^{2n}}+2\sqrt{|\pi_t(\Phi(v')^{-1}\cdot \Phi(v))|}.\]
Using \eqref{1ineq} and \eqref{2ineq}, we deduce that 
\[\left\|\big(\pi(\phi(v))-\pi(\phi(v')),\pi_t(\Phi(v')^{-1}\cdot \Phi(v))\big)\right\|\leq 4\sqrt[4]{\eta}|\pi(v)-\pi(v')|_{\R^{2n}}.\]
Concerning the other term, 
\[\begin{split}\left\|\big(0,-\omega(\pi(v)-\pi(v'),\pi(\phi(v))-\pi(\phi(v')))\big)\right\|&=2\sqrt{|\omega(\pi(v)-\pi(v'),\pi(\phi(v))-\pi(\phi(v')))|}\\ &\leq 2\sqrt{\frac{1}{2}|\pi(v)-\pi(v')|_{\R^{2n}}|\pi(\phi(v))-\pi(\phi(v'))|_{\R^{2n}}}\\ &\leq 2\sqrt[4]{\eta}|\pi(v)-\pi(v')|_{\R^{2n}},\end{split}\]
where we used \eqref{1ineq} and the simple inequality $|\omega(z,z')|\leq\frac{1}{2}|z|_{\R^{2n}}|z'|_{\R^{2n}}$ for every $z,z'\in \R^{2n}$. Combining the previous inequalities, we finally get \eqref{lipgoal}.
    \end{proof}

\section{Quantitative notions of rectifiability for low-dimensional sets in $\mathbb{H}^n$}\label{s:QuantRectif}

We discuss quantitative conditions that express in a  multi-scale fashion how well a low-dimensional set $E\subset \mathbb{H}^n$ is covered or approximated by ``nice'' objects such as Lipschitz images, intrinsic Lipschitz graphs, or horizontal planes. This intuition is made precise with the notions of \emph{big pieces} (Section \ref{s:BPLI}), \emph{geometric lemmas} (Section \ref{s:GLem}) and \emph{corona decompositions} (Section \ref{ss:CoronaDef}), originating from the work of David and Semmes in Euclidean spaces \cite{David1,David2}. Here we present new and known incarnations of these classical notions.

\subsection{Big pieces of Lipschitz images}\label{s:BPLI}
\begin{definition}[$s$-regular sets]\label{d:k-reg}
  Let $s>0$. A set $E\subset\H^n$ is said to be \textit{AD-$s$-regular} (or simply $s$-\textit{regular}) if it is closed and there exists a constant $C_E>1$ such that
  \begin{equation}\label{def_Ahlfors}
       C_E^{-1}r^s\leq\mathcal H^s(E\,\cap\,B(x,r) )\leq C_Er^s\quad\text{ for every }x\in E,\, 0<r\leq\diam(E),\,r<\infty. \end{equation}
%  where $\mathcal H^k$ denotes the Hausdorff measure associated to the Korányi distance.
We call the \textit{regularity constant} of $E$ the best possible constant $C_E$ satisfying \eqref{def_Ahlfors}.
\end{definition}
\begin{definition}\label{def_BPLI}
    Let $k\in\{1,\dots,n\}$. We say that a $k$-regular set $E\subset\H^n$ has \emph{big pieces of Lipschitz images of subsets of $\R^k$} (BPLI) if there exist constants $c, L>0$ such that for every $x\in E$, $0<r<\diam(E)$ there exist $A \subset B^k(0,r)$ and an $L$-Lipschitz function $f:A\to\H^n$ with the property that 
    \begin{equation}\label{eq:BPLI}\mathcal H^k\left(E\,\cap B(x,r)\cap f(A)\right)\geq cr^k.
    \end{equation}
\end{definition}
Slightly different definitions are used in the literature, but in our setting they turn out to be all equivalent.
First, by the Lipschitz extension result \cite{WY10}, any Lipschitz function $f: A\subset\R^k\to\H^n$ can be extended on the full $\R^k$ with controlled Lipschitz constant. Notice that this is possible exactly when $k\leq n$ (see \cite{BF09}). Hence it is equivalent to consider functions defined on $B^k(0,r)$ or on $\R^k$, and replace in \eqref{eq:BPLI} the set $f(A)$ by $f(B^k(0,r))$.
Remarkably, this is further equivalent to the condition where $f:A\subset B^k(0,r)\to \H^n$ is required to be \textit{bilipschitz} (by \cite[Corollary 1.9]{schul09}), as in \cite[Theorem 1.1]{Hahlomaa}. In particular, Theorem \ref{t:ImprovedHahlomaa} and its corollaries in Section \ref{s:Intro} all hold with the conclusion  ``big pieces of Lipschitz images'' replaced by ``big pieces of bi-Lipschitz images'' in the sense of \cite[p.1]{Hahlomaa}; see also \cite[\S 1.2]{Bate}. We decided to focus on Definition \ref{def_BPLI} since in \cite{Bate}, BPLI is taken as a definition of uniform rectifiability in metric spaces.

%A different definition arises if one replaces Lipschitz images with intrinsic Lipschitz graphs. In this case we say that $E\subset\H^n$ has \textit{big pieces of intrinsic Lipschitz graphs} (BPiLG) if it is $k$-regular and there exist constants $c,L>0$ such that for every $x\in E$, $0<r<\diam(E)$ there exists an intrinsic Lipschitz graph $\Gamma$ defined on a $k$-dimensional horizontal subgroup $V$ with intrinsic Lipschitz constant bounded by $L$ and
%\[\mathcal H^k(E\cap B(x,r)\cap \Gamma)\geq cr^k.\] Since the graph map associated to an intrinsic Lipschitz function is metrically Lipschitz for $k\leq n$, it follows that (BPiLG)$\,\Rightarrow\,$(BPLI). The converse is not true already in the Euclidean case, as an unpublished example of T. Hrycak shows; hence we expect a similar behavior also in $\H^n$. (?)
%}}

Because a Lipschitz function $f:\R^k\to(\H^n,d)$ is locally Euclidean Lipschitz, while the converse is not true, proving that a set $E\subset \H^n$ satisfies BPLI is not just a Euclidean result. Consequently, sufficient conditions for BPLI in $\mathbb{H}^n$, as introduced in the following, have to take into account the horizontal structure  and cannot be mere copies of 
 the conditions in \cite{David1}.

\subsection{Geometric lemmas}\label{s:GLem}

For any $E\subset \H^n$ and $x\in\H^n$ we denote by
\[d(x, E):=\inf\,\{d(x,y):y\in E\}\]
the distance between the point $x$ and the set $E$.
\begin{definition}\label{def_betaknumbers}
   Let $n\in \mathbb{N}$, $E\subset \H^n, k\in\{1,\ldots,n\}, p\geq 1, x\in\H^n, r>0$. We define
   \begin{align*}\beta_{p, \mathcal V_k}(x,r)=\beta^E_{p, \mathcal V_k}(x,r)&:=\inf_{V\in\mathcal V_k}\left(r^{-k}\int_{E\cap B(x,r)}\left(\frac{d(y,V)}{r}\right)^p\,d\mathcal H^k(y)\right)^\frac{1}{p},\quad p<\infty\\
  \beta_{\infty,\mathcal V_k}(x,r)= \beta^E_{\infty,\mathcal V_k}(x,r)&:=\inf_{V\in\mathcal V_k} \sup_{y\in E\cap B(x,r)} \frac{d(y,V)}{r}.
   \end{align*}
\end{definition}

\begin{definition}\label{def_geometriclemma}
    Let $E\subset \H^n$ be AD-$k$-regular, $1\leq p\leq\infty$, $1\leq q<+\infty$.  We say that \textit{$E$ satisfies the $q$-geometric lemma with respect to $\beta_{p,\mathcal V_k}$}, and we write  $E\in{\rm GLem}(\beta_{p,\mathcal V_k},q)$ if there exists $C>0$ such that for every $R>0$ and every $x\in E$ it holds
    \begin{equation}\label{geometriclemma}\int_0^R\int_{E\cap B(x,R)}\beta_{p, \mathcal V_k}(y,r)^q\,d\mathcal H^k(y)\frac{dr}{r}\leq CR^k.\end{equation}
    \end{definition}

    Condition \eqref{geometriclemma} is one way of expressing that the set $E$ is well approximated by horizontal planes at most places and scales. This is equivalent to saying that $\beta_{p,\mathcal V_k}(y,r)^q d\mathcal H^k(y)\,\frac{dr}{r}$ is a Carleson measure on $E\times \R^+$.
    Some of the results we will use have been stated in terms of different versions of geometric lemmas, so we  review the relevant definitions here. Instead of using the double integral in \eqref{geometriclemma}, geometric lemmas can be formulated with the help of 
    %\emph{multiresolution families} or 
    \emph{systems of dyadic cubes}.

    \subsubsection{Carleson conditions for dyadic cubes}

The following is a special case of a construction due to Christ \cite{christ}, with minor modifications that we will indicate below.
    
\begin{definitionT}[Dyadic cubes] \label{dt:dyad}
    Assume that $E\subset \mathbb{H}^n$ is a $k$-regular set with regularity constant $C_E$. 
    %Let $\mathbb{J}=\mathbb{Z}$ if $E$ is unbounded, and  $\mathbb{J}=\{j\in \mathbb{Z}:\,j\geq J_0\}$ where $J_0 \in \mathbb{Z}$ is such that $
    %2^{-{J_0}}
  %  \rho^{J_0}\leq \mathrm{diam}(E)<
    %2^{-{J_0}+1}
   % \rho^{J_0+1}$ otherwise.
    Then there exist constants $\varrho \in (0,1)$ and $D\in (1,\infty)$, depending only on $C_E$ and $k$, and a collection $\mathcal{D}=\cup_{j\in \mathbb{J}}\mathcal{D}_j$ where, for every $j\in\mathbb J$, $\mathcal D_j$ is a family of pairwise disjoint Borel sets and 
    \begin{enumerate}
        \item For each $j\in \mathbb{J}$, $E=\bigcup_{Q\in \mathcal{D}_j}Q$.
        \item If $Q_1,Q_2 \in \mathcal{D}$ and $Q_1\cap Q_2\neq \emptyset$, then $Q_1 \subseteq Q_2$ or $Q_2\subset Q_1$.
        \item If $j\in \mathbb{J}$ and $Q\in \mathcal{D}_j$, then $\mathrm{diam}(Q)\leq D \varrho^{j}$.
        \item For each $j\in \mathbb{J}$ and $Q\in \mathcal{D}_j$, there exists $x_Q\in E$ such that $B(x_Q,D^{-1}\rho^{j})\cap E \subset Q$.
        \item $\mathcal{H}^k(\{x\in Q:\, d(x,E\setminus Q)\leq \eta \varrho^j\})\leq D \eta^{\frac{1}{D}}\mathcal{H}^k(Q)$ for all $j\in \mathbb{J}$, $Q\in \mathcal{D}_j$, $\eta>0$.
    \end{enumerate}
 Here $\mathbb{J}=\mathbb{Z}$ if $E$ is unbounded, and  $\mathbb{J}=\{j\in \mathbb{Z}:\,j\geq J_0\}$ where $J_0 \in \mathbb{Z}$ is such that $
    %2^{-{J_0}}
    \rho^{J_0+1}\leq \mathrm{diam}(E)<
    %2^{-{J_0}+1}
    \rho^{J_0}$ otherwise.
     
 We also define
    \begin{displaymath}
        \mathcal{D}_{Q_0}:=\{Q\in \mathcal{D}:\,Q\subset Q_0\},\quad Q_0\in \mathcal{D},
    \end{displaymath}
    and for a given constant $\lambda>1$, we set
    \begin{displaymath}
        \lambda Q:= \{x\in E:\, d(x,Q)\leq (\lambda-1)\mathrm{diam}(Q)\}.
    \end{displaymath}
\end{definitionT}
Choosing $D$ large enough, we may and will assume that
\begin{equation}\label{eq:H_(14)}
    D^{-1}\rho^j \leq \mathrm{diam}(Q)\leq D \rho^j \quad \text{and}\quad D^{-1}\rho^{jk}\leq \mathcal{H}^k(Q)\leq D\rho^{jk},\quad Q\in \mathcal{D}_j,\,j\in \mathbb{J},
    \end{equation}
    cf.\ \cite[(14)]{Hahlomaa} or similar computations leading to \cite[(2.8)]{FV1}.
If $Q\in\mathcal D_j$ for some $j\in\mathbb J$, we denote by $\mathcal C(Q)$ the set of \textit{children} of $Q$, namely $\mathcal C(Q):=\{R\in \mathcal D_{j+1}: R\subset Q\}$. 
 Also, for $\mathcal S\subset \mathcal D$, we let $\min(\mathcal S)$ to be the family of minimal (with respect to the inclusion) cubes in $\mathcal S$.
\begin{remark}\label{r:Dyad1}
The cubes in Definition and Theorem \ref{dt:dyad} are essentially as in \cite[(8)--(14)]{Hahlomaa} with $\varrho = 1/\alpha$ and  $\mathcal{D}_j$ playing the role of cubes in generation ``$-j$''.   
The main difference is that we assume that the cubes of a fixed generation cover $E$ entirely, not only up to a $\mathcal{H}^k$-null set. The existence of such cubes can be proven by modifying Christ's original construction as explained in \cite[Section 4]{MR2965363}, see also \cite[Lemma 2.6.1]{Bate}. The resulting cubes are not necessarily open, but  Borel sets, which is sufficient to run the argument in  \cite{Hahlomaa}.
\end{remark}

\begin{remark}\label{r:Dyad2}
Given dyadic cubes as in Definition and Theorem \ref{dt:dyad} for some $\varrho\in (0,1)$, one can always construct a new family that satisfies the same properties with ``$\varrho$'' replaced by ``$1/2$'' and some of the constants possibly depending on $\varrho$, see the proof of \cite[Proposition 2.12]{MR3589162}. Since the original $\varrho$ can be chosen depending only on $k$ and $C_E$, we could without loss of generality assume in Definition and Theorem \ref{dt:dyad}  that $\varrho=1/2$. In particular, all $k$-regular sets in $\mathbb{H}^n$ admit dyadic systems with the properties stated in \cite[Theorem and Definition 2.5]{FV1}.
We did not to specify the value of $\varrho$ in Definition and Theorem \ref{dt:dyad} in order to have a more flexible definition that allows us to apply results stated in the literature for various values of $\varrho$.
\end{remark}

\begin{remark}\label{r:Dyad3}
The dyadic systems used in \cite{Bate} have all the properties required in our Definition and Theorem \ref{dt:dyad}. In fact,
\cite[Lemma 2.6.1]{Bate} states additional properties concerning the ``centers'' of the cubes and the values of constants. These will not be needed for most of our paper, only when applying  \cite[Theorem 9.0.1]{Bate}, we will restrict attention to dyadic systems as in \cite[Lemma 2.6.1]{Bate} (as we may by Remark \ref{r:IndepDyadSyst}).
\end{remark}

We will compare geometric lemmas for different coefficient functions that we now introduce. We recall from Definition \ref{d;SubspaceDef} that 
$\mathcal{V}_k$ denotes the family of all horizontal $k$-planes in $\mathbb{H}^n$, whereas  $A(2n,k)$ stands for Grassmanian of all affine $k$-planes in $\mathbb{R}^{2n}$.

\begin{definition}[Coefficient functions]\label{d:Coefficient functions}
Let $n\in \mathbb{N}$, $k\in \{1,\ldots,n\}$, $E\subset \mathbb{H}^n$ be a $k$-regular set with a dyadic system $\mathcal{D}$, $Q\in \mathcal{D}$, $p\geq 1$, and $\lambda >1$. We define 
\begin{itemize}
\item the \emph{horizontal $\beta$-numbers}
\begin{align*}\beta_{p,\mathcal{V}_k}(\lambda Q)&:=\inf_{V\in\mathcal V_k}\left(\fint_{\lambda Q}\left(\frac{d(y,V)}{\diam (\lambda Q)}\right)^p\,d\mathcal{H}^k(y)\right)^{\frac{1}{p}},\quad p<\infty,\\
\beta_{\infty,\mathcal{V}_k}(\lambda Q)&:=\inf_{V\in\mathcal V_k}
\,\sup_{y\in \lambda Q} \frac{d(y,V)}{\diam (\lambda Q)}
\end{align*}
    \item  the \emph{stratified $\beta$-numbers}
    \begin{align*}\widehat{\beta}_{p,\mathcal{V}_k}(\lambda Q)&:=\inf_{V\in\mathcal V_k}\left(\left[\fint_{\lambda Q}\left(\tfrac{d_{\mathrm{Eucl}}(\pi(y),\pi(V))}{\diam (\lambda Q)}\right)^p\,d\mathcal{H}^k(y)\right]^{\frac{2}{p}}+\left[\fint_{\lambda Q}\left(\tfrac{d(y,V)}{\diam (\lambda Q)}\right)^p\,d\mathcal{H}^k(y)\right]^{\frac{4}{p}}\right)^{\frac{1}{4}}, p<\infty,\\
\widehat{\beta}_{\infty,\mathcal{V}_k}(\lambda Q)&:=\inf_{V\in\mathcal V_k}\,\left(\sup_{y\in \lambda Q}\left[\frac{d_{\mathrm{Eucl}}(\pi(y),\pi(V))}{\diam (\lambda Q)}\right]^2+\sup_{y\in \lambda Q}\left[\frac{d(y,V)}{\diam (\lambda Q)}\right]^4\right)^{\frac{1}{4}}
\end{align*}
    \item  the \emph{horizontal projection $\beta$-numbers}
    \begin{align*}\beta_{p,\pi,\mathcal{V}_k}(\lambda Q)&:=\inf_{V\in\mathcal V_k}\left(\fint_{\lambda Q}\left(\frac{d_{\mathrm{Eucl}}(\pi(y),\pi(V))}{\diam (\lambda Q)}\right)^p\,d\mathcal{H}^k(y)\right)^{\frac{1}{p}},\quad p<\infty,\\
\beta_{\infty,\pi,\mathcal{V}_k}(\lambda Q)&:=\inf_{V\in\mathcal V_k}\,  \sup_{y\in \lambda Q} \frac{d_{\mathrm{Eucl}}(\pi(y),\pi(V))}{\diam (\lambda Q)}
\end{align*}
    \item  the \emph{projection $\beta$-numbers}
    \begin{align*}\beta_{p,\pi,A(2n,k)}(\lambda Q)&:=\inf_{W\in A(2n,k)}\left(\fint_{\lambda Q}\left(\frac{d_{\mathrm{Eucl}}(\pi(y),W)}{\diam (\lambda Q)}\right)^p\,d\mathcal{H}^k(y)\right)^{\frac{1}{p}},\quad p<\infty,\\
\beta_{\infty,\pi,A(2n,k)}(\lambda Q)&:=\inf_{W\in A(2n,k)}\,  \sup_{y\in \lambda Q} \frac{d_{\mathrm{Eucl}}(\pi(y),W)}{\diam (\lambda Q)}
\end{align*}
\item the \emph{projection $\iota$-numbers}
\begin{align*}\iota_{p,\mathcal{V}_k}(\lambda Q) &:=\inf_{V\in\mathcal V_k }\left(\fint_{\lambda Q}\fint_{\lambda Q}\left(\frac{\left|d(x,y)-d(P_V(x),P_V(y))\right|}{\mathrm{diam}(\lambda Q)}\right)^p\,d\mathcal{H}^k(x)\,d\mathcal{H}^k(y)\right)^{\frac{1}{p}},\quad p<\infty,\\\iota_{\infty,\mathcal{V}_k}(\lambda Q) &:=\inf_{V\in\mathcal V_k }\,  \sup_{y\in \lambda Q} \frac{\left|d(x,y)-d(P_V(x),P_V(y))\right|}{\mathrm{diam}(\lambda Q)} \end{align*}
\end{itemize}
\end{definition}
\begin{remark}\label{r:BallGeom}
The functions $h\in \{\beta, \widehat\beta,\beta_{\pi},\iota\}$, $h:\mathcal{D} \to [0,\infty)$, have obvious counterparts $h:E \times [0,\infty) \to [0,\infty)$ defined using balls as in Definition \ref{def_betaknumbers} instead of cubes. More precisely, one replaces ``$\fint_{\lambda Q}$'' with ``$r^{-k}\int_{E\cap B(x,r)}$'', and ``$\mathrm{diam}(\lambda Q)$'' with ``$r$''.
Abusing notation, we will denote the values of these functions by $h(y,r)$. 
%This terminology will be applied in Lemma~\ref{l:DiffGeomLem}.
\end{remark}

Horizontal $\beta_{\infty}$-numbers in $\mathbb{H}^n$ were introduced for $k=n=1$ in connection with the traveling salesman theorem \cite{MR2371434}; the higher-dimensional $L^1$-based variants were used by Hahlomaa \cite{Hahlomaa}. The stratified $\beta$-numbers we define here are inspired by Li's definition \cite{Li} for $p=\infty$ in Carnot groups. The projection $\iota$-numbers were introduced by the first author and Violo in \cite{FV2}. The projection $\beta$-numbers have to the best of our knowledge not been used before, and we comment on them in the next remark.

\begin{remark}\label{r:AffVsFullProjBeta}
The horizontal projection $\beta$-numbers $\beta_{p,\pi,\mathcal V_k}$ arise naturally in relation with the stratified $\beta$-numbers $\widehat\beta_{p,\pi,\mathcal V_k}$ and the horizontal $\beta$-numbers $\beta_{p,\mathcal V_k}$. %By definition, it is clear that $\beta_{p,\pi,\mathcal V_k}^2+\beta_{p,\mathcal V_k}^4\leq \widehat\beta_{p,\mathcal V_k}^4$
Since, for any $V\in \mathcal V_k$, the projection $\pi(V)$ is an \textit{affine isotropic subspace} of $\R^{2n}$ (namely $\pi(V)=z+W$, with $z\in \R^{2n}$ and $W$ an isotropic subspace of $\R^{2n}$) and, conversely, any  $k$-dimensional affine isotropic subspace of $\R^{2n}$ can be expressed as $\pi(V)$ for some $V\in\mathcal V_k$, it follows that in the definition of $\beta_{p,\pi,\mathcal V_k}$ one can equivalently consider isotropic subspaces of dimension $k$ in $\R^{2n}$ instead of projections of affine horizontal $k$-planes. Since the projection $\beta$-numbers are instead computed by minimizing among \textit{all} $W\in A(2n,k)$, it follows that, for any $p\geq 1, \lambda>1$,
\[\beta_{p,\pi,A(2n,k)}(\lambda Q)\leq \beta_{p,\pi,\mathcal V_k}(\lambda Q).\] 
\end{remark}

The properties stated in the next two lemmas hold more generally for all coefficient functions in Definition \ref{d:Coefficient functions}, but we focus here on the statements that will be applied later in the paper.

\begin{lemma}\label{lem_translations}
 For any $x,y, z\in\H^n$, $V\in \mathcal V_k$ it holds \[|\pi(x)-\pi(y)|=|\pi(z\cdot x)-\pi(z\cdot y)|,\quad d(P_V(x),P_V(y))=d(P_{z\cdot V}(z\cdot x),P_{z\cdot V}(z\cdot y)).\] 
 
 Moreover, if $z\in \mathbb{H}^n$ and $E$ is a $k$-regular set with a dyadic system $\mathcal{D}$, then $\{z\cdot Q\}$ is a dyadic system for the $k$-regular set $z\cdot E$ and, for all $p\in [1,\infty]$, $\lambda\geq 1, Q\in \mathcal{D}$, we have the following identities between the coefficient functions associated to $E$ and $z\cdot E$, respectively:
  \[\widehat\beta_{p,\mathcal{V}_k}(\lambda Q)=\widehat \beta_{p,\mathcal{V}_k}(\lambda (z\cdot Q)), \; \beta_{p,\pi,A(2n,k)}(\lambda Q)=\beta_{p,\pi,A(2n,k)}(\lambda (z\cdot Q)), \;\iota_{p,\mathcal V_k}(\lambda Q)=\iota_{p,\mathcal V_k}(\lambda(z\cdot Q)).\]
\end{lemma}

\begin{proof}
The claim about the dyadic systems holds due to the left-invariance of $d$. 
     The property $|\pi(x)-\pi(y)|=|\pi(z\cdot x)-\pi(z\cdot y)|$ follows immediately from the structure of the group law and the invariance under translations of the Euclidean distance.
     Together with the left-invariance of $d$ and the resulting left-invariance of $\mathcal{H}^k$, it implies the claimed left-invariance property of the $\widehat \beta$-numbers and of the $\beta_{\pi}$-numbers, observing also that $\mathcal{V}_k =\{z\cdot V:\, V\in \mathcal{V}_k\}$ for all $z\in \mathbb{H}^n$, and $A(2n,k)=\{u+ W:\, W\in A(2n,k)\}$ for all $u\in \mathbb{R}^{2n}$.
     
     %left-translations act transitively of the family $\mathcal{V}_k$ of horizontal $k$-planes.
     
 For the second property, recall from Definition \ref{d:AffHeisProj} that if $V=q\cdot V_0$ with $V_0$ isotropic, then $P_V(x)=q\cdot P_{V_0}(q^{-1}\cdot x)$. This implies that $P_{z\cdot V}(z\cdot x)=z\, q\cdot P_{V_0}(q^{-1}z^{-1}z\,x)=z\cdot P_V(x).$ In particular we get that
  \begin{equation}\label{eq_transinv}d(P_V(x),P_V(y))=d(z\cdot P_V(x),z\cdot P_V(y))=d(P_{z\cdot V}(z\cdot x),P_{z\cdot V}(z\cdot y)).\end{equation}
  This identity, together with the left-invariance properties  of $d$ and  $\mathcal{H}^k$, implies the claimed left-invariance property for the $\iota$-numbers. 
\end{proof}

Similar arguments as in the proof of Lemma \ref{lem_translations} yield the following result.

\begin{lemma}\label{l:RotInv}
    If $\varphi:\mathbb{H}^n\to \mathbb{H}^n$ is a rotation and $E$ is a $k$-regular set with a dyadic system $\mathcal{D}$, then $\{\varphi(Q)\}$ is a dyadic system for the $k$-regular set $\varphi( E)$ and, for all $p\in [1,\infty]$, $\lambda\geq 1, Q\in \mathcal{D}$, we have
  \[ \beta_{p,\pi,A(2n,k)}(\lambda Q)=\beta_{p,\pi,A(2n,k)}(\lambda \varphi( Q)).\]
\end{lemma}

\begin{lemma}\label{l:NonAffineProj}
Let $V= z \cdot V_0 \in \mathcal{V}_k$ with $z\in \mathbb{H}^n$ and $V_0 \in \mathcal{V}_k^0$. Then
\begin{displaymath}
    d(P_V(x),P_V(y))=d(P_{V_0}(x),P_{V_0}(y)),\quad x,y\in\mathbb{H}^n.
\end{displaymath}
In particular, in the definition of $\iota_{p,\mathcal{V}_k}$ we can replace ``$\mathcal{V}_k$'' by ``$\mathcal{V}_k^0$'' without changing the values of the coefficients. 
\end{lemma}

\begin{proof}
   By \eqref{eq_transinv}, the fact that $V_0=V_{0,I}\times \{0\}$ is a subgroup, the structure of the group law and the linearity of Euclidean orthogonal projections, we obtain
  \[\begin{split}d(P_V(x),P_V(y))&=d(P_{V_0}(z^{-1}x),P_{V_0}(z^{-1}y))=|\pi_{V_{0,I}}(\pi(z^{-1}x))-\pi_{V_{0,I}}(\pi(z^{-1}y))|\\ &=|\pi_{V_{0,I}}(\pi(x))-\pi_{V_{0,I}}(\pi(y))|=d(P_{V_0}(x), P_{V_0}(y)).\qedhere\end{split}\]
\end{proof}

\begin{definition}[Geometric lemma]\label{d:GL}
Let $n\in \mathbb{N}$, $k\in \{1,\ldots,n\}$, and let $h$ be one of the coefficient functions in Definition \ref{d:Coefficient functions}.
Given $1\leq q<+\infty$, a $k$-regular set $E\subset \mathbb{H}^n$ with a dyadic system $\mathcal{D}$ %$=(\mathcal{D}_j)_{j\in \mathbb{Z}}$,
and $\lambda >1$, we say that $E$ satisfies the
\emph{$q$-geometric lemma with respect to $h$}, and  write
$E\in\mathrm{GLem}(h,q)$, if there exists a constant $C$ such that
we have
\begin{equation}\label{eq:SGL}
\sum_{Q\in \mathcal{D}_{Q_0}} h(\lambda Q)^q\, \mathcal{H}^k(Q) \leq C \mathcal{H}^k(Q_0), \quad
Q_0\in \mathcal{D}.
\end{equation}
%In this case, we also write $E\in\mathrm{GLem}(h,q,C)$.
\end{definition}

A few remarks are in order:
\begin{remark}\label{r:IndepDyadSyst}
 Dyadic systems are not unique, but if $h$ is one of the coefficient functions in Definition \ref{d:GL}, and $E$ satisfies \eqref{eq:SGL} for one dyadic system $\mathcal{D}$, then it fulfills the same condition for \emph{all} possible such systems (with possibly different, but quantitatively controlled constants). This was proven in \cite[Remark 2.16]{FV1} for dyadic systems that satisfy the properties of Definition and Theorem \ref{dt:dyad} with $\varrho=1/2$. 
 The result for arbitrary $\varrho$, with $C$ depending 
possibly on $\varrho$, follows by the procedure in the proof of \cite[Proposition 2.12]{MR3589162}. Alternatively, one can pass via the integral characterization in Lemma \ref{l:DiffGeomLem} to verify that the  validity of the geometric lemma does not depend on the chosen dyadic system on $E$. 
 %Given a dyadic system with a different value of $\varrho$, one can assign one with $\varrho=1/2$ following the procedure in the proof of \cite[Proposition 2.12]{MR3589162}. By construction, one of the systems satisfies \eqref{eq:SGL} if and only if the other one does, with $C$ depending 
% possibly on $\varrho$.
 \end{remark}

 \begin{remark}\label{r_independenceoflambda}
  For $h$ as in Definition \ref{d:GL}, the validity of \eqref{eq:SGL} (up to changing the constant ``$C$'' in a quantitative way) is independent of the choice of $\lambda>1$, see \cite[Remark 2.30]{FV1} and \cite[Lemma 2.23]{FV1}, which holds as stated for $K_0 >1$.
\end{remark}

Geometric lemmas for standard Jones-type $\beta$-numbers in Euclidean spaces can be equivalently stated with a Carleson measure condition in the spirit of Definition \ref{def_geometriclemma} or with a discrete condition in the spirit of Definition \ref{d:GL}, see, for instance, the comment around (1.51) in  \cite[p.20]{David2}. An analogous reformulation is possible in our setting.

\begin{lemma}\label{l:DiffGeomLem}
Let $n\in \mathbb{N}$, $k\in \{1,\ldots,n\}$, $q\in [1,+\infty)$, and let $h$ be one of the coefficient functions in Definition \ref{d:Coefficient functions}.
Then  a $k$-regular set $E\subset \mathbb{H}^n$ satisfies $E\in\mathrm{GLem}(h,q)$ in the sense of Definition \ref{d:GL} if and only if there exists $C>0$ such that for any $R>0$ and $x\in E$ it holds
\begin{equation}\label{eq:IntGLem}
    \int_0^R \int_{E\cap B(x,R)}h(y,r)^q\,d\mathcal{H}^k(y)\,\frac{dr}{r}\leq C R^k.
\end{equation}
\end{lemma}
Since the argument is standard, we only sketch the proof.
\begin{proof} Assume first that $E$ satisfies \eqref{eq:IntGLem}. Let $\mathcal{D}$ be a dyadic system on $E$, $\lambda>1$, write $\mu=\mathcal{H}^k_\llcorner E$, and fix $Q_0 \in \mathcal{D}$. Let $j_0$ be such that $Q_0 \in \mathcal{D}_{j_0}$. The coefficient function $h$ has the following crucial property: there exists a constant $M>1$, depending only on $C_E$, $k$, and $\lambda$, such that 
for all $Q\in \mathcal{D}_j$, $Q\subset Q_0$, we have 
\begin{equation}\label{eq:h-ball-cube}
    h(\lambda Q) \leq M h(y,M r),\quad y\in B(x_Q,\rho^{j})\cap E,\, r\in [\rho^{j},\rho^{j-1}].
\end{equation}
This, and the $k$-regularity of $E$, allows us to bound the sum over dyadic cubes as follows:
\begin{align}\label{eq:BoundDiscr}
    \sum_{Q\in\mathcal{D}_{Q_0}}h(\lambda Q)^q \mu(Q) &\lesssim
    \sum_{j\geq j_0} \int_{\rho^{j}}^{\rho^{j-1}}\sum_{Q\in \mathcal{D}_j\cap \mathcal{D}_{Q_0}}\int_{B(x_Q,\rho^j)}h(y,Mr)^q\,d\mu(y)\,\frac{dr}{r}\notag\\
   & =  \sum_{j\geq j_0} \int_{\rho^{j}}^{\rho^{j-1}} \int_{B(x_{Q_0},C_1 \rho^{j_0})}h(y,Mr)^q \sum_{Q\in \mathcal{D}_j\cap \mathcal{D}_{Q_0}}\chi_{B(x_Q,\rho^{j})}(y)\,d\mu(y)\,\frac{dr}{r},
\end{align}
where we have used in the last step that there exists a constant $C_1>1$, depending only on $C_E$ and $k$, such that $B(x_Q,\rho^j) \subset B(x_{Q_0},C_1 \rho^{j_0})$ for all $Q\in \mathcal{D}_j\cap  \mathcal{D}_{Q_0}$.  By the properties of dyadic systems (Definition and Theorem \ref{dt:dyad}), we have $d(x_Q,x_{Q'})\geq D^{-1} \rho^j$ for $Q,Q'\in \mathcal{D}_j$. The $k$-regularity of $E$ and a standard volume counting argument then imply that $\mathrm{card}(Q\in \mathcal{D}_j\cap \mathcal{D}_{Q_0}:\, y\in B(x_Q,\rho^j))\lesssim 1$ for all $y\in E$. This finally allows to bound \eqref{eq:BoundDiscr} by 
\begin{displaymath}
    \int_0^{\rho^{j_0-1}}\int_{B(x_{Q_0},C_1 \rho^{j_0})}h(y,Mr)^q \, d\mu(y)\,\frac{dr}{r},
\end{displaymath}
from where $E\in\mathrm{GLem}(h,q)$ follows (after a change of variables in $r$) since $Q_0$ was chosen arbitrarily. Here, the constant in $\mathrm{GLem}(h,q)$ depends on the constant $C$ in \eqref{eq:IntGLem}, as well as possibly on $C_E$, $k$, and $\lambda$. 

\medskip

In the opposite direction, assume that $E\in\mathrm{GLem}(h,q)$ in the sense of Definition \ref{d:GL}. Let $\mathcal{D}$  be a dyadic system on $E$, pick $x\in E$ and $R>0$. Let $j_0\in \mathbb{Z}$ be such that $\rho^{j_0}\leq R<\rho^{j_0-1}$. We begin by writing
\begin{equation}\label{eq:doubleIntBdd}
    \int_0^R \int_{B(x,R)} h(y,r)^q\, d\mu(y)\, \frac{dr}{r}\sim \sum_{j\geq j_0} \int_{\rho^j}^{\rho^{j-1}}\int_{B(x,R)}h(y,\rho^{j-1})^q\,d\mu(y)\,\frac{dr}{r},
\end{equation}
where we used that $h(y,r)\lesssim h(y,\rho^{j-1})$ for $r\in [\rho^{j},\rho^{j-1}]$ thanks to $k$-regularity of $E$. Moreover, there exists $Q_0\in \mathcal{D}$ such that $B(x,R) \subset K Q_0$ with $R \sim \diam(Q_0)$ and $K$ depending only on $k$ and $C_E$, see e.g., \cite[(2.9)]{FV1}. We cover $KQ_0$ with a number $m \lesssim 1$ of cubes $Q_{0,1},\ldots,Q_{0,m}$ from generation $j_0$ as in \cite[(2.12)]{FV1}. An inequality similar to \eqref{eq:h-ball-cube} with the roles of balls and cubes reversed, and the assumption $E\in\mathrm{GLem}(h,q)$ (with Remark \ref{r_independenceoflambda}) finally allow us to bound the expression in \eqref{eq:doubleIntBdd} by $\mu(Q_{0,1})+\cdots + \mu(Q_{0,m}) $, which concludes the proof.
\end{proof}

\begin{definition}[Weak geometric lemma]\label{d:WGL} Let $n\in \mathbb{N}$ and $k\in \{1,\ldots,n\}$.
Given  a $k$-regular set $E\subset \mathbb{H}^n$ with a dyadic system $\mathcal{D}$ %=(\mathcal{D}_j)_{j\in \mathbb{Z}}$,
 and   $h$ one of the coefficient functions in Definition \ref{d:Coefficient functions}, we say that $E$ satisfies the
\emph{weak geometric lemma with respect to $h$}, and  write
$E\in\mathrm{WGL}(h)$, if for each $\varepsilon>0$ and $\lambda>1$, there is a constant $C=C(\varepsilon,\lambda)$ such that
\begin{equation}\label{eq:WGL}
\sum_{Q\in \mathcal{D}_{Q_0}, h(\lambda Q)>\varepsilon} \mathcal{H}^k(Q) \leq C(\varepsilon,\lambda) \mathcal{H}^k(Q_0), \quad
Q_0\in \mathcal{D}.
\end{equation}
%In this case, we also write $E\in\mathrm{WGL}(h,C)$.    
\end{definition}

\begin{remark}\label{r:WGLNothingelse}
Clearly, if $E\subset \H^n$ is $k$-regular, the implication $E\in\mathrm{Glem}(h,q)\Rightarrow E\in \mathrm{WGL}(h)$ holds for any $q\in[1,+\infty)$. The converse implication is not true in general. For an example in Euclidean $\mathbb{R}^2$, see \cite[Section 20]{David1}. By isometrically embedding $\mathbb{R}^2$ into $\mathbb{H}^2$, and using \cite[Lemma 4.38]{FV2}, this also produces an example of a set $E\subset \mathbb{H}^2$ with $E\in \mathrm{WGL}(\beta_{\infty,\mathcal{V}_1})$, yet $E\notin\mathrm{Glem}(\beta_{\infty,\mathcal{V}_1},q)$ for $q=2$. Moreover, since the example in \cite[Section 20]{David1} is not Euclidean rectifiable in $\mathbb{R}^2$, the set $E\subset \mathbb{H}^2$ fails to be Euclidean rectifiable and, in particular, cannot have BPLI in the sense of Definition \ref{def_BPLI}.

An example in $\mathbb{H}^1$ that satisfies $\mathrm{WGL}(\beta_{\infty,\mathcal{V}_1})$ (since it satisfies a strong geometric lemma with exponent $4$) yet fails $\Gamma\notin\mathrm{Glem}(\beta_{\infty,\mathcal{V}_1},q)$ for $q=2$ is provided by the curve  $\Gamma$ in Section \ref{s:Juillet}. 
\end{remark}

\begin{remark}
The weak geometric lemmas in Definition \ref{d:WGL} can be equivalently stated in terms of a Carleson set condition, similarly as the geometric lemmas allow for the reformulation  in 
Lemma \ref{l:DiffGeomLem}; see for instance \cite[Lemma 2.6.5]{Bate}. This also shows that the validity of a weak geometric lemma does not depend on the specific choice of dyadic system $\mathcal{D}$.
\end{remark}

We next discuss relations between geometric lemmas for the  functions in Definition \ref{d:Coefficient functions}.

\begin{proposition}\label{p:FromIotaToStratifAssumpt}
    Let $n\in \mathbb{N}$ and $k\in \{1,\ldots,n\}$. Assume that $E\subset \mathbb{H}^n$ is a $k$-regular set and $\mathcal{D}$ a dyadic system on $E$. Then we have, for $\lambda \geq 1$,
    \begin{displaymath}
        \widehat{\beta}_{1,\mathcal{V}_k}(\lambda Q)^4 \lesssim_{\lambda} \iota_{1,\mathcal{V}_k}(\lambda Q),\quad Q\in \mathcal{D}.    \end{displaymath}
    In particular, if
$E\in \mathrm{GLem}(\iota_{1,\mathcal{V}_k},1)$, then $E\in \mathrm{GLem}(\widehat{\beta}_{1,\mathcal{V}_k},4)$.
\end{proposition}

\begin{proof} Let $E$ be $k$-regular with a dyadic system $\mathcal{D}$. For simplicity, we denote $\mu:=\mathcal{H}^k\llcorner_E$. 
Fix $\epsilon>0$. Then, by definition of $\iota_{1,\mathcal V_k}(\lambda Q)$, there exists $V\in\mathcal V_k$ such that 
  \[\fint_{\lambda Q}\fint_{\lambda Q}\frac{\left|d(x,y)-d(P_V(x),P_V(y))\right|}{\mathrm{diam}(\lambda Q)}\,d\mu(x)\,d\mu(y)< \iota_{1,\mathcal V_k}(\lambda Q)+\epsilon.\]
  Therefore there exists at least a point $z_0\in \lambda Q$ such that \begin{equation}\label{eq_z_0}\fint_{\lambda Q}\frac{\left|d(z_0,y)-d(P_V(z_0),P_V(y))\right|}{\mathrm{diam}(\lambda Q)}\,d\mu(y)< \iota_{1,\mathcal{V}_k}(\lambda Q)+\varepsilon.\end{equation}
By Lemma \ref{lem_translations} it suffices to prove $\widehat{\beta}_{1,\mathcal{V}_k}(\lambda[z_0^{-1}\cdot Q])^4\lesssim \iota_{1,\mathcal{V}_k}(\lambda [z_0^{-1}\cdot Q])$, so without loss of generality, we will assume in the following that $z_0=0$. Moreover, 
by  Lemma \ref{l:NonAffineProj}, we can assume that $V=V_I\times \{0\}\in \mathcal V_k^0$. 
  In this case $P_V(y)=(\pi_{V_I}(\pi(y)),0)$. Then for every $y\in\lambda Q$ \begin{align}
    |d(z_0,y)-|\pi_{V_I}(\pi(z_0))-\pi_{V_I}(\pi(y))||&=\|y\| -|\pi_{V_I}(\pi(y))|\notag\\
    &\geq |\pi(y)|-|\pi_{V_I}(\pi(y))|\notag\\
    &= \frac{|\pi(y)|^2-|\pi_{V_I}(\pi(y))|^2}{|\pi(y)|+|\pi_{V_I}(\pi(y))|}=\frac{|\pi_{V_I^{\bot}}(\pi(y))|^2}{|\pi(y)|+|\pi_{V_I}(\pi(y))|}\geq \frac{|\pi_{V_I^{\bot}}(\pi(y))|^2}{2\, \mathrm{diam}(\lambda Q)}\notag\\
    &=\frac{d_{\mathrm{Eucl}}(\pi(y),V_I)^2}{2\,\mathrm{diam}(\lambda Q)}.\label{eq:estHalfStratif1}
\end{align}
This estimate will be useful for one of the two summands appearing in the definition of $\widehat{\beta}_{1,\mathcal{V}_k}(\lambda Q)$. For the other summand, we use a similar estimate.

Denoting $a:=\|y\|$ and $b:=|\pi_{V_I}(\pi(y))|$, we have
\begin{displaymath}
     |d(z_0,y)-|\pi_{V_I}(\pi(z_0))-\pi_{V_I}(\pi(y))||= a-b = \frac{a^4-b^4}{a^3 + a^2 b + ab^2 + b^3}\gtrsim \frac{a^4-b^4}{\mathrm{diam}(\lambda Q)^3}.
\end{displaymath} Here
\begin{align*}
    a^4 -b^4 &= |\pi(y)|^4 + 16 \pi_t(y)^2 - |\pi_{V_I}(\pi(y))|^4\\
    &=(|\pi_{V_I}(\pi(y))|^2 +|\pi_{V_I^{\bot}}(\pi(y))|^2)^2 + 16 \pi_t(y)^2 - |\pi_{V_I}(\pi(y))|^4\\
    &= |\pi_{{V_I}^{\bot}}(\pi(y))|^4 + 2 |\pi_{V_I}(\pi(y))|^2 |\pi_{{V_I}^{\bot}}(\pi(y))|^2 + 16 \pi_t(y)^2\\
    &\gtrsim \|P_V(y)^{-1}\cdot y\|^4\geq d(y,V)^4,
\end{align*}
where we used the fact that \begin{align*}\|P_V(y)^{-1}\cdot y\|^4&=|\pi(P_V(y)^{-1}\cdot y)|^4+16(\pi_t(P_V(y)^{-1}\cdot y))^2\\&=|\pi_{{V_I}^\bot}(\pi(y))|^4+16(\pi_t(y)-\omega(\pi_{V_I}(\pi(y)),\pi(y)))^2\\ &\leq |\pi_{{V_I}^\bot}(\pi(y))|^4+32\pi_t(y)^2+8|\pi_{V_I}(\pi(y))|^2 |\pi_{V_I^{\bot}}(\pi(y))|^2,\end{align*}
where $\omega$ is the form appearing in the group law \eqref{grouplaw}.
Thus
\begin{equation}\label{eq:estHalfStratif2}
     |d(z_0,y)-|\pi_{V_I}(\pi(z_0))-\pi_{V_I}(\pi(z))|| \gtrsim \frac{d(y,V)^4}{\mathrm{diam}(\lambda Q)^3}.
\end{equation} 

Finally combining \eqref{eq:estHalfStratif1} and \eqref{eq:estHalfStratif2}, we obtain
  \begin{align*}
       & \frac{d(z_0,y)-|\pi_{V_I}(\pi(z_0))-\pi_{V_I}(\pi(y))|}{\mathrm{diam}(\lambda Q)}\\&= \frac{1}{2}\frac{d(z_0,y)-|\pi_{V_I}(\pi(z_0))-\pi_{V_I}(\pi(y))|}{\mathrm{diam}(\lambda Q)} + \frac{1}{2}\frac{d(z_0,y)-|\pi_{V_I}(\pi(z_0))-\pi_{V_I}(\pi(y))|}{\mathrm{diam}(\lambda Q)}\\
        &\gtrsim \frac{d_{\mathrm{Eucl}}(\pi(y),V_I)^2}{\mathrm{diam}(\lambda Q)^2}+\frac{d(y,V)^4}{\mathrm{diam}(\lambda Q)^4}.
    \end{align*}
Recalling the definition of $\widehat{\beta}$, and applying Jensen's inequality with $\varphi(x)=x^2$ and $\varphi(x)=x^4$, respectively, we  deduce that
\begin{align*}
    \widehat{\beta}_{1,\mathcal{V}_k}(\lambda Q)^4 &\leq 
    \fint_{\lambda Q} \frac{\mathrm{dist}_{\mathrm{Eucl}}(\pi(y),V_I)^2}{\mathrm{diam}(\lambda Q)^2}+\frac{d(y,V)^4}{\diam (\lambda Q)^4}\,d\mu(y)\\
    &\lesssim \fint_{\lambda Q} \frac{|d(z_0,y)-|\pi_{V_I}(\pi(z_0))-\pi_{V_I}(\pi(y))||}{\diam (\lambda Q)}d\mu(y)\stackrel{\eqref{eq_z_0}}\leq \iota_{1,\mathcal V_k}(\lambda Q)+\varepsilon.
\end{align*}
We get the conclusion by the arbitrariness of $\varepsilon$.
 \end{proof}

\begin{proposition}\label{p:FromStratifiedToWeakAssumpt}
    Let $n\in \mathbb{N}$ and $k\in \{1,\ldots,n\}$. Assume that $E\subset \mathbb{H}^n$ is a $k$-regular set and $\mathcal{D}$ a dyadic system on $E$. Then we have, for $\lambda \geq 1$,
    \begin{equation}\label{eq:FromStratifiedToWeakAssumpt}
   \beta_{1,\pi,\mathcal{V}_k}(\lambda Q)^2 +\beta_{1,\mathcal{V}_k}(\lambda Q)^4 \leq      \widehat{\beta}_{1,\mathcal{V}_k}(\lambda Q)^4,\quad Q\in \mathcal{D}.
    \end{equation}
   Moreover, if $E\in \mathrm{GLem}(\widehat{\beta}_{1,\mathcal{V}_k},4)$, then $E\in \mathrm{GLem}(\beta_{1,\pi,A(2n,k)},2)$ and $E\in \mathrm{WGL}(\beta_{\infty,\mathcal{V}_k})$.
\end{proposition}

\begin{proof}
We recall from Definition \ref{d:Coefficient functions} that
\begin{displaymath}
    \widehat{\beta}_{1,\mathcal{V}_k}(\lambda Q)^4=\inf_{V\in\mathcal V_k}\left\{\left[\fint_{\lambda Q}\frac{d_{\mathrm{Eucl}}(\pi(y),\pi(V))}{\diam (\lambda Q)}\,d\mathcal{H}^k(y)\right]^{2}+\left[\fint_{\lambda Q}\frac{d(y,V)}{\diam (\lambda Q)}\,d\mathcal{H}^k(y)\right]^{4}\right\},
\end{displaymath}
which immediately implies \eqref{eq:FromStratifiedToWeakAssumpt}. Clearly, this also shows that if $E\in \mathrm{GLem}(\widehat{\beta}_{1,\mathcal{V}_k},4)$, then $E\in \mathrm{GLem}(\beta_{1,\pi,\mathcal{V}_k},2)$ and $E\in \mathrm{GLem}(\beta_{1,\mathcal{V}_k},4)$. Since $\beta_{1,\pi,A(2n,k)}\leq \beta_{1,\pi,\mathcal{V}_k}$ by Remark \ref{r:AffVsFullProjBeta}, this implies also  $E\in \mathrm{GLem}(\beta_{1,\pi,A(2n,k)},2)$.

To conclude the proof of the proposition, it remains to show that $E\in \mathrm{GLem}(\beta_{1,\mathcal{V}_k},4)$ implies $E\in \mathrm{WGL}(\beta_{\infty,\mathcal{V}_k})$. This follows by the same argument as in Euclidean spaces, which works irrespective of the integrability exponents; see \cite[p.27]{David1} or \cite[p.27]{David2}.
\end{proof}

\begin{proposition}\label{p:FromHorizToStartif}
    Let $n\in \mathbb{N}$ and $k\in \{1,\ldots,n\}$. Assume that $E\subset \mathbb{H}^n$ is a $k$-regular set and $\mathcal{D}$ a dyadic system on $E$. Then we have, for $\lambda \geq 1$,
    \begin{equation}\label{eq:FromStratifiedToWeakAssumpt1}
  \widehat{\beta}_{1,\mathcal{V}_k}(\lambda Q)^4 \lesssim \beta_{1,\mathcal{V}_k}(\lambda Q)^2,\quad Q\in \mathcal{D}.
    \end{equation}
In particular, if $E\in \mathrm{GLem}(\beta_{1,\mathcal{V}_k},2)$, then 
   $E\in \mathrm{GLem}(\widehat{\beta}_{1,\mathcal{V}_k},4)$.
\end{proposition}

\begin{proof}
Recalling that $d_{\mathrm{Eucl}}(\pi(x),\pi(y))\leq d(x,y)$ for all $x,y\in \mathbb{H}^n$, this follows directly from Definition \ref{d:Coefficient functions} since
\begin{align*}
    \widehat{\beta}_{1,\mathcal{V}_k}(\lambda Q)^4&=\inf_{V\in\mathcal V_k}\left\{\left[\fint_{\lambda Q}\frac{d_{\mathrm{Eucl}}(\pi(y),\pi(V))}{\diam (\lambda Q)}\,d\mathcal{H}^k(y)\right]^{2}+\left[\fint_{\lambda Q}\frac{d(y,V)}{\diam (\lambda Q)}\,d\mathcal{H}^k(y)\right]^{4}\right\}\\
   & \leq \inf_{V\in\mathcal V_k}\left\{2\left[\fint_{\lambda Q}\frac{d(y,V)}{\diam (\lambda Q)}\,d\mathcal{H}^k(y)\right]^{2}\right\}.\qedhere
\end{align*}
\end{proof}

\subsubsection{Multiresolution families}

\begin{definition}[Dyadic nets and multiresolution families]\cite{schul,Juillet}\label{d:multiRes}
    Given $E\subset \H^n$ compact, a \textit{dyadic net} in $E$ is a collection $\Delta=(\Delta_j)_{j\in\mathbb Z}$ of subsets of $E$ such that
    \begin{enumerate}
        \item[(i)] $\Delta_{j}\subset \Delta_{j+1}$ for every $j\in\mathbb Z$;
        \item[(ii)] If $x_1,x_2\in \Delta_j$, then $x_1=x_2$ or $d(x_1,x_2)>2^{-j};$
        \item[(iii)] If $y\in E$ and $j\in\mathbb Z$, there exists $x\in \Delta_j$ such that $d(y,x)\leq 2^{-j}$.
    \end{enumerate}
    For fixed dyadic net $(\Delta_j)_j$ in $E$, $1\leq p\leq +\infty$ and $d\in\mathbb N$, let also
    \begin{equation}\label{eq:BDef}
    B^\Delta_{p,d}(E):=\sum_{j\in\mathbb Z}2^{-j}\sum_{x\in\Delta_j}\beta^2_{p,\mathcal V_d}(x, A\,2^{-j}),\end{equation}
    for some constant $A>1$.
    The corresponding \textit{multiresolution family} is the collection
    \[ \mathcal G:=\{B(x, A\,2^{-j}): x\in\Delta_j, \,j\in\mathbb Z\}.\]
\end{definition}
In Section \ref{s:Juillet}, we will need the following connection (for $k=1$) between a geometric lemma and a multiresolution family as in Definition \ref{d:multiRes}. 
\begin{lemma}\label{Glemvsmultires}
If $E\subset\mathbb{H}^n$ is bounded and $k$-regular with  $E\in{\rm GLem}(\beta_{p,\mathcal V_k},2)$, then $B_{p,k}^\Delta(E)<\infty$ for every dyadic net in $E$.   
\end{lemma}
The proof is standard, so we only sketch the idea. Let $\mathcal{D}$ be a system of dyadic cubes as in Definition \ref{dt:dyad}. Recalling Remark \ref{r:Dyad2}, we may assume that the conditions hold with $\varrho=1/2$, so that a cube of generation $j$ satisfies $\mathrm{diam}(Q)\sim 2^{-j}$.
The idea for Lemma \ref{Glemvsmultires} is to assign to each $B\in \mathcal{G}$ (as in Definition \ref{d:multiRes}) a cube $Q_B \in \mathcal{D}$ so that $B\cap E \subset K Q_B$ for a constant independent of $B$, and then bound the $\beta$-number of $B$ in terms of the corresponding number for $K Q_B$. The same $Q\in \mathcal{D}$ might arise as $Q=Q_B$ for several $B\in \mathcal{G}$, but this multiplicity can be controlled using the $k$-regularity of $E$ and property (ii) of the dyadic net $\Delta$. The desired conclusion then follows from the assumed geometric lemma, which we may equivalently consider in the version with dyadic cubes or double integrals, according to Lemma \ref{l:DiffGeomLem}.

\subsection{Corona decompositions}\label{ss:CoronaDef}
In this section, we define different versions of corona decompositions for $k$-regular sets in $\mathbb{H}^n$. All of them are based on the notion of \emph{coronization}, which is an adaptation of the original definition by David and Semmes \cite[(3.14)-(3.18)]{David2}.

\begin{definition}[Coronization]\label{d:coronization}
 Let $n\in \mathbb{N}$, $k\in \{1,\ldots,n\}$, and $E\subset \mathbb{H}^n$ be $k$-regular with a dyadic system $\mathcal{D}$.  A \emph{coronization of $E$ with constant $C>0$} is a decomposition $\mathcal{D}=\mathcal{G} \dot{\cup}\mathcal{B}$ of the dyadic cubes into a  \emph{good set} $\mathcal G$ and a \emph{bad set}
 $\mathcal B$ with the following properties:
    \begin{itemize}
        \item for every $R\in\mathcal{D}$ \begin{equation*}
            \sum_{Q\in\mathcal B, Q\subseteq R} \mathcal{H}^k(Q)\leq C\mathcal{H}^k(R);
        \end{equation*}
        \item the good set $\mathcal G$ can be partitioned into a family $\mathcal F$ of disjoint \emph{trees} $\mathcal S$ (also called \emph{stopping time regions}) such that:
        \begin{itemize}
           \item each $\mathcal S\in\mathcal F$ is \textit{coherent}: it has a (unique) maximal element, denoted by $Q(\mathcal S)$, that contains all other elements of $\mathcal{S}$ as subsets, has the property that if 
        $Q\in \mathcal S, Q'\in\mathcal{D}$ with $Q\subseteq Q'\subseteq Q(\mathcal S)$, then $Q'\in\mathcal S$, and finally is such that if $Q\in \mathcal{S}$, then either all of the children of $Q$ lie in $\mathcal{S}$ or none of them do;
         \item for every $R\in\mathcal{D}$ \begin{equation}\label{carlesonpackingcond}
            \sum_{\mathcal S\in\mathcal F, Q(\mathcal S)\subseteq R} \mathcal{H}^k(Q(\mathcal S))\leq C\mathcal{H}^k(R).
        \end{equation}
    \end{itemize}
 \end{itemize}
\end{definition}
Coronizations are useful if the trees $\mathcal{S}$ are such that $E$ has good (approximation) properties at the scales and locations of each $Q\in \mathcal{S}$. In this case we will say that $E$ admits a \emph{corona decomposition} (with respective properties). As in the Euclidean case (\cite[p.20]{David1} or \cite[p.58]{David2}) the existence of such a decomposition should be independent of the specific choice of dyadic system in the definition, but we will not need this fact.

We introduce three definitions of corona decomposition for regular subsets $E\subset \H^n$.
The first one, a corona decomposition by intrinsic Lipschitz graphs with small constants, 
%was introduced for $k=n=1$ in \cite{MR4299821}, motivated by an application to singular integral operators on regular curves. It was shown that $1$-dimensional intrinsic Lipschitz graphs (with arbitrary constants) satisfy such a corona decomposition. This was later extended to all $\mathbb{H}^n$ in \cite{didonato}.
%The definition
is inspired by the classical definition in $\R^n$ of a corona decomposition 
%by Lipschitz graphs with small constants 
(see  \cite[Definition 3.19]{David2} or \cite{David1}). It is also related to a corona decomposition for intrinsic Lipschitz functions by intrinsic Lipschitz functions with small constants, which was established for $k=n=1$ in \cite{MR4299821}, motivated by an application to singular integral operators on regular curves. Moreover, it is an instance of the general corona decompositions  studied in \cite{MR4485846}. By \cite[Theorem 1.1]{MR4485846}, a set with (ILG-C) therefore has big pieces squared (in the sense of \cite[Definition 2.11]{MR4485846}) of intrinsic Lipschitz graphs.
\begin{definition}[Corona decomposition by intrinsic Lipschitz graphs (ILG-C)]\label{d:ILG-C}
Let $n\in \mathbb{N}$, $k\in \{1,\ldots,n\}$, and $E\subset \mathbb{H}^n$ be $k$-regular.  We say that $E$ admits a \textit{corona decomposition by intrinsic Lipschitz graphs} (ILG-C) if, for every $\eta>0$, there exists a constant $C=C(\eta)>0$ such that $E$ admits a coronization $\mathcal{D}=\mathcal{G}\dot{\cup}\mathcal{B}$ with constant $C$  
where,  for each $\mathcal S\in\mathcal F$, there exists a $k$-dimensional intrinsic Lipschitz graph $\Gamma=\Gamma_\mathcal S$ with intrinsic Lipschitz constant bounded by $\eta$ so that, for all $Q\in\mathcal S$, it holds that
     \begin{equation}\label{eq:ILGcorona}
     \dist(x,\Gamma)\leq \eta\,\diam ( Q)\quad\text{if }x\in E, \,\dist(x,Q)\leq \diam(Q).
     \end{equation}
\end{definition}

\begin{remark}\label{r:ILG-Cmodif}
We will need an improvement of (ILG-C), where we can replace ``$d(x,Q)\leq \mathrm{diam}(Q)$'' by ``$d(x,Q)\leq N \mathrm{diam}(Q)$'' in Definition \ref{d:ILG-C} for some constant $N>1$, with the constant $C=C(\eta,N)$ in the coronization allowed to depend also on $N$. This is a priori a stronger condition, but in fact the two versions are equivalent. In the Euclidean case, this is stated on \cite[p.20]{David1}, see also
\cite[Lemma 3.31]{David2}, The proof does not use any features of Euclidean spaces and relies purely on axiomatic properties of dyadic families and coronizations, see the outline of the proof in \cite{David1,David2}.
\end{remark}

The next version of the corona decomposition is inspired by \cite[Lemma 4.3]{Hahlomaa}. 
Here and in the following, $K_0$ is a fixed constant, which will be chosen big enough depending on $k$ and $C_E$, as in \cite[p.4]{Hahlomaa}. We also assume that at least $K_0\geq 2$.

\begin{definition}[Corona decomposition by horizontal planes (P-C)% and (P$L$-C)
]\label{d:P-C}
Let $n\in \mathbb{N}$, $k\in \{1,\ldots,n\}$, and $E\subset \mathbb{H}^n$ be $k$-regular.  
%\begin{itemize}
%\item 
We say that $E$ admits a \textit{corona decomposition  by horizontal planes} (P-C)  if, for every $\eta>0$, there exists a constant $C=C(\eta)>0$ such that $E$ admits a coronization $\mathcal{D}=\mathcal{G}\dot{\cup}\mathcal{B}$ with constant $C$  
where,  for each $\mathcal S\in\mathcal F$,
there exists an affine horizontal plane $V_\mathcal S$ such that, whenever $x,y\in K_0 Q(\mathcal S)$ are chosen in such a way that $d(x,y)>\eta\min\{h_\mathcal S(x),h_\mathcal S(y)\}$, then 
\begin{equation}\label{eq:P-C}
d(x,y)\leq (1+2\eta)d(P_{V_\mathcal S}(x), P_{V_\mathcal S}(y)),\end{equation}
where 
\[h_\mathcal S(x):=\inf\{d(x,Q)+\diam(Q):Q\in\mathcal S\}.\]
%\item We say that $E$ admits a \textit{corona decomposition  by horizontal planes with constant $L>1$}  (P$L$-C) if has the properties stated in the definition of (P-C) except that \eqref{eq:P-C}  is replaced by 
%\begin{equation}\label{eq:PL-C}
%d(x,y)\leq L d(P_{V_\mathcal S}(x), P_{V_\mathcal S}(y))\end{equation}
%and the constant ``$C$'' is allowed to depend also on $L$.
%\end{itemize}
\end{definition}

Hahlomaa worked with a condition analogous to (P-C), but with \eqref{eq:P-C} only required to hold for $d(x,y)>D^{-2} \min\{h_{\mathcal{S}}(x),h_{\mathcal{S}}(y)\}$ for a constant $D$ depending on $k$ and the Ahlfors regularity constant of $E$. However, an inspection of the proof in \cite{Hahlomaa} reveals that in fact (P-C) can be obtained, even under our weaker Carleson-type assumptions on $E$, see the proof of Lemma \ref{l:HahlLem4.3}. The condition (P-C) can be compared more easily with other corona decompositions from the literature.

%\begin{remark}
%If $x,y\in \textcolor{blue}{K_0}Q(\mathcal{S})$ satisfy $d(x,y)>D^{-2}\min\{h_{\mathcal{S}}(x),h_{\mathcal{S}}(y)\}$, then they also satisfy $d(x,y)>C\max\{h_{\mathcal{S}}(x),h_{\mathcal{S}}(y)\}$ for a constant $C$ that depends only on $\alpha$ and $D$.
%\end{remark}

The third type of corona decomposition we consider is the one by normed spaces, introduced by Bate, Hyde, and Schul \cite[Definitions 6.4.1, 6.4.2]{Bate} in metric spaces:
\begin{definition}[Corona decomposition by normed spaces (N-C)]\label{d:N-C} Let $n\in \mathbb{N}$, $k\in \{1,\ldots,n\}$, and $E\subset \mathbb{H}^n$ be $k$-regular. 
We say that $E\subset \H^n$ admits a \textit{Corona decomposition by normed spaces} (N-C) \textit{with constants $\Lambda_1,\Lambda_2\geq 1$} if for every $\eta>0$ there exists $C=C(\eta)>0$ such that $E$ admits a coronization $\mathcal{D}=\mathcal{G}\dot{\cup}\mathcal{B}$ with constant $C$  
where,  for each $\mathcal S\in\mathcal F$,
 there exists a norm $\|\cdot\|_\mathcal S$ on $\R^k$, a point $y_\mathcal S\in\R^k$ and a map $\varphi_\mathcal S: 3B_{Q(\mathcal S)}\to B_{\|\cdot\|_{\mathcal{S}}}(y_\mathcal S, 3 \mathrm{diam}(Q(\mathcal S))) $ such that if $ Q\in\mathcal S$ and $x,y\in 3B_Q$ satisfy $d(x,y)\geq \eta \mathrm{diam}(Q)$, then
     \[\frac{1}{\Lambda _1}d(x,y)\leq \|\varphi_\mathcal S(y)-\varphi_\mathcal S (x)\|_\mathcal S\leq \Lambda_2 d(x,y),\] 
     where  $B_Q$ is defined as in \cite{Bate}, while $B_{\|\cdot\|_{\mathcal{S}}}$ denotes a ball with respect to the norm $\|\cdot\|_{\mathcal{S}}$.
\end{definition}
In the subsequent Lemmas \ref{l:PCimpliesILGC}--\ref{l:P-CtoN-C}, we discuss relations between the different corona decompositions stated in Definitions \ref{d:ILG-C} (ILG-C), \ref{d:P-C} (P-C), and \ref{d:N-C} (N-C).

\begin{lemma}\label{l:PCimpliesILGC} Let $n\in \mathbb{N}$, $k\in \{1,\ldots,n\}$, and $E\subset \mathbb{H}^n$ be $k$-regular. If $E$ satisfies (P-C), then also (ILG-C).    
\end{lemma}

\begin{proof}
Fix $\eta>0$. We will find $\theta=\theta(\eta)=\theta(\eta,K_0,k,C_E)>0$ sufficiently small and apply (P-C) with parameter $\frac{\theta}{K_0}$. 
Let $\mathcal{D}=\mathcal{G}\dot{\cup}\mathcal{B}$ be the coronization given by (P-C) applied with ``$\eta$'' replaced by ``$\frac{\theta}{K_0}$''; this will also serve as a coronization for (ILG-C). We only need to verify the existence of intrinsic Lipschitz graphs as in \eqref{eq:ILGcorona}.

Let $\mathcal{F}$ be the forest of trees associated with the coronization, and fix $\mathcal S\in\mathcal F$. For every $x,y\in K_0Q(\mathcal S)$, with $x\neq y$, let us denote by $Q_{x,y}$ a smallest -- with respect to diameter -- cube 
in $\mathcal{S}$ 
with $x,y\in K_0 Q_{x,y}$. (While the cubes in $\mathcal{S}$ of a fixed generation are disjoint, the $K_0$-enlarged cubes generally are not, so that the choice of $Q_{x,y}$ need not be unique, but any admissible choice will do. Moreover, the assumption that $x,y\in K_0 Q(\mathcal{S})$ ensures that there exists at least one cube with the desired properties.)

We perform a construction inspired by \cite[Remark 6.4.3]{Bate}. Let $\mathcal N$ be a 
%maximal net 
set in $K_0Q(\mathcal S)$ with the following property \textbf{(P)}:
 if $x,y\in\mathcal N$ are distinct, then $d(x,y)\geq \theta\,\diam(Q_{x,y})$, and $\mathcal{N}$ is \emph{maximal} in the sense that it is impossible to add an element to $\mathcal{N}$ while preserving the separation condition.
By Zorn's lemma, the existence of $\mathcal N$ is guaranteed.\\
 We claim that, choosing $\theta=\theta(\eta)$ small enough, then for every $x\in K_0 Q(\mathcal S)$ and any cube $Q\in\mathcal S$ such that $x\in K_0 Q$, there exists $x_\mathcal N\in\mathcal N$ satisfying $d(x,x_\mathcal N)\leq \eta\, \diam(Q)$. In fact, 
let $x\in K_0 Q$ for some $Q\in\mathcal S$ and assume by contradiction that 
\begin{equation}\label{eq:CounterAss}
d(x,y)>\eta\,\diam(Q)\quad \text{for every }y\in\mathcal N.
\end{equation}
In particular $x\not\in\mathcal N$. If we prove that $\mathcal N \cup \{x\}$ still satisfies the separation property in \textbf{(P)}, we then conclude by maximality of $\mathcal N$. In other words, for fixed $y\in\mathcal N$ we want to show that $d(x,y)\geq \theta\, \diam(Q_{x,y})$. Let $R\in\mathcal S$ be the smallest ancestor of $Q$ such that 
\begin{equation}\label{ancestor}\diam(R)\geq \frac{1}{K_0-1}d(x,y)+\diam(Q).\end{equation}
If such an ancestor does not exist in $\mathcal S$, then set $R=Q(\mathcal S)$. In the first case, we have
\[d(y,R)\leq d(y,Q)\leq d(x,y)+d(x,Q)\leq d(x,y)+(K_0-1)\diam(Q)\leq (K_0-1)\diam (R).\]
This means that $y\in K_0 R$. The same conclusion also holds in the case $R=Q(\mathcal S)$, since clearly $y\in K_0 Q(\mathcal S)=K_0 R$. Since also $x\in K_0 R$, we get that $\mathrm{diam}(R)\geq \mathrm{diam}(Q_{x,y})$ by the choice of $Q_{x,y}$. 
Then we can estimate, by the minimality of $R$ and properties of dyadic cubes (see \eqref{eq:H_(14)} and Definition \ref{dt:dyad}), and our counter-assumption \eqref{eq:CounterAss} that
\[\begin{split}D^{-2}\rho\,\diam (R)\leq \tfrac{1}{K_0-1}d(x,y)+\diam(Q)\overset{\eqref{eq:CounterAss}}{\leq} \tfrac{1}{K_0-1}d(x,y)+\tfrac{1}{\eta}d(x,y)=\left(\tfrac{1}{K_0-1}+\tfrac{1}{\eta}\right)d(x,y).\end{split}\]
This holds also in the case in which we could not find the ancestor satisfying \eqref{ancestor}, since in that case
\[D^{-2}\rho\,\diam (R)\leq \diam(R)\leq \frac{1}{K_0-1}d(x,y)+\diam(Q).\]
Hence we find that 
\[d(x,y)\geq \frac{D^{-2}\rho}{\frac{1}{K_0-1}+\frac{1}{\eta}}\diam(R)\]
and, choosing 
\begin{equation}\label{eq:ChoiceTheta}\theta\leq \frac{D^{-2}\rho}{\frac{1}{K_0-1}+\frac{1}{\eta}} \,(\lesssim_{k,C_E} \eta),
\end{equation}
we finally get
\[d(x,y)\geq \theta\,\diam(R)\geq \theta \,\diam(Q_{x,y})\] and the claim is proved.
Now, by construction of $\mathcal N$, for any $x,y\in \mathcal N$, we know \[d(x,y)\geq \theta\, \diam(Q_{x,y})\geq \tfrac{\theta}{K_0} h_\mathcal S(x)\geq \tfrac{\theta}{K_0}\,\min\{h_\mathcal S(x),h_\mathcal S(y)\}.\]
The second inequality follows by the fact that $x\in K_0 Q_{x,y}$, so that $h_\mathcal S(x)\leq d(x,Q_{x,y})+\diam(Q_{x,y})\leq (K_0-1)\diam(Q_{x,y})+\diam(Q_{x,y})=K_0\diam(Q_{x,y}).$
By (P-C), applied with $\frac{\theta}{K_0}$, we deduce that there exists an affine horizontal plane $V_{\mathcal{S}}$ such that
\begin{equation}\label{eq:InjCond}d(x,y)\leq (1+2\tfrac{\theta}{K_0})d(P_{V_\mathcal S}(x),P_{V_{\mathcal S}}(y)),\quad x,y\in \mathcal{N}.\end{equation}
By Lemma \ref{l:NonAffineProj}, we may without loss of generality assume that $V_{\mathcal{S}}$ is a horizontal \emph{subgroup}. Condition \eqref{eq:InjCond} implies that the 
horizontal projection ${P_{V_{\mathcal{S}}}}|_{\mathcal{N}}$ is injective. Thus there exists a well-defined map
 $\Phi_{\mathcal{S}}:=({P_{V_{\mathcal{S}}}}|_{\mathcal{N}})^{-1}:P_{V_S}(\mathcal N)\to\mathcal N$ that sends $P_{V_\mathcal S}(x)\mapsto x$, and $\mathcal{N}= \Phi_{\mathcal{S}}(P_{V_{\mathcal{S}}}(\mathcal{N}))$ is an in intrinsic graph over $P_{V_{\mathcal{S}}}(\mathcal{N})\subset V_{\mathcal{S}}$. More precisely, by the injectivity of ${P_{V_{\mathcal{S}}}}|_{\mathcal{N}}$, we obtain that for every $v\in P_{V_{\mathcal{S}}}(\mathcal{N})$, there exists a unique $\varphi_{\mathcal{S}}(v)$ in the complementary orthogonal subgroup $W$ such that $v\cdot \varphi_{\mathcal{S}}(v) \in \mathcal{N}$. Then $\Phi_{\mathcal{S}}$ is the graph map of $\varphi_{\mathcal{S}}$, and by \eqref{eq:InjCond}, the latter 
 is intrinsic Lipschitz with intrinsic Lipschitz constant $\lesssim \sqrt[4]{\theta}$  (see Proposition \ref{prop_lipgraph}). Thus, by \eqref{eq:ChoiceTheta}, the intrinsic Lipschitz constant of $\varphi_{\mathcal{S}}$ can be bounded by $\lesssim_{k,C_E} \sqrt[4]{\eta}$.
 By \cite{didonato}, we can extend $\varphi_{\mathcal{S}}$ to the full space $V_{\mathcal S}$ with a control on the Lipschitz constant. Let $\Gamma_\mathcal S$ be the associated intrinsic Lipschitz graph.

Let now $Q\in\mathcal S$ and $x\in E$ with $d(x,Q)\leq \diam(Q)$. Then $x\in K_0 Q$ and, by the previous claim, there exists $x_\mathcal N\in\mathcal N$ satisfying $d(x,y)\leq \eta\,\diam(Q)$. Therefore
\begin{equation}\label{eq:GraphApprox}d(x,\Gamma_{\mathcal{S}})\leq d(x,\mathcal N)\leq d(x,x_\mathcal N)\leq \eta \,\diam(Q).\end{equation}
This proves that the chosen coronization satisfies the condition required in (ILG-C). To be precise, the intrinsic graph we have constructed may not have intrinsic Lipschitz constant bounded by $\eta$; this is not a problem, since the constant is controlled by $\eta$, in the sense that it goes to $0$ as $\eta \to 0$, so by choosing $\theta(\eta)$ initially even smaller if needed, we can arrange that the intrinsic Lipschitz constant of $\Gamma_{\mathcal{S}}$ is less than $\eta$ while also \eqref{eq:GraphApprox} holds.
\end{proof}

Next, we consider the converse of the implication in Lemma \ref{l:PCimpliesILGC}.

\begin{lemma}\label{lem_equivcor} Let $n\in \mathbb{N}$, $k\in \{1,\ldots,n\}$, and $E\subset \mathbb{H}^n$ be $k$-regular. If $E$ satisfies  (ILG-C), then also (P-C).    
\end{lemma}

\begin{proof}
Fix $\eta>0$. Our goal is to verify the conditions of (P-C) for this parameter. Find $0<\overline{\eta}\leq \eta$ such that \begin{equation}\label{etabar}\frac{1+\overline \eta^2}{1-6\,\overline \eta^2D^2\rho^{-1}-6\overline \eta}\leq 1+2\eta.\end{equation} Apply now (ILG-C) with parameter $\overline{\eta}^2$ and find the corresponding coronization $\mathcal{D}=\mathcal{G} \dot{\cup}\mathcal{B}$ with forest $\mathcal{F}$ and  constant $C=C(\overline{\eta}^2)=C(\eta,k,C_E)$.
The only property that has to be verified is that for each $\mathcal S\in\mathcal F$ there exists an affine horizontal plane $V_\mathcal S$ such that, whenever $x,y\in {K_0}Q(\mathcal S)$ are chosen in such a way that $d(x,y)>\eta\min\{h_\mathcal S(x),h_\mathcal S(y)\}$, then 
\begin{equation}\label{eq:GoalHorizPlaneCorona}
     d(x,y)\leq (1+2\eta)d(P_{V_\mathcal S}(x), P_{V_\mathcal S}(y)).
     \end{equation}
  Fix $\mathcal S\in\mathcal F$. We take as affine horizontal plane $V_\mathcal S$ the domain of the Lipschitz graph $\Gamma_\mathcal S$ which is associated to $\mathcal{S}$ by the (ILG-C) property. Suppose now $x,y\in {K_0}Q(\mathcal S)$ with $d(x,y)>\eta\min\{h_\mathcal S(x),h_\mathcal S(y)\}$. Assume without loss of generality that $\min\{h_\mathcal S(x),h_\mathcal S(y)\}=h_\mathcal S(x)$. By definition of $h_\mathcal S$, there exists a cube $Q\in\mathcal S$ such that $d(x,y)>\eta (\mathrm{diam}(Q)+d(x,Q))$. Consider now the minimal cube $R\supset Q$, $R\in \mathcal{S}$, such that $\mathrm{diam}({K_0}R)\geq d(x,y)$. Such a cube exists since $\mathrm{diam}({K_0}Q(\mathcal S))\geq d(x,y)$. Then there are two possibilities:
  \begin{itemize}
      \item If there exists a child $\widehat R$ of $R$ which contains $Q$, then by minimality of $R$ and by the properties of dyadic cubes,
      \[\frac{\mathrm{diam}(R)}{D^2\rho^{-1}}\leq \mathrm{diam}(\widehat R){\leq \mathrm{diam}(K_0\widehat R)}< d(x,y),\] %where $\alpha=1/\varrho$ is as in \cite[(14)]{Hahlomaa} and Definition \ref{dt:dyad}.
      Hence $\mathrm{diam}(R)\leq D^2\rho^{-1} \,d(x,y)$;
      \item If instead there exists no child of $R$ containing $Q$, this implies that $R=Q$. Therefore
      \[\overline\eta\, \mathrm{diam}(R)\leq \eta \,\mathrm{diam}(R)=\eta \,\mathrm{diam}(Q)\leq \eta\,(\mathrm{diam}(Q)+d(x,Q))<d(x,y),\]
      which gives $\mathrm{diam}(R)\leq \overline\eta^{-1}d(x,y)$.
  \end{itemize}
  In any case, we get $\mathrm{diam}(R)\leq (D^2\rho^{-1}+\overline\eta^{-1})d(x,y)$. 
  
  Moreover $d(x,R)\leq d(x,Q)<\eta^{-1}d(x,y)\leq \eta^{-1}\mathrm{diam}({K_0}R){\leq 2\eta^{-1}K_0 \mathrm{diam}(R)}$. Similarly, by the triangle inequality, $d(y,R)\leq d(x,y)+d(x, R)\leq (\eta^{-1}+1)\mathrm{diam}({K_0}R){\leq 2(\eta^{-1}+1)K_0\mathrm{diam}(R)}$. Therefore we can apply (ILG-C) (with the improvement from Remark \ref{r:ILG-Cmodif} for the constant $N=2(\eta^{-1}+1)K_0 $) and get
  \begin{equation}\label{eq:LipGraphApproxN}
  d(x,\Gamma_{\mathcal{S}})\leq \overline\eta^2 \mathrm{diam}(R)\quad \text{and}\quad d(y,\Gamma_{\mathcal{S}})\leq \overline\eta^2 \mathrm{diam}(R)
  \end{equation}
  since $d(x,R)\leq N \mathrm{diam}(R)$ and $d(y,R)\leq N \mathrm{diam}(R)$.

  Let now $x_\Gamma$ and $y_\Gamma$ be two points on $\Gamma_{\mathcal{S}}$ which realize $d(x,\Gamma_{\mathcal{S}})$ and $d(y,\Gamma_{\mathcal{S}})$ respectively. Since $\Gamma_{\mathcal{S}}$ is an intrinsic Lipschitz graph with constant smaller than $\overline\eta^2$, we get by metric Lipschitz continuity of the graph map \eqref{eq:GraphMapMetricLip} that 
  \begin{equation}\label{eq:MetricLip}
  d(x_\Gamma, y_\Gamma)\leq (1+\overline\eta^2) d(P_{V_\mathcal S}(x_\Gamma), P_{V_\mathcal S}(y_\Gamma)).
  \end{equation}
  Therefore, if $\overline\eta$ is small enough, \eqref{eq:LipGraphApproxN},\eqref{eq:MetricLip}, and the $1$-Lipschitz property of $P_{V_{\mathcal{S}}}$ yield that
  \[\begin{split}d(x,y)&\leq d(x, x_\Gamma)+ d(x_\Gamma, y_\Gamma)+ d(y,y_\Gamma)
  %\overset{\eqref{eq:LipGraphApproxN},\eqref{eq:MetricLip}}
  {\leq} (1+\overline\eta^2) d(P_{V_\mathcal S}(x_\Gamma), P_{V_\mathcal S}(y_\Gamma))+2\overline\eta^2 \mathrm{diam}(R) \\ &\leq (1+\overline\eta^2)d(P_{V_\mathcal S}(x), P_{V_\mathcal S}(y))+(1+\overline\eta^2)d(P_{V_\mathcal S}(x), P_{V_\mathcal S}(x_\Gamma))+(1+\overline\eta^2)d(P_{V_\mathcal S}(y), P_{V_\mathcal S}(y_\Gamma))\\&\quad +2\overline\eta^2 \mathrm{diam}(R)\\ &\leq (1+\overline\eta^2)d(P_{V_\mathcal S}(x), P_{V_\mathcal S}(y))+(1+\overline\eta^2)d(x,x_\Gamma)+(1+\overline\eta^2)d(y, y_\Gamma)+2\overline\eta^2 \mathrm{diam}(R)\\ &\leq (1+\overline\eta^2)d(P_{V_\mathcal S}(x), P_{V_\mathcal S}(y))+6\overline\eta^2 \mathrm{diam}(R)\\
  &\leq (1+\overline\eta^2)d(P_{V_\mathcal S}(x), P_{V_\mathcal S}(y))+6\overline\eta^2 (D^2\rho^{-1} +\overline\eta^{-1})d(x,y).\end{split}\]
  If $\overline\eta$ is small enough, absorbing the last term by the left-hand side of the inequality, and using \eqref{etabar}, we get the desired estimate \eqref{eq:GoalHorizPlaneCorona}.
  \end{proof}

If a set $E$ admits a corona decomposition by horizontal planes (P-C) then, by definition, for each tree $\mathcal{S}\in \mathcal{F}$, there is a horizontal projection with bi-Lipschitz constant close to $1$ when restricted to cubes $Q\in \mathcal{S}$. Naturally, this implies an analogous condition where ``close to $1$'' is replaced by a definite constant. 

\begin{lemma}\label{l:P-CtoN-C}
    Let $n\in \mathbb{N}$, $k\in \{1,\ldots,n\}$, and $E\subset \mathbb{H}^n$ be $k$-regular. If $E$ satisfies  (P-C), then also (N-C) with constants $\Lambda_1=L$ and $\Lambda_2=1$ for arbitrary $L>1$.
\end{lemma}

\begin{proof}
    Let $n\in \mathbb{N}$, $k\in \{1,\ldots,n\}$, and $E\subset \mathbb{H}^n$ be $k$-regular. Let $L>1$. If $E$ satisfies (P-C), then, for any $\eta$ sufficiently small, we get from \eqref{eq:P-C} that 
\begin{equation}\label{eq:PL-C}
d(x,y)\leq L d(P_{V_\mathcal S}(x), P_{V_\mathcal S}(y))\end{equation}
whenever $x,y\in 3B_Q$ with $d(x,y)\geq \eta\diam(Q)$. % if we allow  the constant ``$C$''  to depend also on $L$.
Then $E$ satisfies the conditions of (N-C) where $(\R^k,\|\cdot\|_\mathcal S)$ is $V_S$ (as in Definition \ref{d:P-C}) equipped with the induced distance, $\varphi_\mathcal S$ is the projection map onto $V_\mathcal S$,
$\Lambda_1=L$ and $\Lambda_2=1$. Notice that the constant ``$C$'' in the definition of (N-C) is allowed to depend also on $\Lambda_1=L$.
\end{proof}

%\begin{lemma}\label{l:P-CtoPL-C} Let $n\in \mathbb{N}$, $k\in \{1,\ldots,n\}$, and $E\subset \mathbb{H}^n$ be $k$-regular. If $E$ satisfies  (P-C), then also (P$L$-C) for every $L>1$.    
%\end{lemma}

%\begin{proof}
%  Assume that $E$ satisfies (P-C), and let $\eta>0$ and $L>1$ be arbitrary constants. In order to verify (P-$L$C) with these constant, 
%   Set $\widetilde{\eta}:=\min \{\eta, (L-1)/2\}$ and apply (P-C) with $\widetilde{\eta}$. 
%\end{proof}

%\begin{lemma}\label{l:PL-CtoN-C} Let $n\in \mathbb{N}$, $k\in \{1,\ldots,n\}$, and $E\subset \mathbb{H}^n$ be $k$-regular. If $E$ satisfies (P$L$-C) for a constant $L>1$, then it satisfies (N-C) with constants $\Lambda_1=1$ and $\Lambda_2=L$.
%\end{lemma}

%\begin{proof}  Assume that $E$ satisfies (P$L$-C) for a constant $L>1$. Then it satisfies the conditions of (N-C) where $(\R^k,\|\cdot\|_\mathcal S)$ is $V_S$ (as in Definition \ref{d:P-C}) equipped with the induced distance, $\varphi_\mathcal S$ is the projection map onto $V_\mathcal S$,
%$\Lambda_1=1$ and $\Lambda_2=L$.
%\end{proof}

\section{An elaboration on Juillet's construction}\label{s:Juillet}
The goal of this section is to prove Theorem \ref{t:JuilletCurveInNotIn} by using a construction due to Juillet.
%, which will be achieved in Theorems \ref{gammaglem} and \ref{t:JuilletCurveGLem}.
In \cite{Juillet}, the author constructed a Lipschitz curve $\omega:[0,1]\to\H^1$ whose associated $\beta_{\infty,\mathcal V_1}$-numbers are not square-summable, providing in this way a counterexample to the classical traveling salesman problem (with exponent 2) in $\H^1$, which shows that the condition formulated in \cite{MR2371434} for a compact set $E\subset\H^1$ to be contained in a rectifiable curve is not necessary. The curve $\omega$ is obtained as a horizontal lift of a Lipschitz curve $\omega^\C:[0,1]\to\R^2$, whose construction requires a uniform approximation by polygonal curves. Let $\Gamma$ denote the support of the curve $\omega$. It was observed in \cite[Proposition 4.37]{FV2} that $\Gamma\not\in{\rm GLem}(\beta_{\infty,\mathcal V_1},2)$, but this does not a priori rule out the possibility that $\Gamma$ could satisfy
${\rm GLem}(\beta_{p,\mathcal V_1},2)$
with some exponent $p<\infty$. In order to show that the assumptions of our main result, Theorem \ref{t:ImprovedHahlomaa}, are strictly weaker than the ones used by Hahlomaa in \cite{Hahlomaa}, we aim to prove that $\Gamma$ is a $1$-regular set in $(\H^1, d)$ (Theorem \ref{injective}) with $\Gamma\not\in{\rm GLem}(\beta_{1,\mathcal V_1},2)$ (Theorem \ref{gammaglem}), even if $\Gamma\in {\rm GLem}(\widehat\beta_{1,\mathcal V_1},4)$ (Theorem \ref{t:JuilletCurveGLem}). The converse implication, 
${\rm GLem}(\beta_{1,\mathcal V_1},2){\,\Rightarrow\rm GLem}(\widehat\beta_{1,\mathcal V_1},4)$, 
always holds by Proposition \ref{p:FromHorizToStartif}.

As in the rest of the paper, we work here with the Kor\'{a}nyi distance $d$. In \cite{Juillet}, the sub-Riemannian distance is used instead, but this change is insignificant due to the bi-Lipschitz equivalence of the two metrics.

We recall the following definition, which can be given in any metric space.
\begin{definition}
 Let $\gamma:[a,b]\to\H^1$ be a (continuous) curve. If $d$ denotes the Korányi distance, the \textit{length} of $\gamma$ in $(\H^1,d)$ is defined as
 \[\ell_d(\gamma):=\sup_{k\in\N}\left\{\sum_{i=1}^kd(\gamma(s_{i-1}),\gamma(s_i)):a=s_0<s_1<\dots<s_k=b\right\}.\]
\end{definition}
If  $\gamma$ is a horizontal curve, then its length $\ell_d(\gamma)$ coincides with the \textit{Euclidean} length of the projection $\pi\circ\gamma$ in $\R^2$ (see \cite[Proposition 1.1]{MR3417082} for the analogous statement with the sub-Riemannian distance $d_{cc}$, and \cite[below Prop. 6.2]{MR3739202}, \cite[Lemma 5.1]{MR4062792} or \cite[Proposition 2.1]{MR2738530} for the inequality $d\leq d_{cc}$).

We now recall the main steps in the construction of the curve $\omega$ with trace $\Gamma$ in Theorem \ref{t:JuilletCurveInNotIn}, since we will use them in our argument (see \cite{Juillet} for more details about the construction).
Let 
\begin{equation}\label{eq:theta_n}
\theta_n:=\frac{C_0}{n}
\end{equation}
be a sequence of angles, where the constant $C_0$ is chosen small enough (we will assume $C_0\leq 0.2$ as in the original construction and point out all the instances where additional requirements for $C_0$ appear). We start by defining a sequence $(\omega^\C_n)_{n\in\N}$ of planar curves in $\R^2$. The curve $\omega_0^\C:[0,1]\to \mathbb{R}^2$ is simply the constant speed parameterization of the segment connecting the points $(-1,0)$ and $(1,0)$. Each curve $\omega_n^\C:[0,1]\to \mathbb{R}^2$ is a polygonal line formed by $4^n$ segments of the same length $l_n$, parameterized with constant speed $L_n = 4^n \cdot l_n$. Inductively, the curve $\omega_{n+1}^\C$ is obtained from $\omega_{n}^\C$ by replacing each segment line in $\omega_n^\C$ by a polygonal curve (keeping the same starting and ending points) made of four new segments of length $l_{n+1}=\frac{l_n}{4\cos\theta_{n+1}}$, forming an angle $\theta_{n+1}$ with the original one.
\begin{figure}[h]
        \centering
        \includegraphics[width=17 cm, height=5 cm]{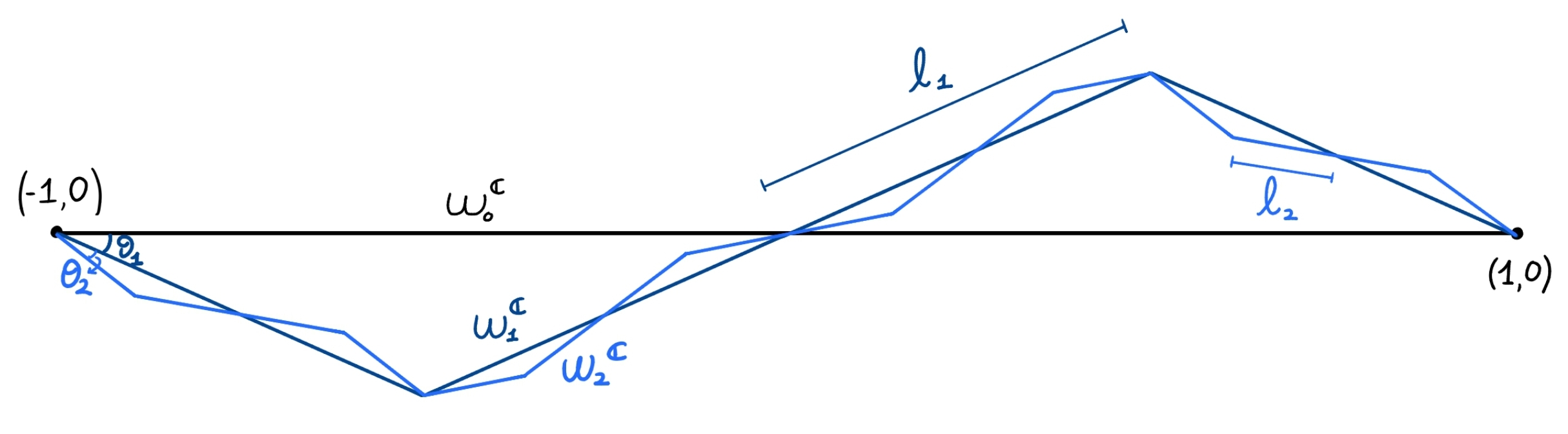}\caption{The construction of the curve $\omega$}
    \end{figure}
    
\noindent Notice that, by construction, 
\begin{equation}\label{eq:**}
    2\cdot4^{-n}\leq l_n\leq 2^{-n+1}\text{ for every }n\geq 0.
\end{equation}
It is proved in \cite[Proposition 3.1 and Lemma 3.2]{Juillet} that the sequence of curves $(\omega_n^\C)_n$ uniformly converges to a Lipschitz curve $\omega^\C$ of length $L\leq 2.4$, again parameterized with constant speed. 
Actually, the proof of \cite[Lemma 3.2]{Juillet} shows that
\begin{equation}\label{eq:*}
    L \leq 2 \cdot \exp(0.08)\leq 2.2.
\end{equation}
Moreover \begin{equation}\label{limitlength}L=\lim_{n\to\infty} 4^nl_n= \sup_n\, 4^nl_n,\end{equation} where $4^nl_n$ represents the length of the curve $\omega_n^\C$. \medskip\\Finally, let $\omega:[0,1]\to\H^1$ be the horizontal lift of the curve $\omega^\C$ starting from $(-1,0,0) $. We denote by $\omega_n$ the horizontal lift of $\omega_n^\C$ starting from the same point. Notice that the Heisenberg length of the curve $\omega$ is still $L$, while  the one of $\omega_n$ is $4^n l_n$. By \cite[Lemma 3.3]{Juillet}, for every $n\leq m$, the curves $\omega_n$ and $\omega_{m}$ coincide at each parameter point of the form $\frac{\sigma}{4^n}$ for $\sigma=0,1,\dots,4^n$: this is easy to see for the curves $\omega_n^\C$ and $\omega_m^\C$, while it holds also for the lifted curves since the signed area spanned by $\omega_n^\C$ and $\omega_m^\C$ on an interval $[\frac{\sigma}{4^n},\frac{\sigma+1}{4^n}]$ is the same; recall the comment below Definition \ref{d:horiz}.
Moreover, $\omega_n\to\omega$ uniformly on $[0,1]$. This is not explicitly stated in \cite{Juillet}, but it follows by 
the Arzel\`{a}–Ascoli theorem 
from the uniform convergence  $\omega_n^\C\to\omega^\C$, the horizontality of $\omega_n$ and the equi-boundedness of the derivatives, $|\dot{\omega}_n^\C|\leq L_n\leq L$.
Denoting by $\Gamma$ the support of $\omega$, our main goal is now to prove that $\Gamma\not\in{\rm GLem}(\beta_{1,\mathcal V_1},2)$, see Theorem \ref{gammaglem}.
%,  even if $\Gamma\in {\rm Glem}(\widehat\beta_{1,\mathcal V_1},4)$, see Theorem \ref{t:JuilletCurveGLem}. 
\medskip\\
We first show that $\omega$ is injective: the same argument is also used to prove 1-regularity of $\Gamma=\omega([0,1])$ (with respect to the Korányi distance in $\H^1$):
\begin{theorem}\label{injective}
    The curve $\omega^\C$ (hence also  $\omega$) is injective, and $\Gamma$ is 1-regular in $(\mathbb{H}^1,d)$.
\end{theorem}
\begin{proof}
The proof of injectivity relies on the following claim:\medskip\\
  \textbf{Claim}. Let $p=\omega^\C(\sigma/4^n)$ for some $\sigma=0,1,\dots, 4^n$. Then
  \begin{equation}\label{eq:claim} B_{\rm Eucl}(p,0.6\, l_n)\,\cap\,\omega^\C\left([0,1]\setminus\left[\frac{\sigma-1}{4^n},\frac{\sigma +1}{4^n}\right]\right)=\emptyset. 
  \end{equation}
We now assume the claim and prove that $\omega^\C$ is injective. Let $0\leq a<b\leq 1$ and choose $n\in\N$, $\sigma=0,1,\dots, 4^n$ such that 
\begin{equation}\label{distancefroma}
    \left|a-\frac{\sigma}{4^n}\right|\leq \frac{1}{2\cdot4^n}
\end{equation}
and 
\begin{equation}\label{distancefromb}
   \left|b-\frac{\sigma}{4^n}\right|> \frac{1}{4^n}. 
\end{equation}
This is always possible, since one can choose $n$ large enough such that
\[|a-b|> \frac{3}{2\cdot4^n}\] and subsequently choose $\sigma=0,1,\dots,4^n$ such that \eqref{distancefroma} holds. By the triangular inequality, \eqref{distancefromb} holds as well.
From \eqref{distancefroma} it follows that 
\begin{equation}\label{eq:ainball}\left|\omega^\C(a)-\omega^\C\left(\frac{\sigma}{4^n}\right)\right|\leq L\left|a-\frac{\sigma}{4^n}\right|\leq L\,\frac{1}{2\cdot 4^n}\overset{\eqref{eq:*}}{\leq} \frac{2.4}{2\cdot 4^n}=0.6\frac{2}{4^n}\overset{\eqref{eq:**}}{\leq} 0.6 \,l_n.\end{equation} Hence $\omega^\C(a)\in B_{\rm Eucl}(p,0.6\, l_n)$. On the other hand,  \eqref{distancefromb} means that $b\in[0,1]\setminus [\frac{\sigma-1}{4^n},\frac{\sigma +1}{4^n}]$. Therefore, the claim implies that $\omega^\C(b)\not\in B_{\rm Eucl}(p,0.6\, l_n)$ and in particular we deduce $\omega^\C(a)\neq \omega^\C(b)$, proving injectivity of $\omega^\C$ (and clearly of $\omega$).\medskip\\
 Assuming the claim, we now show that $\omega([0,1])$ is 1-regular. Fix now $x=(z,t)\in \omega([0,1])$ and $r>0$. In order to prove upper 1-regularity, we can assume that $r<1/20$, since otherwise we use the simple estimate
  \[\mathcal H^1\left(B(x,r)\,\cap\,\omega\left([0,1]\right)\right)\leq  L\leq 20 L\,r.\]
For $r<1/20$,  choose $n\in\N$ such that $l_{n+1}/40\leq r<l_n/40$.
  By $L$-Lipschitz continuity of $\omega^\C$, there exists $\sigma=0,1,\dots 4^n$ such that $|z-p|\leq \frac{L}{2\cdot 4^n}$, where $p=\omega^\C(\sigma/4^n)$. In particular, %\textcolor{blue}{(Here explain why we can assume $L\leq 2.2$ instead of 2.4. Maybe already in the construction of the curve)}
  \[|z-p|\leq \frac{L}{2\cdot 4^n}\overset{\eqref{eq:*}}{\leq} \frac{2.2}{2\cdot 4^n}=\frac{11}{20}\frac{2}{4^n}\overset{\eqref{eq:**}}{\leq} \frac{11}{20}l_n.\]
  Hence, if $z'\in B_{\rm Eucl}(z,r)$, then
  \[|z'-p|\leq |z'-z|+|z-p|\leq r+\tfrac{11}{20} l_n\leq \left(\tfrac{1}{40}+\tfrac{11}{20}\right)l_n\leq 0.6\, l_n.\]
  This proves that $B_{\rm Eucl}(z,r)\subseteq B_{\rm Eucl}(p,0.6\,l_n)$. Moreover, \eqref{eq:claim} implies that
  \begin{equation}\label{eq:BallCurveIncl}
B_{\rm Eucl}(p,0.6\, l_n)\,\cap\,\omega^\C\left([0,1]\right)\subseteq \omega^\C\left(\left[\frac{\sigma-1}{4^n},\frac{\sigma +1}{4^n}\right]\right).
  \end{equation}
  Altogether, this shows that
  \begin{align*}
      \pi\left(B(x,r)\right)\cap \pi\left(\omega([0,1])\right)
     & = \pi\left(B(x,r)\right)\cap \omega^{\mathbb{C}}\left([0,1] \right)=B_{\mathrm{Eucl}}(z,r)\cap \omega^{\mathbb{C}}\left([0,1]\right)\\
  &{=}  B_{\rm Eucl}(p,0.6\,l_n)\cap \omega^{\mathbb{C}}\left([0,1]\right)
 \overset{\eqref{eq:BallCurveIncl}}{=} \omega^\C\left(\left[\frac{\sigma-1}{4^n},\frac{\sigma +1}{4^n}\right]\right).
  \end{align*}
Therefore, $\pi\left(B(x,r)\cap \omega([0,1])\right)\subset \omega^\C\left(\left[\frac{\sigma-1}{4^n},\frac{\sigma +1}{4^n}\right]\right)$. The injectivity of $\omega^{\mathbb{C}}$  implies then that also
\begin{equation}\label{eq:CovByCurv}
    B(x,r)\cap \omega([0,1]) \subset \omega\left(\left[\frac{\sigma-1}{4^n},\frac{\sigma +1}{4^n}\right]\right).
\end{equation}
Moreover,
by the $L$-Lipschitz continuity of $\omega^\C$ and the same estimates as in \eqref{eq:ainball}, we have \begin{equation}\label{1regularity}\mathcal{H}^1_{\mathrm{Eucl}}\left(\omega^\C\left(\left[\frac{\sigma-1}{4^n},\frac{\sigma +1}{4^n}\right]\right)\right)\leq L\,\frac{2}{4^n}\overset{\eqref{eq:**},\eqref{eq:*}}{\leq} 2.4\,l_n.\end{equation}
The upper 1-regularity of $\omega([0,1])$ then follows since, by \eqref{eq:CovByCurv} and the horizontality of the curve $\omega$,
  \[\begin{split}\mathcal H^1(B(x,r)\cap\, \omega([0,1]))\leq \mathcal H^1\left(\omega\left(\left[\frac{\sigma-1}{4^n},\frac{\sigma +1}{4^n}\right]\right)\right)&= \mathcal{H}^1_{\mathrm{Eucl}}\left(\omega^{\mathbb{C}}\left(\left[\frac{\sigma-1}{4^n},\frac{\sigma +1}{4^n}\right]\right)\right)\\&\overset{\eqref{1regularity}}{\leq} 2.4\cdot 4\cdot 40 \,r.\end{split}\]
 
  The lower regularity of $\omega([0,1])$ is immediate since it is a  connected set.
\end{proof}
\begin{proof}[Proof of the Claim \eqref{eq:claim} ]
   We prove the claim by induction on $n$. Let us first show the conclusion for $n=1$.  Since $\omega^\C$ is $L$-Lipschitz, then $\omega^\C\left([0,1]\setminus\left[\frac{\sigma-1}{4^n},\frac{\sigma +1}{4^n}\right]\right)$ lies in the $\frac{L}{16}$-neighbourhood of \[E:=\left\{\omega^\C\left(\frac{\tau}{16}\right): \tau=0,\dots,4\sigma-4\,\text{ or }\tau=4\sigma+4,\dots, 16\right\}.\]
   Hence, if $q\in \omega^\C\left([0,1]\setminus\left[\frac{\sigma-1}{4^n},\frac{\sigma +1}{4^n}\right]\right)$, then there exists $\bar q\in E$ such that $|q-\bar q|\leq \frac{L}{16}$. Then
   \begin{equation}\label{qnotinball}\left|q-\omega^\C\left(\frac{\sigma}{4}\right)\right|\geq\left|\bar q-\omega^\C\left(\frac{\sigma}{4}\right)\right|-\left|q-\bar q\right|\geq l_1- \frac{L}{16}\geq l_1-\frac{2.4}{16}.\end{equation}  Here we exploit the fact that $\left|\bar q-\omega^\C\left(\frac{\sigma}{4}\right)\right|\geq l_1$: this follows if the constant $C_0$ (and hence the angle $\theta_1$) in \eqref{eq:theta_n} is chosen small enough, so that if $0\leq t_1<t_2<t_3\leq 16$ are natural numbers, then
   \begin{align}\label{eq:Ordner}
   \left|\omega^\C\left(\frac{t_1}{16}\right)-\omega^\C\left(\frac{t_2}{16}\right)\right|
   &\leq
   (t_2-t_1)  \ell_2\\
&\overset{\eqref{eq:theta_n}}{\leq} \left(1 + t_2-t_1\right)\cdot \cos (\theta_1+\theta_2) \ell_2\notag\\
   &\leq (t_3-t_1) \cdot  \cos(\theta_1+\theta_2)\ell_2
   \leq \left|\omega^\C\left(\frac{t_1}{16}\right)-\omega^\C\left(\frac{t_3}{16}\right)\right|.\notag
   \end{align}
   The last inequality holds true since, if $\theta_1$ and $\theta_2$ are chosen small enough, the projection of $\omega^{\mathbb{C}}_2$ to the segment between $\omega^{\mathbb{C}}(0)$ and  $\omega^{\mathbb{C}}(1)$
   preserves the order in the following sense:
   if $\tau<\tau'$, then the first coordinates satisfy $(\omega_2)_1^{\mathbb{C}}(\tau/16)<(\omega_2)_1^{\mathbb{C}}(\tau'/16)$, and moreover, the individual segments in $\omega^{\mathbb{C}}_2$ project to intervals of length at least $ \cos(\theta_1+\theta_2)\ell_2$.

   %notice first that the points $\omega^{\mathbb{C}}(0)$, $\omega^{\mathbb{C}}(\frac{1}{4})$, $\omega^{\mathbb{C}}(\frac{1}{2})$, $\omega^{\mathbb{C}}(\frac{3}{4})$, $\omega^{\mathbb{C}}(1)$ remain in this order when projected to the line segment between $\omega^{\mathbb{C}}(0)$ and $\omega^{\mathbb{C}}(1)$.
   
   Since $\theta_n\leq \theta_1$, an analogous conclusion as in \eqref{eq:Ordner} also holds for ``packages" of future generations, namely $\left|\omega^\C\left(\frac{t_1}{4^n}\right)-\omega^\C\left(\frac{t_2}{4^n}\right)\right|\leq \left|\omega^\C\left(\frac{t_1}{4^n}\right)-\omega^\C\left(\frac{t_3}{4^n}\right)\right|$ for every $n\geq 2$ and $16\,\rho\leq t_1<t_2<t_3\leq 16\,(\rho+1) $ for $  \rho=0,1,\dots, 4^{n-2}-1$.
   Moving back to \eqref{qnotinball}, since by construction $l_1\geq \frac{1}{2}$, it follows $l_1-\frac{2.4}{16}\geq \,0.7 l_1>0.6 \,l_1$. We deduce that $q\not\in B_{\rm Eucl}(\omega^\C\left(\sigma/4\right), 0.6\,l_1)$, as desired.\medskip\\
   We now deal with the inductive step. Let $p=\omega^\C(\sigma/4^n)$ for some $\sigma=0,1,\dots, 4^n$. We distinguish two cases, according to the value of $\sigma$. If $\sigma=4\sigma_0$ for some $\sigma_0=0,1,\dots, 4^{n-1}$, we can write $p=\omega^\C(\sigma_0/4^{n-1})$. By induction hypothesis, since clearly $B_{\rm Eucl}(p,0.6\,l_n)\subset B_{\rm Eucl}(p, 0.6 \,l_{n-1})$, we deduce that
   \[B_{\rm Eucl}(p,0.6\,l_n)\cap \omega^\C\left([0,1]\setminus\left[\frac{\sigma_0-1}{4^{n-1}},\frac{\sigma_0 +1}{4^{n-1}}\right]\right)=\emptyset.\]
   Notice that $\left[\frac{\sigma_0-1}{4^{n-1}},\frac{\sigma_0 +1}{4^{n-1}}\right]=\left[\frac{\sigma-4}{4^n},\frac{\sigma+4}{4^n}\right]$: hence we only need to show that \[\begin{cases} B_{\rm Eucl}(p,0.6\,l_n)\cap \omega^\C\left(\left[\frac{\sigma-4}{4^n},\frac{\sigma-1}{4^n}\right]\right)=\emptyset \\ B_{\rm Eucl}(p,0.6\,l_n)\cap \omega^\C\left(\left[\frac{\sigma+1}{4^n},\frac{\sigma+4}{4^n}\right]\right)=\emptyset.\end{cases}\]
 These properties can be proved exactly with the same argument as in the $n=1$ case, replacing $l_1$ with $l_n$ and 16 with $4^{n+1}$, using the relation $l_n\geq 2/4^n$ from \eqref{eq:**}.\medskip\\
 If instead $\sigma\in \{0,1,\ldots,4^n\}$ in $p=\omega^\C(\sigma/4^n)$  is not an integer multiple of 4, choose $\sigma_0=0,1,\dots, 4^{n-1}$ such that $\frac{\sigma_0}{4^{n-1}}<\frac{\sigma}{4^n}<\frac{\sigma_0+1}{4^{n-1}}.$
 In particular $\sigma\in\{4\sigma_0+1, 4\sigma_0+2, 4\sigma_0+3\}$. Our goal is now to prove
 \begin{equation}\label{intermediategoal}
     B_{\rm Eucl}(p,0.6\,l_n)\cap \omega^\C\left([0,1]\setminus\left[\frac{\sigma_0}{4^{n-1}},\frac{\sigma_0 +1}{4^{n-1}}\right]\right)=\emptyset.
 \end{equation}
 Once this is achieved, we reach the conclusion by arguing again as in the previous case.  %Hence it remains to show \eqref{intermediategoal}. 
 Let $p_0:=\omega^\C(\frac{\sigma_0}{{4^{n-1}}})$ and $p_1:=\omega^\C(\frac{\sigma_0+1}{4^{n-1}})$. We first prove that
 \begin{equation}\label{intermediateclaim2} B_{\rm Eucl}(p,0.6\,l_n)\subseteq  B_{\rm Eucl}(p_0,0.6\,l_{n-1})\cup B_{\rm Eucl}(p_1,0.6\,l_{n-1}). \end{equation}
Assume for instance $\sigma=4\sigma_0+1$. Then, if $q\in B_{\rm Eucl}(p, 0.6\,l_n)$, 
\[|q-p_0|\leq |q-p|+|p-p_0|\leq 0.6\,l_n+l_n=1.6\,\frac{l_{n-1}}{4\cos \theta_n}<0.6\,l_{n-1}.\]
 If instead $\sigma=4\sigma_0+3$, then we argue in the same way considering $p_1$ instead of $p_0$. Hence it remains to discuss the case $\sigma=4\sigma_0+2$. In this situation, $|p_0-p|=|p_1-p|=l_{n-1}/2$. Consider $q\in B_{\rm Eucl}(p, 0.6\,l_n)$ and assume for instance that $|q-p_0|\leq |q-p_1|$ (the other case is analogous). By Pythagorean theorem 
 \[|q-p_0|\leq\sqrt{ (0.5 \,l_{n-1})^2+(0.6\,l_n)^2}=\left(0.25+\frac{0.36}{16\cos^2\theta_n}\right)^\frac{1}{2}l_{n-1}\leq 0.6\,l_{n-1},\]
proving \eqref{intermediateclaim2}. Let us now conclude by showing \eqref{intermediategoal}.\medskip\\
Let $q\in B_{\rm Eucl}(p, 0.6\,l_n)$. By \eqref{intermediateclaim2}, $q\in B_{\rm Eucl}(p_i, 0.6\,l_{n-1})$ for some $i\in\{0,1\}$. Let us assume without loss of generality that $i=0$. Then, by inductive hypothesis, $q\not\in \omega^\C([0,1]\setminus\left[\frac{\sigma_0-1}{4^{n-1}},\frac{\sigma_0+1}{4^{n-1}}\right])$. In order to establish \eqref{intermediategoal}, we need then to prove that $q\not\in \omega^\C\left(\left[\frac{\sigma_0-1}{4^{n-1}},\frac{\sigma_0}{4^{n-1}}\right]\right)$. 
This is achieved using a similar argument as before: since $\omega^\C$ is $L$-Lipschitz, $\omega^\C\left(\left[\frac{\sigma_0-1}{4^{n-1}},\frac{\sigma_0}{4^{n-1}}\right]\right)$ lies in the $\frac{L}{4^{n+1}}$-neighbourhood of 
\[E_0:=\left\{\omega^\C\left(\frac{16\sigma_0-16}{4^{n+1}}\right),\omega^\C\left(\frac{16\sigma_0-15}{4^{n+1}}\right),\dots,\omega^\C\left(\frac{16\sigma_0}{4^{n+1}}\right)\right\}.\]
Assume by contradiction that $q=\omega^\C(s)$ for some $s\in \left[\frac{\sigma_0-1}{4^{n-1}},\frac{\sigma_0}{4^{n-1}}\right]$. Then there exists $\bar q\in E_0$ such that $|q-\bar q|\leq \frac{L}{4^{n+1}}$. Hence
\begin{equation}\label{eq:ball_est}
|q-p|\geq |p-\bar q|-|q-\bar q|\geq l_n-\frac{L}{4^{n+1}}\geq l_n-\frac{2.4}{4^{n+1}}.
\end{equation}
The fact that $|p-\bar q|\geq l_n$ follows by the analogue observation after \eqref{qnotinball}, if $C_0$ is chosen small enough. However, while we previously had to consider only the points in one ``package'', we now need the same conclusion for two consecutive ``packages''. To this end,  
we also use the fact that the angle between  two consecutive segments of $\omega_{n-1}^\C$ is always greater than $\pi-2\theta_1$, which can be made very close to $\pi$ by choosing $C_0$ small enough; recall \eqref{eq:theta_n}.
\\ Since $l_n\geq 2\cdot 4^{-n}$, it follows that $l_n-\frac{2.4}{4^{n+1}}\geq 0.7\, l_n>0.6\,l_n$, which implies by \eqref{eq:ball_est} that $q\not\in B_{\rm Eucl}(p, 0.6\,l_n)$, a contradiction.
\end{proof}

We next aim to show that $\Gamma=\omega([0,1])$ is so badly approximable by horizontal lines that it does not satisfy  ${\rm GLem}(\beta_{1,\mathcal V_1},2)$. To this end, we will use the following standard estimate for the distance between two points in $\H^1$ in terms of the Euclidean distance between their projections and the area spanned by any horizontal curve connecting them. 
\begin{lemma}\label{distance}
    Let $\gamma:[a,b]\to\H^1$ be a horizontal curve joining $p_1=\gamma(a)$ and $p_2=\gamma(b)$. Then 
    \begin{equation}\label{eq:DistEst}
    d(p_1,p_2)\lesssim |\pi(p_1)-\pi(p_2)|+ \sqrt{|A_\gamma|},\end{equation}
    where $A_\gamma$ denotes the algebraic area enclosed by the curve $\pi\circ\gamma$ concatenated with the segment joining $\pi(p_2)$ and $\pi(p_1)$. 
\end{lemma}
\begin{proof}
As both sides of the inequality \eqref{eq:DistEst} are invariant by left translations, we may assume without loss of generality that $p_1=\gamma(a)=0$. Since
\begin{displaymath}
    d(\gamma(b),0)\lesssim |\pi(\gamma(b))|+\sqrt{|\gamma_3(b)|}
\end{displaymath}
and, by horizontality of $\gamma=(\gamma_1,\gamma_2,\gamma_3)$, we have $\dot{\gamma}_3=\frac{1}{2}(\gamma_1\dot{\gamma}_2-\gamma_2 \dot{\gamma}_1)$ almost everywhere, the claim follows, cf. \cite[(2.22)] {CDPT} or \cite[(1.1)]{Juillet}.
\end{proof}

\begin{theorem}\label{gammaglem}
    $\Gamma\not\in{\rm GLem}(\beta_{1,\mathcal V_1},2)$.
\end{theorem}

\begin{proof}
 We will show that 
 \begin{equation}\label{eq:B_infty}
     B_{1,1}^\Delta(\Gamma):=\sum_{k\in\mathbb Z}2^{-k}\sum_{x\in\Delta_k}\beta_{1,\mathcal V_1}(x, A\,2^{-k})^2=+\infty,
 \end{equation}
where $\Delta=(\Delta_k)_k$ is a dyadic net and $B_{1,1}^{\Delta}$ is defined as in \eqref{eq:BDef} with a constant $A\geq 5$, cf.\ \cite[(0.1)]{Juillet}. Then the conclusion follows from Lemma \ref{Glemvsmultires}. Notice that the index $k$ is chosen to uniformize the notation with the one used in \cite{Juillet} and should not be confused with the dimension of $\Gamma$, which is 1.
In order to verify \eqref{eq:B_infty}, we need suitable lower bounds for $\mathrm{card}(\Delta_k)$ and $\beta_{1,\mathcal V_1}(x, A\,2^{-k})$, $x\in \Delta_k$, respectively.

 It is proved in \cite[\S 4]{Juillet} that $\mathrm{card}(\Delta_k)\geq 2^k$, $k\in \mathbb{N}$.
 Moreover, in \cite[(4.1)]{Juillet}, a lower bound for $\beta_{\infty,\mathcal V_1}(x, A\,2^{-k})$ is derived. We need to verify an analogous bound for $\beta_{1,\mathcal V_1}(x, A\,2^{-k})$, which requires a finer analysis: we have to show that  $\Gamma\cap B(x, A\,2^{-k})$ stays at large distance from every horizontal line not only in a point, but in a large $\mathcal{H}^1$-measure set of points.
 
 In \cite[Section 4]{Juillet} it is shown
 that, taking $n=\lceil k/2\rceil$, for every $x\in\Delta_k$ there exists $\sigma\in\{0,1,\dots, 4^n-1\}$ such that\begin{equation}\label{pieceofcurve}\omega\left(\left[\frac{\sigma}{4^n},\frac{\sigma+1}{4^n}\right]\right)\subseteq B(x,A\,2^{-k}).\end{equation}
 Nevertheless, it also holds that 
% \[\mathcal H^1_d(\Gamma\cap B(x, A\,2^{-k}))\lesssim\mathcal H^1_d(\omega([\sigma/4^n,(\sigma+1)/4^n])).\]
 \begin{equation}\label{control}\mathcal H^1(\Gamma\cap B(x, A\,2^{-k}))\lesssim\mathcal H^1\left(\omega\left(\left[\frac{\sigma}{4^n},\frac{\sigma+1}{4^n}\right]\right)\right).\end{equation}
 Indeed, on the one hand, $\mathcal H^1(\Gamma\cap B(x, A\,2^{-k}))\lesssim 2^{-k}$ by 1-regularity of $\Gamma$ established in Theorem \ref{injective}. On the other, since $\pi:(\H^1,d)\to(\R^2,d_e)$, $\pi(x,y,t):=(x,y)$, is a 1-Lipschitz map, by construction of $\omega$ we deduce\[\begin{split}\mathcal H^1\left(\omega\left(\left[\frac{\sigma}{4^n},\frac{\sigma+1}{4^n}\right]\right)\right)&\geq \mathcal H^1_{\mathrm{Eucl}}\left(\pi\left(\omega\left(\left[\frac{\sigma}{4^n},\frac{\sigma+1}{4^n}\right]\right)\right)\right)\\&=\mathcal H^1_{\mathrm{Eucl}}\left(\omega^\C\left(\left[\frac{\sigma}{4^n},\frac{\sigma+1}{4^n}\right]\right)\right)\\&= \frac{1}{4^n}\mathcal H^1_{\mathrm{Eucl}}\left(\omega^\C\left([0,1]\right)\right)\sim 2^{-k}.\end{split}\]
 Hence \eqref{control} follows. Recall the construction 
 of the curve described after \eqref{eq:theta_n}.
 Proceeding as in \cite{Juillet}, we rescale $\omega\left(\left[\frac{\sigma}{4^n},\frac{\sigma+1}{4^n}\right]\right)$ using the similitudes of $\H^1$ (see \cite[Section 1.3]{Juillet}) and we obtain another horizontal curve $\Gamma_n$ which passes through the subset $\Lambda_{\theta_{n+1}}$ given by
 {\small\begin{equation}\label{eq:Lambda5Pts}\left\{(-1,0,0), \left(-\frac{1}{2},-\frac{\tan \theta_{n+1}}{2},\frac{\tan \theta_{n+1}}{4}\right), \left(0,0,\frac{\tan \theta_{n+1}}{4}\right), \left(\frac{1}{2},\frac{\tan \theta_{n+1}}{2},\frac{\tan \theta_{n+1}}{4}\right), (1,0,0)\right\}\end{equation}}and coincides with the curve obtained as $\Gamma$, but using the sequence of angles $(\theta_{n+m})_{m\geq 1}$. The scaling factor between $\Gamma_n$ and $\omega\left(\left[\frac{\sigma}{4^n},\frac{\sigma+1}{4^n}\right]\right)$ is $l_n/2$. Let us denote $\theta\equiv\theta_{n+1}$ and let
 %\[d_\theta^1(\ell):=\fint_{\Gamma_n} d(y,\ell)\,d\mathcal H^1_d(y)\qquad\ell\in\mathcal V_1;\]
 \[D_\theta^1:=\inf_{\ell\in\mathcal V_1} \fint_{\Gamma_n} d(y,\ell)\,d\mathcal H^1(y).\]
 Easy computations show that
 \[\begin{split}\beta^{\,\Gamma}_{1,\mathcal V_1}(x, A\,2^{-k})&\sim\frac{1}{2^{-k}}\inf_{\ell\in \mathcal V_1}\fint_{\Gamma\cap B(x, A\,2^{-k})}d(y,\ell)\,d\mathcal H^1(y)\\ &\overset{\eqref{pieceofcurve},\eqref{control}}{\gtrsim}\frac{1}{2^{-k}}\inf_{\ell\in \mathcal V_1}\fint_{\omega\left(\left[\frac{\sigma}{4^n},\frac{\sigma+1}{4^n}\right]\right)}d(y,\ell)\,d\mathcal H^1(y)\\ &=\frac{1}{2^{-k}}\frac{l_n}{2} D_\theta^1 \gtrsim D_\theta^1.\end{split}\]
 where we used the 1-regularity of $\Gamma$, \eqref{pieceofcurve} and \eqref{control}, the invariance of the distance and the fact that $l_n\geq 4^{-n}\gtrsim2^{-k}$, recall \eqref{eq:**}.

 In order to conclude \eqref{eq:B_infty}, we need a large enough lower bound for $D^1_{\theta}$, which is the content of the following claim.
 \medskip

\noindent \textbf{Claim 1:} $D_\theta^1\gtrsim\sqrt\theta$.\\
 Assuming the validity of Claim 1, we now conclude that
 \[B_{1,1}^\Delta(\Gamma):=\sum_{k\in\mathbb Z}2^{-k}\sum_{x\in\Delta_k}\beta_{1,\mathcal V_1}(x, A\,2^{-k})^2\gtrsim\sum_{k\in\mathbb N}2^{-k}\,2^k\,(D_\theta^1)^2\gtrsim \sum_{k\in\mathbb N}\theta_{\lceil k/2\rceil+1}\overset{\eqref{eq:theta_n}}{\sim}\sum_{k\in\mathbb N}\tfrac{1}{\lceil k/2\rceil+1}=+\infty.\]
 It remains to prove Claim 1, that can be read as follows: there exists a constant $\widetilde K>0$ (not depending on $n$, and in particular not on  $\theta=\theta_{n+1}$) such that, for all $\ell \in \mathcal V_1$, it holds
 \[\fint_{\Gamma_n} d(y,\ell)\,d\mathcal H^1(y)\geq \widetilde K\sqrt \theta.\]
 By \cite[Proposition 4.1]{Juillet} there exists a constant $K>0$ such that, for any $\ell\in\mathcal V_1$ and any $0<n\in\mathbb N$, there exists $\bar y\in \Lambda_{\theta}$ (for $\Lambda_{\theta}$ as defined in \eqref{eq:Lambda5Pts}) with the property that $d(\bar y,\ell)\geq K\sqrt\theta$.
Let $0<n\in\N$ and $\ell\in \mathcal V_1$. % and find $\bar y\in \Lambda_\theta$ with the previous property. 
We want to show that there exists a subset $N\subseteq \Gamma_n$ satisfying $\mathcal H^1(N)\geq c\,\mathcal H^1(\Gamma _n)$, for some constant $c$ independent of $\theta$ and $\ell$, such that
\begin{equation}\label{eq:PropN}
d(y,\ell)\geq \frac{K}{4}\sqrt{\theta}\quad\text{for every $y\in N$.}\end{equation}
If this is possible, then we reach the conclusion by simple estimates:
\[\begin{split} \fint_{\Gamma_n} d(y,\ell)\,d\mathcal H^1(y)&=\frac{1}{\mathcal H^1(\Gamma_n)}\int_{\Gamma_n} d(y,\ell)\,d\mathcal H^1(y)\geq \frac{1}{\mathcal H^1(\Gamma_n)}\int_{N} d(y,\ell)\,d\mathcal H^1(y)\\ &\geq \frac{c}{\mathcal H^1(N)}\int_{N}\frac{K}{4}\sqrt\theta \,\,d\mathcal H^1(y)=\frac{cK}{4}\sqrt\theta=:\widetilde K\sqrt\theta.
\end{split}\]
Hence, it remains to construct the set $N$ with the previous properties.
\medskip\\ For each $n\in\N$, let $\Gamma_{n,1}$ be the polygonal curve which connects the points in $\Lambda_{\theta}$, in ascending order with respect to the first coordinate of $\H^1\equiv \R^3$ (not to be confused with the rescaled copy $\Gamma_n$ of $\Gamma)$. Notice that $\Gamma_{n,1}$ is a horizontal curve. We show now that there exists a subset $M\subseteq \Gamma_{n,1}$ with $\mathcal H^1(M)\geq c'\mathcal H^1(\Gamma_{n,1})$ and such that \begin{equation}\label{propofM}d(y,\ell)\geq \frac{K}{2}\sqrt\theta\quad\text{for every }y\in M.\end{equation}
Then we will construct the corresponding set $N$ on $\Gamma_n$.\medskip\\
In order to define the set $M\subset \Gamma_{n,1}$, we use ideas from \cite{Juillet}. We denote the points in $\Lambda _\theta$ by $\{A,B,C,D,E\}$, with respect to the same order used above. (We recall that ``$A$'' has also denoted a constant in 
 \eqref{eq:B_infty}, but there should be no risk of confusion in the following).
We distinguish three cases, according to the value of the planar angle $\varphi\in[0,\frac{\pi}{2}]$ between $\ell^{\,\C}$ and the line $B^\C D^\C$, where $\ell^{\mathbb{C}}$, $B^{\mathbb{C}}$, $D^{\mathbb{C}}$ are the images of the fixed horizontal line $\ell$ and the points $B,D\in \mathbb{H}^1$ under the projection $\pi$, see Figure \ref{eq:FigCases}.
\begin{figure}[h]
        \centering
        \includegraphics[width=17.5 cm, height=11 cm]{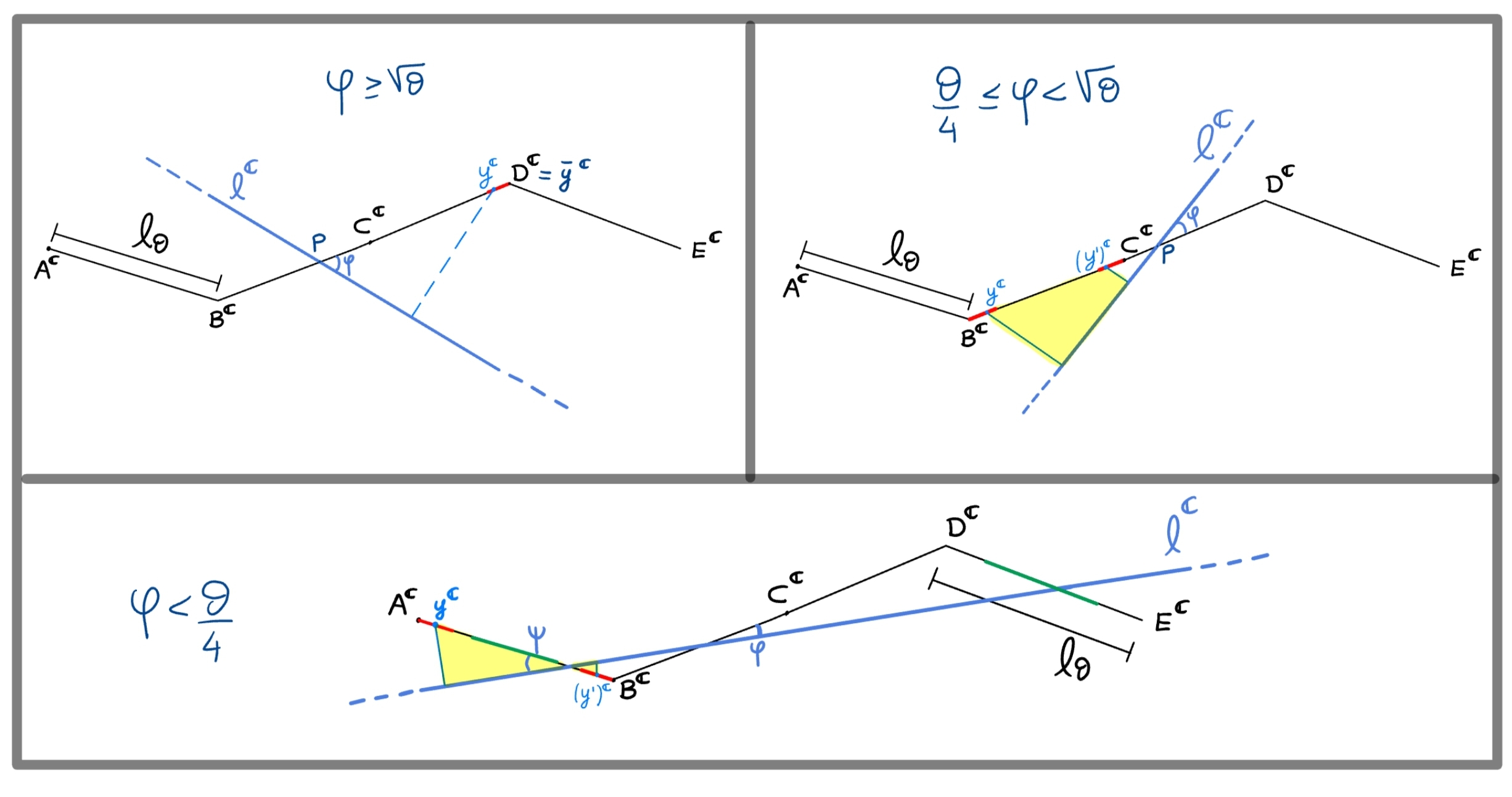}\caption{The three cases described below}\label{eq:FigCases}
    \end{figure}

\begin{itemize}
    \item If $\varphi\geq \sqrt\theta$, then $\ell^{\,\C}$ intersects $B^\C D^\C$ in a single point $P$. Then we can choose
    \[\bar y=\begin{cases}
        D &  \text{if }\ d_{\mathrm{Eucl}}(P, D^\C)\geq d_{\mathrm{Eucl}}(P, B^\C)\\ B & \text{if }\ d_{\mathrm{Eucl}}(P, D^\C)< d_{\mathrm{Eucl}}(P, B^\C)
    \end{cases} \]
    and in this case we let \begin{equation*}\label{eq:DefM}M:=\{y\in C\bar y:d(y,\bar y)\leq l_\theta/10\},\end{equation*}
    where $l_\theta$ is the distance between two consecutive points in $\Lambda_\theta$ (and it is also the Euclidean distance between the corresponding projections on $\R^2$).\\ Notice that $\mathcal H^1(M)=\frac{1}{40}\mathcal H^1(\Gamma_{n,1}).$ 
    Moreover, if $y\in M$, by simple trigonometry
    \[d(y,\ell)\geq d_{\mathrm{Eucl}}(y^\C, \ell^{\,\C})=d_{\mathrm{Eucl}}(P,y^\C)\sin\varphi\geq \frac{1}{2}l_\theta\sin \sqrt\theta,\]
    which, by choice of $K$ in \cite[Proposition 4.1]{Juillet}, is greater than $\frac{1}{2}K\sqrt{\theta}$.
    \item If $\frac{\theta}{4}\leq \varphi<\sqrt\theta$, at least one of the segments $B^\C C^\C$ or $C^\C D^\C$ is not intersected by the line $\ell ^{\,\C}$. Without loss of generality we can assume $\ell^{\,\C}$ does not intersect $B^\C C^\C$. % and, as before, let $P$ be the intersection between $\ell^{\,\C}$ and the line $B^\C D^\C$. 
    For every $y\in BC$ %such that $d(y,B)\leq d(B,C)/10$,
    we let $y'$ to be the unique point on $BC$ such that $d(y,B)=d(y',C)$. This also means that \[d_{\mathrm{Eucl}}(y^\C,B^\C)=d(y,B)=d(y',C)=d_{\mathrm{Eucl}}((y')^{\,\C},C^\C).\]
    Fix now $y\in BC$ such that $d(y,B)\leq l_\theta/10$ and consider the Euclidean trapezoid $T$ in $\R^2$ having as vertices the points $y^\C$, $(y')^\C$ and the corresponding projections onto $\ell^{\,\C}$. %, which we will call $y^\ell$ and $(y')^\ell$.
    Then elementary geometric computations show that
    \[\text{Area}\,(T)\geq\frac{1}{8}l_\theta^2\sin\varphi\cos\varphi,\]
    which gives, arguing as in the proof of Proposition 4.1 in \cite{Juillet} that $\sqrt{\text{Area}\,(T)}\geq \frac{\sqrt \theta}{8\sqrt\pi}.$
    Applying \cite[Lemma 2.5]{Juillet} (in particular see the second line in that proof), by the fact that $y$ and $y'$ lie on the same horizontal line $BC$, we deduce that
    \[d(y,\ell)+d(y',\ell)\geq \sqrt {\text{Area}\,(T)}\geq \frac{1}{8\sqrt\pi}\sqrt\theta.\]
    Therefore there exists a point $p=p(y)\in\{y,y'\}$ such that \begin{equation}\label{eq:p_ell_est}d(p,\ell)\geq \sqrt {\text{Area}\,(T)}\geq \frac{1}{16\sqrt\pi}\sqrt\theta\geq \frac{K}{2}\sqrt\theta,\end{equation}
    by choosing $K\leq \frac{1}{8\sqrt{\pi}}$.
    If we define 
    \[M:=\{p(y):y\in BC,\, d(y,B)\leq l_\theta/10\},\]
    then $\mathcal H^1(M)=\frac{1}{10}l_\theta=\frac{1}{40}\mathcal H^1(\Gamma_{n,1})$ and, by  \eqref{eq:p_ell_est}, $d(p,\ell)\geq \frac{K}{2}\sqrt\theta$ for every $p\in M$.
    \item If $\varphi<\frac{\theta}{4}$, then, by \cite[Lemma 4.2]{Juillet}, the line $\ell^{\,\C}$ cannot intersect both the central segments of $A^\C B^\C$ and $D^\C E^\C$. Assume for instance that $\ell^{\,\C}$ does not intersect the central segment of $A^\C B^\C$. Let $\psi$ be the positive angle between $\ell^{\,\C}$ and the line $A^\C B^\C$: it has been proved by Juillet that $\frac{7\theta}{4}\leq \psi\leq \frac{\pi}{4}$. Let $P$ be the intersection between the line $A^\C B^\C$ and $\ell^\C$ and assume for instance that $d_{\mathrm{Eucl}}(P, B^\C)\leq d_{\mathrm{Eucl}}(P, A^\C)$. Similarly as in the previous case, for every $y\in AB$, let $y'$ denote the unique point on the same segment such that $d(y,A)=d(y',B)$. Fix $y\in AB$ with $d(y,A)\leq l_\theta/10$ and denote by $T$ the Euclidean trapezoid in $\R^2$ having as vertices the points $y^\C$, $(y')^\C$ and their projections onto $\ell ^{\,\C}$. It can happen that this trapezoid is self-intersecting. So, using the same strategy of Juillet and exploiting the fact that $d(y^\C, A^\C)=d(y,A)\leq l_\theta/10$,  its algebraic area can be estimated by
    \[\begin{split}\text{Area}\,(T)&\geq \left(\frac{3}{4}l_\theta-\frac{1}{10}l_\theta\right)^2\frac{\sin(2\psi)}{4}-\left(\frac{1}{4}l_\theta \right)^2\frac{\sin(2\psi)}{4}\\ &=\frac{9}{25}l_\theta^2\frac{\sin 2\psi}{4}\geq\frac{9}{400}\sin 2\psi\geq\frac{1}{4\cdot 32}\sin(2\psi). \end{split}\]
    Hence, using the same estimates as in \cite{Juillet}, $\sqrt{\text{Area}\,(T)}\geq \frac{1}{2}\sqrt{\theta\frac{7}{32\pi}}$. Applying Lemma 2.5 in the same paper we conclude exactly as in the case $\frac{\theta}{4}\leq \varphi<\sqrt{\theta}$.
\end{itemize}
\medskip
In conclusion, we have found a set $M\subset \Gamma_{n,1}$ with $\mathcal{H}^1(M)\gtrsim \mathcal{H}^1(\Gamma_{n,1})$ and the property required in \eqref{propofM}.
Now we use $M$ to construct the desired set $N\subset \Gamma_n$ with 
$\mathcal{H}^1(N)\gtrsim \mathcal{H}^1(\Gamma_n)$ and property
\eqref{eq:PropN}. 
First of all, we need to introduce some notation. Let $\gamma_n:[0,1]\to \H^1$ be a constant-speed parameterization of $\Gamma_n$ starting from the point $A$. In particular, if $\Psi:\H^1\to \H^1$ denotes the similitude which maps $\omega([\frac{\sigma}{4^n},\frac{\sigma+1}{4^n}])$ to $\Gamma_n$, then $\gamma_n(s)=(\Psi\circ\,\omega)(\frac{\sigma+s}{4^n})$. We also define $\gamma_{n,i}:[0,1]\to\H^1$, for $i\in\N$, as \[\gamma_{n,i}(s):=(\Psi\circ\omega_{n+i})\left(\frac{\sigma+s}{4^n}\right),\] so that $\Gamma_{n,i}$ is the support of $\gamma_{n,i}$.
Notice that $\gamma_{n,i}$ and $\gamma_n$ are horizontal curves. Specifically, $\gamma_{n,1}:[0,1]\to\H^1$ turns out to be a constant speed parameterization of $\Gamma_{n,1}$ starting from $A$. %, which can be obtained as $\gamma_{n,1}(s):=(\Psi\circ \omega_{n+1})(\frac{\sigma+s}{4^n})$.
 If $y\in \Gamma_n$, we denote by $y_{(i)}$ the point on $\Gamma_{n,i}$ corresponding to $y$, which means 
 \begin{equation}\label{eq:yi}
 y_{(i)}:=\gamma_{n,i}(\gamma_n^{-1}(y)).\end{equation}
 Recall here that $\omega$ (hence $\gamma_n$) is injective by Theorem \ref{injective}.
 Our goal is now to estimate the distance between $\Gamma_n$ and $\Gamma_{n,1}$. More precisely, for fixed $y\in \Gamma_n$, we want to estimate $d(y,y_{(1)})$. Recall that $\gamma_{n,i}\to \gamma_n$ uniformly (since the same holds for $\omega_n$ and $\omega$), hence in particular $y_{(i)}\to y$ as $i\to\infty$. Therefore
 \[d(y, y_{(1)})\leq \sum_{i=1}^{+\infty}d(y_{(i+1)}, y_{(i)}).\]
 The terms on the right-hand side can be estimated thanks to the following\medskip \\
 \textbf{Claim 2:} there exists an absolute constant $C>0$ such that, for every $0<i\in\N$ and $y\in \Gamma_n$, if $\bar d=\bar d_{y,i}$ is defined as 
\begin{equation}\label{dbar}\bar d:=\bar d_{y,i}:=\min\left\{\left|\gamma_n^{-1}(y)-\frac{\tau}{4^i}\right|:\tau=0,1,\dots,4^i\right\},\end{equation}
 then $d(y_{(i+1)}, y_{(i)})\leq C\sqrt{\theta}\, \bar d\, L$, where $L$ is the length of $\omega^{\mathbb{C}}$ (and of $\omega$), and $y_{(i)},y_{(i+1)}$ are as in \eqref{eq:yi}.
In particular it holds that \[d(y_{(i+1)}, y_{(i)})\leq \frac{C}{2^i}\sqrt{\theta}.\]
\phantom{}\\ 
 We now assume the validity Claim 2 and conclude the proof. Choose $\bar i\geq 1$ such that \begin{equation}\label{eq:tail}\sum_{i>\bar i}^{+\infty}\frac{1}{2^i}<\frac{K}{8\,C}\end{equation}
 where $K$ and $C$ are found as above. Fix also $0<\delta<\frac{1}{10}$ so small such that $\delta\, C\,L\leq 3K/8$.
 Define $F\subset \Gamma_n$ as 
 \[F:=\left\{y\in\Gamma_n: |\gamma_n^{-1}(y)-\bar s|\leq \delta\left(\frac{1}{4}\right)^{\bar i}\;\;\text{for some }\bar s\in\left\{0,\frac{1}{4},\frac{1}{2},\frac{3}{4},1\right\}\right\}.\]
 %Assume $y\in\Gamma_n$ is chosen such that there exists $\bar s\in\{0,\frac{1}{4},\frac{1}{2},\frac{3}{4},1\}$ with $|\gamma_n^{-1}(y)-\bar s|\leq \delta(\frac{1}{4})^{\bar i}$. 
If $y\in F$ and $i\leq \bar i$, it follows that $\bar d_{y,i}\leq \delta\left(\frac{1}{4}\right)^{\bar i}$. Hence, by Claim 2 and the last estimates,
 \[\begin{split} d(y, y_{(1)})\leq \sum_{i=1}^{+\infty}d(y_{(i+1)}, y_{(i)})&=\sum_{i=1}^{\bar i}d(y_{(i+1)}, y_{(i)})+ \sum_{i>\bar i}^{+\infty}d(y_{(i+1)}, y_{(i)})\\ &\leq \sum_{i=1}^{\bar i}C\sqrt{\theta} \,\delta \left(\frac{1}{4}\right)^{\bar i}L+ \sum_{i>\bar i}^{+\infty}C\sqrt{\theta}\frac{1}{2^i}\\
 &\leq C\sqrt\theta\,L\,\delta\left(\sum_{i=1}^{+\infty}\left(\frac{1}{4}\right)^i\right) + C\sqrt\theta\sum_{i>\bar i}^{+\infty}\frac{1}{2^i}\\ &\overset{\eqref{eq:tail}}{\leq }\frac{3K}{8}\sqrt\theta\,\frac{1}{3}+C\sqrt\theta\frac{K}{8C}=\frac{K}{4} \sqrt\theta.\end{split}\]
 Finally, using the set $M$ with property \eqref{propofM}, we define $N\subset \Gamma_n$ as
 \[N:=\{y\in F: y_{(1)}\in M\}.\]
 Then, combining the previous estimate with \eqref{propofM}, it follows that for every $y\in N$ and each $\ell\in \mathcal V_1$,
 \[d(y,\ell)\geq d(y_{(1)},\ell)-d(y,y_{(1)})\geq \frac{K}{2}\sqrt\theta-\frac{K}{4}\sqrt\theta=\frac{K}{4}\sqrt\theta,\]
 as required in \eqref{eq:PropN}.
 We need also to show that $\mathcal H^1(N)\geq c\,\mathcal H^1(\Gamma _n)$, for some constant $c$.
 First of all, by definition of $F$ it is clear that $\gamma_{n}^{-1}(F)$ coincides with the $\delta\left(\frac{1}{4}\right)^{\bar i}$-neighborhood of $\left\{0,\frac{1}{4},\frac{1}{2},\frac{3}{4},1\right\}$ inside $[0,1]$.
% $\mathcal L^1(\gamma_n^{-1}(E))=8\,\delta\left(\frac{1}{4}\right)^{\bar i}.$ 
Moreover, note that the set $M\subset \Gamma_{n,1}$ constructed in the three cases discussed below \eqref{propofM} is contained in a line segment %(parameterized by $\gamma_{n,1}$ with constant speed) 
and, in addition, it has the following property: for every $\epsilon<l_\theta/10$, $M$ contains a point which is at distance exactly $\epsilon$ from one of the two ends of the segment (see the construction of $M$). Since $\gamma_{n,1}$ parameterizes $\Gamma_{n,1}$ with constant speed, then the same property holds for $\gamma_{n,1}^{-1}(M)$: there is $\tau\in\{0,1,2,3\}$ such that, for every $\epsilon<1/40$, there exists $s\in\gamma_{n,1}^{-1}(M)$ with $\left|s-\frac{\tau}{4}\right|=\eps$ or $\left|s-\frac{\tau+1}{4}\right|=\eps$. \medskip\\ Since $\gamma_n^{-1}(N)=\gamma_n^{-1}(F)\,\cap \gamma_{n,1}^{-1}(M)$, we deduce by the previous remarks that \[\mathcal L^1(\gamma_n^{-1}(N))\geq\delta\left(\frac{1}{4}\right)^{\bar i}.\] 
Finally, since $\gamma_n$ parametrizes $\Gamma_n$ with constant speed, we conclude that 
\[\mathcal H^1(N)\geq \delta\left(\frac{1}{4}\right)^{\bar i}\mathcal H^1(\Gamma_n).\]
This concludes the proof of Claim 1, up to the proof of Claim 2.
In order to complete the proof of Theorem \ref{gammaglem}, it remains to show the validity of Claim 2. Notice that the final part of the claim follows from the fact that $\bar d\,L\leq \left(\frac{1}{2}\frac{1}{4^i}\right)\cdot 4\leq \frac{1}{2^i}$ for every $0<i\in\N$. Hence, to prove Claim 2, it suffices to show that, for $y\in\Gamma_n$ and $0<i\in\N$, 
\begin{equation}\label{eq:Goaldy_i}
d(y_{(i+1)}, y_{(i)})\leq C\sqrt{\theta}\, \bar d\, L.
\end{equation}
In order to compute this distance we use Lemma \ref{distance}. Let $\bar\tau$ be the value reaching the minimum in the definition \eqref{dbar} of $\bar d = \bar d_{y,i}$. The curves $\gamma_{n,i}$ and $\gamma_{n,i+1}$ coincide at $\frac{\bar\tau}{4^i}$: this follows from the fact that
\[\gamma_{n,i}\left(\frac{\bar\tau}{4^i}\right)=\Psi\left(\omega_{n+i}\left(\frac{4^i\sigma +\bar\tau}{4^{n+i}}\right)\right)=\Psi\left(\omega_{n+i+1}\left(\frac{4^i\sigma+\bar\tau}{4^{n+i}}\right)\right)=\gamma_{n,i+1}\left(\frac{\bar\tau}{4^i}\right),\]
since $\omega_{n+i}$ and $\omega_{n+i+1}$ coincide on integer multiples of $\frac{1}{4^{n+i}}$. \medskip\\ Choose $\rho:=\bar \tau$ if $\gamma_n^{-1}(y)\geq\frac{\bar\tau}{4^i}$, while $\rho:=\bar\tau-1$ if $\gamma_n^{-1}(y)\leq\frac{\bar\tau}{4^i}$.
Notice that the set $\gamma_{n,i}([\frac{\rho}{4^i},\frac{\rho+1}{4^i}])$ is a horizontal segment, parameterized with constant speed. The set $\gamma_{n,i+1}([\frac{\rho}{4^i},\frac{\rho+1}{4^i}])$ is instead a polygonal curve (again parameterized with constant speed), whose projection on $\R^2$ forms, with the projection of $\gamma_{n,i}([\frac{\rho}{4^i},\frac{\rho+1}{4^i}])$, two congruent isosceles triangles with base angles equal to $\theta_{n+i+1}$ (see Figure \ref{fig:gamma_{n,i}}). Since both curves are parametrized with constant speed, the line segment $\pi(y_{(i)})\,\pi(y_{(i+1)})$ is orthogonal to $\pi(\gamma_{n,i}([\frac{\rho}{4^i},\frac{\rho+1}{4^i}]))$.
\begin{figure}[h]
        \centering
        \includegraphics[width=17 cm, height=5.1 cm]{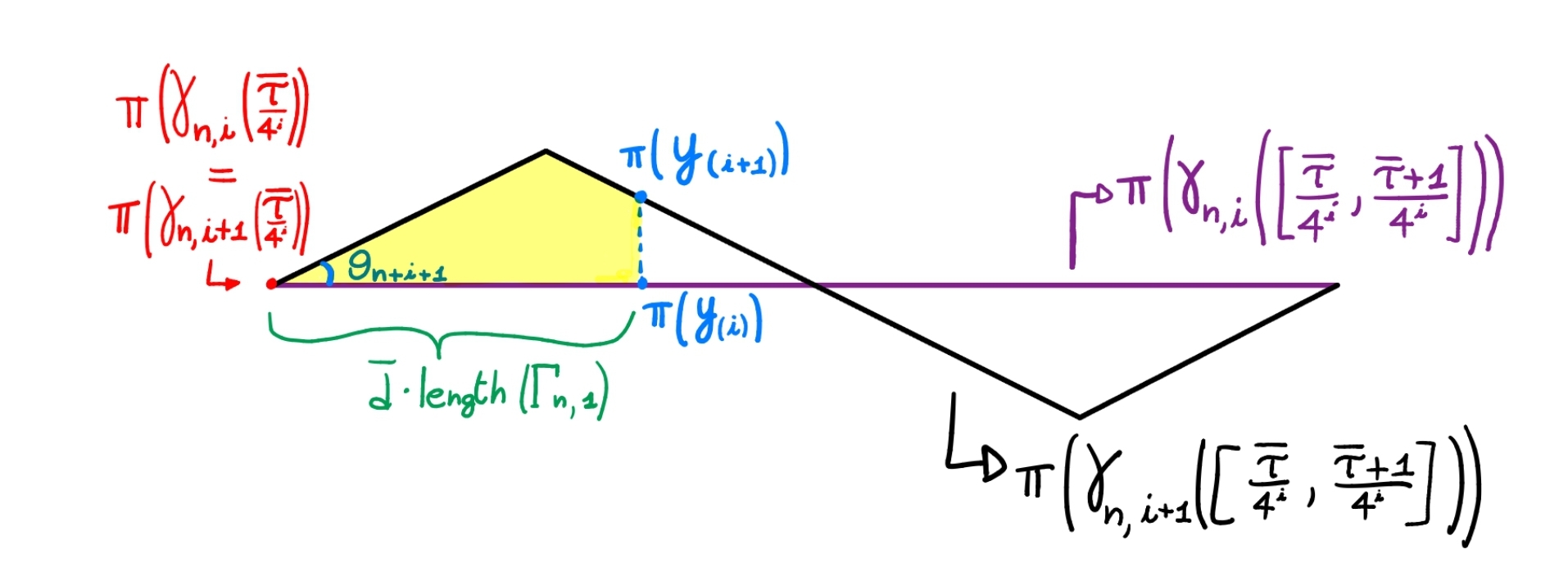}\caption{Part of the curves $\pi\circ\gamma_{n,i}$ and $\pi\circ\gamma_{n,i+1}$.}\label{fig:gamma_{n,i}}
    \end{figure}
    
\noindent Elementary computations show that 
\[\begin{split}d_e(\pi(y_{(i)}),\pi(y_{(i+1)}))&\leq  d_e\left(\pi\left(\gamma_{n,i}\left(\bar \tau/{4^i}\right)\right), \pi(y_{(i)})\right)\tan \theta_{n+i+1}\\&\leq \bar d\,\,\ell_d (\Gamma_{n,i})\tan \theta\\ & \leq \bar d\,L \,2\theta\leq \bar d\,L\sqrt\theta.\end{split}\]
In the previous estimates, we used the fact that $\theta_{n+i+1}\leq \theta\equiv\theta_{n+1} \leq 1/4$, the property that $\Gamma_{n,i}$ is parametrized with constant speed and the inequalities $\ell_d(\Gamma_{n,i})\leq\,\ell_d(\Gamma_n)\leq L$ (which follow since $\ell_d(\Gamma_n)=\sup_i\, \ell_d(\Gamma_{n,i})$ as in \eqref{limitlength} and by the fact that $\theta_{n+j}\leq \theta_j$ for $j\geq 1$).\medskip\\
Similarly, if we compute the area $A$ enclosed by the (three or four) segments joining $\pi(y_{(i)})$, $\pi(\gamma_{n,i}(\bar\tau/4^i))$, $\pi(y_{(i+1)})$, i.e. traveling on $\pi(\Gamma_{n,i})$, on $\pi(\Gamma_{n,\,i+1})$ and on $\pi(y_{(i+1)})\pi(y_{(i)}) $, one gets
\[\begin{split}A&\leq \frac{1}{2}d_{\mathrm{Eucl}}\left(\pi\left(\gamma_{n,i}\left(\bar \tau/{4^i}\right)\right), \pi(y_{(i)})\right)^2\tan \theta_{n+i+1}\\&\leq\frac{1}{2}\bar d^{\,2}\ell^2_d (\Gamma_{n,i})\tan\theta\\ &\leq \frac{1}{2}\bar d^{\,2}L^2\,2\theta=\bar d^{\,2}L^2\theta.\end{split}\]
Applying Lemma \ref{distance}, since $\gamma_{n,i}(\bar\tau/4^i)=\gamma_{n,i+1}(\bar\tau/4^i)$, we get 
\[d(y_{(i)}, y_{(i+1)})\lesssim d_{\mathrm{Eucl}}(\pi(y_{(i)}),\pi(y_{(i+1)}))+\sqrt A\lesssim \bar d\, L\sqrt\theta\]
as required in \eqref{eq:Goaldy_i}, 
and we conclude the proof of Claim 2, and thus of Theorem \ref{gammaglem}.
\end{proof}

\begin{theorem}\label{t:JuilletCurveGLem}
    $\Gamma\in{\rm GLem}(\widehat\beta_{1,\mathcal V_1},4)$.
\end{theorem}
\begin{proof}
    By \cite[Theorem 1.5]{Li}, if $E\subset \mathbb{H}^1$ and $\Gamma_E$ is any curve containing $E$, then 
    \begin{equation*}\label{gammatraveling}\int_0^{+\infty}\int_{\H^1}\widehat\beta_{\infty,\mathcal V_1}(x,r)^4\,dx\,\frac{dr}{r^4}\lesssim \mathcal H^1(\Gamma_E),\end{equation*}
where the $\widehat{\beta}$-numbers refer to those of $E$.
It is then a standard argument to deduce from this property applied for dyadic cubes %$E$
on the $1$-regular curve $\Gamma$ that  $\Gamma\in{\rm GLem}(\widehat\beta_{\infty,\mathcal V_1},4)$, for instance with the help of \cite[Corollary 3.2]{FV1}. See also
  the proof of \cite[Proposition 3.1]{chouli}. Finally, 
 $\Gamma\in{\rm GLem}(\widehat\beta_{\infty,\mathcal V_1},4)$ implies $\Gamma\in{\rm GLem}(\widehat\beta_{1,\mathcal V_1},4)$.
\end{proof}

\section{Sufficient conditions for uniform rectifiability}\label{s:SuffCond}%\textcolor{magenta}{Adjust this section with intro.}
In this section, we establish several sufficient conditions for a $k$-regular set $E\subset \H^n$ to have \emph{big pieces of Lipschitz images of subsets of $\R^k$}, for $1\leq k\leq n$. All geometric lemmas in the following statements can be phrased equivalently in terms of discrete sums over dyadic cubes, or in terms of double integrals over points and scales, according to Lemma \ref{l:DiffGeomLem}. We recall the statement of Theorem~\ref{t:ImprovedHahlomaa}: 

\medskip

A $k$-regular set in $\mathbb{H}^n$, for $1\leq k\leq n$, with the geometric lemma $\mathrm{GLem}(\beta_{1,\pi,A(2n,k)},2)$ and the weak geometric lemma $\mathrm{WGL}(\beta_{\infty,\mathcal{V}_k})$ has BPLI.  

\medskip

The core of the proof of Theorem \ref{t:ImprovedHahlomaa} is the following result.

\begin{theorem}\label{t:FromWeakAssToCoronaByNormed}
    Let $n\in \mathbb{N}$ and $k\in \{1,\ldots,n\}$. If $E\subset \mathbb{H}^n$ is a $k$-regular set with $E\in \mathrm{GLem}(\beta_{1,\pi,A(2n,k)},2)$ and $E\in \mathrm{WGL}(\beta_{\infty,\mathcal{V}_k})$, then $E$ has a corona decomposition by horizontal planes (P-C).
\end{theorem}

Once Theorem \ref{t:FromWeakAssToCoronaByNormed} is established, Theorem \ref{t:ImprovedHahlomaa} follows by the work of Bate, Hyde, and Schul \cite{Bate} combined with the results in Section \ref{ss:CoronaDef} for the different corona decompositions.

\begin{proof}[Proof of Theorem \ref{t:ImprovedHahlomaa} assuming Theorem \ref{t:FromWeakAssToCoronaByNormed}]
    By Theorem \ref{t:FromWeakAssToCoronaByNormed}, if $E\subset \mathbb{H}^n$ is  $k$-regular with $E\in \mathrm{GLem}(\beta_{1,\pi,A(2n,k)},2)$ and $E\in \mathrm{WGL}(\beta_{\infty,\mathcal{V}_k})$, it has (P-C). It follows from Lemma 
    %\ref{l:P-CtoPL-C} and Lemma \ref{l:PL-CtoN-C} 
    \ref{l:P-CtoN-C}
    that $E$ has a corona decomposition by normed spaces. \cite[Theorem B]{Bate} yields that $E$ has BPLI.
\end{proof}

The proof of Theorem \ref{t:FromWeakAssToCoronaByNormed} is postponed to the next section. Here we limit ourselves to discussing the assumptions in Theorems \ref{t:ImprovedHahlomaa}  and \ref{t:FromWeakAssToCoronaByNormed}, and deducing the corollaries stated in the introduction. First, it is clear that the weak geometric lemma $\mathrm{WGL}(\beta_{\infty,\mathcal{V}_k})$ alone does not imply the BPLI property of a $k$-regular set in $\mathbb{H}^n$ ($1\leq k\leq n$), see for instance Remark \ref{r:WGLNothingelse}. Moreover, a single summability condition on the $\beta_\pi$-numbers, as in the statement of Theorem \ref{t:ImprovedHahlomaa}, cannot suffice to prove BPLI as the following example shows.

\begin{example}[$\beta_{\pi}$-number conditions alone do not imply BPLI]\label{ex:ProjBetaNotEnough} Let $K\subset [0,1]$ be the standard Cantor middle-half set of Hausdorff dimension $1/2$. The set $E=\{(0,t):\; t\in K\}\subset \mathbb{H}^1$ is $1$-regular in the Heisenberg distance. Since $\pi(E)=\{0\}$ is entirely contained in any $1$-dimensional subspace of $\mathbb{R}^{2}$, the associated $\beta_{1,\pi,\mathcal{V}_k}$-numbers (hence also the $\beta_{1,\pi,A(2n,k)}$-numbers) vanish identically. However, $E$ cannot have BPLI, for if it did, then it would be contained in a regular curve (\cite[Corollary 4.6]{FV1}) and \cite{LiSchul1} (with an argument as in the proof of Theorem \ref{t:JuilletCurveGLem}) would therefore imply that $E\in \mathrm{GLem}(\beta_{\infty,\mathcal{V}_1},4)$.
However, this is impossible since $E$ is badly approximable by horizontal lines. Indeed, for every $x\in E$ and $0<r<\mathrm{diam}(E)$, the set $E\cap B(x,r)$ contains two (in fact many) points on the vertical $t$-axis at Heisenberg distance $\gtrsim r$ from each other. This implies that for ever horizontal line $V\in \mathcal{V}_1$ considered in the infimum $\beta_{\infty,\mathcal{V}_1}(x,r)$, there exists at least one point in $E\cap B(x,r)$ with distance $\gtrsim r$ from $V$. Therefore $\beta_{\infty,\mathcal{V}_1}(x,r)\gtrsim 1$ for $x\in E$, $0<r<\mathrm{diam}(E)$.
\end{example}

As discussed in the introduction, Theorem \ref{t:ImprovedHahlomaa} leads to several corollaries related to existing literature. In particular, uniform rectifiability for $k$-regular sets in $\H^n$ can be inferred from geometric lemmas expressed in terms of Hahlomaa's  horizontal $\beta$-numbers \cite{Hahlomaa}, Li's stratified $\beta$-numbers \cite{Li}, or the $\iota$-numbers introduced by the first author and Violo \cite{FV1}.
    \begin{proof}[Proof of Corollary \ref{c:HahlomaaWithHorizontal} ($\mathrm{GLem}(\beta_{1,\mathcal{V}_k},2)$ implies BPLI)]
  This follows from Theorem \ref{t:ImprovedHahlomaa}       with the help of Propositions \ref{p:FromStratifiedToWeakAssumpt} and \ref{p:FromHorizToStartif}.
    \end{proof}

\begin{proof}[Proof of Corollary \ref{c:HahlomaaWithStratifASs}  ($\mathrm{GLem}(\widehat{\beta}_{1,\mathcal{V}_k},4)$ implies BPLI)]
    This follows from Theorem \ref{t:ImprovedHahlomaa} with the help of Proposition \ref{p:FromStratifiedToWeakAssumpt}.
\end{proof}

\begin{proof}[Proof of Corollary \ref{c:HahlomaaWithIotaASs} ($\mathrm{GLem}(\iota_{1,\mathcal{V}_k},1)$ implies BPLI)]
    This follows from Theorem \ref{t:ImprovedHahlomaa} with  Propositions \ref{p:FromIotaToStratifAssumpt} and \ref{p:FromStratifiedToWeakAssumpt}.
    \end{proof}
Through the equivalence of (P-C) with (ILG-C), we also obtain from \cite{Bate} that (ILG-C) implies BPLI, which mirrors one of the implications proven by David and Semmes in Euclidean spaces \cite{David1}. 
\begin{proof}[Proof of Corollary \ref{t:FromCoronaToBPLI} ((ILG-C) implies BPLI)]
By Lemma \ref{lem_equivcor}, if $E$ satisfies (ILG-C), then also (P-C). The result then follows by 
 Lemma \ref{l:P-CtoN-C} and \cite[Theorem B]{Bate}. 
\end{proof}

\section{Proof of the corona decomposition}\label{s:ProofCorona}

In this section, we revisit the corona decomposition and  prove Theorem \ref{t:FromWeakAssToCoronaByNormed}. The general procedure for constructing corona decompositions is outlined in \cite[p.19]{David1}.  We closely follow the proof by Hahlomaa \cite[p.5-17]{Hahlomaa}, which, in turn, is significantly inspired by David and Semmes' original construction of a corona decomposition for sets satisfying the geometric lemma $\mathrm{GLem}(\beta_1,2)$ in Euclidean spaces. Our task is to verify that all the proof steps can be carried out under our strictly weaker assumptions $E\in \mathrm{GLem}(\beta_{1,\pi,A(2n,k)},2)$ and $E\in \mathrm{WGL}(\beta_{\infty,\mathcal{V}_k})$. We will present the arguments in all detail where they deviate from \cite{Hahlomaa}. For the parts of the proof that follow \cite{Hahlomaa}, we omit details if they can either be verified with a straightforward computation or they proceed as in the standard Euclidean construction  \cite{David1}. The structure of the proof is as follows:
\begin{itemize}
    \item[\ref{ss:PrelimThm}] Preliminaries for the proof of Theorem \ref{t:FromWeakAssToCoronaByNormed}
    \item[\ref{ss:CoronaStep1}] Assigning planes to good cubes
   \item[\ref{ss:BuildingCoronization}] Building the coronization
\item[\ref{ss:VerifP-C}] Verifying the corona decomposition by horizontal planes (P-C)
\end{itemize}

Hahlomaa does not explicitly state the conclusion of his proof in the form of a corona decomposition, but we decided to do so for two reasons: first, it is straightforward to deduce from the corona decomposition by horizontal planes (P-C) the one by normed spaces (N-C), which allows us to apply the recent result by Bate, Hyde, and Schul \cite{Bate}. Second, as we have shown in Lemma \ref{l:PCimpliesILGC} and Lemma \ref{lem_equivcor}, (P-C) is equivalent to the corona decomposition by intrinsic Lipschitz graphs (ILG-C), so Theorem \ref{t:FromWeakAssToCoronaByNormed} can be seen as a Heisenberg counterpart for David and Semmes' proof that the geometric lemma implies the  existence of corona decompositions by Lipschitz graphs, recall  Corollary \ref{t:FromWeakAssToILG-C}.

\begin{proof}[Proof of Corollary  \ref{t:FromWeakAssToILG-C}] ($\mathrm{Glem}(\beta_{1,\pi, A(2n,k)}, 2)$ and $\mathrm{WGL}(\beta_{\infty,\mathcal V_k})$ imply (ILG-C)).
    This follows immediately from Theorem \ref{t:FromWeakAssToCoronaByNormed} and  Lemma \ref{l:PCimpliesILGC}. 
\end{proof}

\subsection{Preliminaries for the proof of Theorem \ref{t:FromWeakAssToCoronaByNormed}}\label{ss:PrelimThm}
Let $E\subset\H^n$ be a $k$-regular set with $E\in \mathrm{GLem}(\beta_{1,\pi,A(2n,k)},2)$ and $E\in \mathrm{WGL}(\beta_{\infty,\mathcal{V}_k})$. Let $C_E$ denote the regularity constant of $E$. We also use the notation $\mu:=\mathcal H^k|_E$. Given $V_1,V_2\in\mathcal V_k$, we define the \textit{angle} between $V_1$ and $V_2$ as 
\begin{equation}\label{eq:angle}
\angle(V_1,V_2):=\min\{C\geq 1: d(x,y)\leq Cd(P_{V_2}(x),P_{V_2}(y))\;\text{ for all }x,y\in V_1\}.
\end{equation}
\begin{remark}\label{r:angle}
This is the definition used in \cite[p.4]{Hahlomaa}, which we adopt here to facilitate comparison of certain arguments in our proof with the corresponding ones in \cite{Hahlomaa}. However, this definition of angle differs from the ones used in \cite{David1,David2,FV2}. Notably, $\angle(V_1,V_2)=1$ for $V_1=V_2$ according to \eqref{eq:angle} (whereas this angle would be $0$
 according to the definitions in the cited references). Moreover, if, e.g.,  $V_1,V_2$ are two horizontal lines through the origin with an angle of $\delta$, then  $\angle(V_1,V_2)\sim 1+ \frac{\delta^2}{2}$, for small $\delta$, with the definition \eqref{eq:angle}.% 
 \end{remark}
 \begin{remark}
As proved in \cite[Lemma 2.3]{Hahlomaa}, the angle $\angle(V_1,V_2)$ can be computed by only considering the projections of $V_1$ and $V_2$ on $\R^{2n}$:
     \[\angle(V_1,V_2)=\min\{C\geq 1:|x-y|\leq C|\pi_{\pi(V_2)}(x)-\pi_{\pi(V_2)}(y)|\;\text{ for all }x,y\in \pi(V_1)\}.\]
This is applied in the proof of Lemma  \ref{lem_propertiesofg} and Lemma \ref{l:HLem7.1}. Although the definition is not expressed in a symmetric form, it is not difficult to see that $V_1$ and $V_2$ can be interchanged.\end{remark}
If $A\subset\H^n$ and $r>0$ we will denote by $A(r)$ the $r$-neighborhood of $A$, % with respect to the Koranyi distance on $\H^n$,
that is defined by
\[A(r):=\{x\in\H^n: d(x,A)\leq r\}.\]
We will show that there exists a constant $\eta_0>0$, depending only on $k$ and  $C_E$, such that $E$ satisfies the conditions in Definition \ref{d:P-C} (P-C)
 for all $0<\eta<\eta_0$ (and thus, a posteriori, for all $\eta>0$). 
 %For now, we fix $\eta>0$ arbitrarily, and we will show that the following works whenever $\eta$ is small enough, depending only on $k$ and $C_E$.
\textbf{All the following constants may depend on $k$ and the $k$-regularity constant $C_E$ of $E$ without explicit mentioning.}

\begin{enumerate}
%\item \textcolor{magenta}{The constant $K_0>1$ will be chosen large enough.\\ The precise conditions will be stated in Sections \ref{ss:Step3} -- \ref{sss:Step4}.}
%\item \textcolor{magenta}{The constant $0<\eta\leq \eta_0$ will be chosen small enough.\\ The precise conditions are stated  in Sections \ref{sss:Step 2} %\ref{ss:Step3} 
%-- \ref{sss:Step4}.}
%\item \textcolor{magenta}{The constant $K>1$ will be chosen large enough depending on $K_0$ and $\eta$.\\ The precise conditions will be stated in the proof of Lemma \ref{l:HahlLem4.3}, as well as in Sections \ref{sss:Step 2} and \ref{sss:Step4}.} 
%\item \textcolor{magenta}{Finally, the constant $0<\varepsilon<1$ will be chosen small enough, depending on $\eta$ and $K$.\\ The precise conditions will be stated in Section \ref{ss:CoronaStep1}, in the proof of Lemma \ref{l:HahlLem4.3}, as well as in Sections \ref{sss:Step 2} and \ref{sss:Step4}.}
\item The constant $K_0>1$ will be chosen large enough, as specified in Sections \ref{ss:Step3} -- \ref{sss:Step4}.
\item The constant $0<\eta\leq \eta_0$ will be chosen small enough, as in Sections \ref{sss:Step 2} %\ref{ss:Step3} 
-- \ref{sss:Step4}.
\item The constant $K>1$ will be chosen sufficiently large, depending on $K_0$ and $\eta$; see the proof of Lemma \ref{l:HahlLem4.3} and Sections \ref{sss:Step 2} and \ref{sss:Step4}. 
\item Finally, the constant $0<\varepsilon<1$ will be chosen small enough, depending on $\eta$ and $K$, as explained in Section \ref{ss:CoronaStep1}, in the proof of Lemma \ref{l:HahlLem4.3}, as well as in Sections \ref{sss:Step 2} and \ref{sss:Step4}.
\end{enumerate}
The constants $K$ and $\varepsilon$ will be used to define the good cubes in the coronization for parameter $\eta$, and as such they are allowed to depend on $\eta$.\\ 
%\textcolor{magenta}{CHECK THESE}
We fix $0<\eta<\eta_0$ and aim to show that $E$ satisfies the conditions of (P-C) for this constant. That is, we have to show that $E$ admits a coronization (with a constant $C=C(\eta)$ depending on $\eta$) such that the projection condition \eqref{eq:P-C} holds with parameter $\eta$.

\subsection{Assigning planes to good cubes}
\label{ss:CoronaStep1}

Let $\mathcal{D}$ be an arbitrary but fixed dyadic system on $E$. We define a ``good family'' $\mathcal{G}_1\subset \mathcal{D}$ such that each 
$Q\in \mathcal{G}_1$ is well approximated by a horizontal plane. More precisely, for fixed parameters $\varepsilon>0$, $K>1$ (which will be suitably determined later), we let
\[\mathcal G_1=\mathcal G_1(\varepsilon, K):=\{Q\in\mathcal D: KQ\subset V(\varepsilon^2\mathrm{diam}(Q))\;\;\text{for some }V\in\mathcal V_k\}.\]
For each $Q\in\mathcal G_1$, we then fix $V_Q\in\mathcal V_k$ to be an affine horizontal $k$-dimensional plane with the property that $KQ\subset V_Q(\varepsilon^2\diam(Q))$.
The assumption $E\in \mathrm{WGL}(\beta_{\infty,\mathcal{V}_k})$ implies a Carleson-packing condition for the remaining cubes $\mathcal{D}\setminus \mathcal{G}_1$ (in 
\cite[p.5]{Hahlomaa} this is deduced from the assumption $E\in \mathrm{GLem}(\beta_{1,\mathcal{V}_k},2)$). 
\begin{lemma}
    There exists a constant $C=C(\varepsilon, K)$ such that
    \begin{equation}\label{eq_carlesoncondition}\sum_{Q\in \mathcal D\setminus \mathcal G_1,\, Q\subset R}\mu(Q)\leq C\mu(R),\quad R\in\mathcal D.\end{equation}
\end{lemma}

\begin{comment}
\begin{proof}
If $Q\in\mathcal D\setminus\mathcal G_1$ and $V\in\mathcal V_k$, then there exists $y\in KQ$ such that \[d(y,V)> \varepsilon^2d(Q)\geq \frac{\varepsilon^2}{2K-1} d(KQ).\]  
This implies that $\beta_{\infty,\mathcal V_k} (KQ)> \varepsilon^2/2K$ and, by the assumption $E\in\mathrm{WGL}(\beta_{\infty,\mathcal{V}_k})$, it follows
\[\sum_{Q\in \mathcal D\setminus \mathcal G_1,\, Q\subset R}\mu(Q)\leq\sum_{\beta_{\infty,\mathcal V_k}(KQ)>\frac{\varepsilon^2}{2K},\,Q\subset R} \mu(Q)\leq C\mu(R),\quad R\in\mathcal D.\qedhere\]
\end{proof}
\end{comment}

If 
a good cube $Q\in \mathcal{G}_1$ is well approximated by \emph{some} horizontal plane $V\in \mathcal{V}_k$,
%in the sense that $Q\subset V(2K\varepsilon^2d(Q))$,
then this plane $V$ must make a small angle with $V_Q$, as made precise in the following lemma.
%\textcolor{blue}{, namely $\angle(V_Q ,V)\leq 1+\varepsilon$}.

\begin{lemma}\label{l:HahLem3.2}
    If $Q\in \mathcal{G}_1$ and $V\in \mathcal{V}_k$ is such that $Q\subset V(2K\varepsilon^2\diam(Q))$ then $\angle(V_Q ,V)\leq 1+\varepsilon$.
\end{lemma}

This is 
\cite[Lemma 3.2]{Hahlomaa}, which holds choosing $\varepsilon$ small enough depending on $K$.
%\footnote{\textcolor{magenta}{Compare Lemma \ref{l:HahLem3.2} with \cite[Lemma 5.13]{David1}. In particular, explain $\varepsilon^2$}}
Bearing in mind Remark \ref{r:angle} about the definition of angles, this is similar in spirit to  \cite[Lemma 5.13]{David1} in the Euclidean setting, although weaker in terms of the exponents of $\varepsilon$. Lemma \ref{l:HahLem3.2} will be applied in Section \ref{sss:Step 2} to construct a Lipschitz function $g:\mathbb{R}^k \to \mathbb{R}^{2n-k}$ for each tree $\mathcal{S}$ in the coronization that we are going to construct in the following. 
%and will be applied in Section \ref{ss:VerifP-C}.

\subsection{Building the coronization}\label{ss:BuildingCoronization}
We now describe how to build the coronization $\mathcal{D}=\mathcal{G}\dot{\cup}\mathcal{B}$. This is done exactly as in \cite[p.6--7]{Hahlomaa} taking the parameter $\delta=\eta$ (see also \cite[Section 7]{David1} for the original Euclidean construction). The reader can think that the good cubes $\mathcal{G}$ are simply the family $\mathcal{G}_1$ from Section \ref{ss:CoronaStep1}. This will be true for bounded $E$, and in the case of unbounded $E$ one selects a family $\mathcal{G}\subset \mathcal{G}_1$ to run a suitable localization argument similar to  \cite[p.38]{David1}. The family $\mathcal{G}$ still satisfies \eqref{eq_carlesoncondition} and is then divided into a forest $\mathcal{F}\subset \mathcal P(\mathcal G)$ of disjoint trees $\mathcal S$, %that are defined with the help of angles between the planes $V_Q$, $Q\in \mathcal{S}$. 
which is constructed in such a way that the following properties hold:
\begin{enumerate}
\item Each $\mathcal S\in\mathcal F$ has a maximal element (with respect to inclusion), which will be denoted by $Q(\mathcal S)$;
\item If $Q\in\mathcal S, R\in \mathcal D$ are such that $Q\subset R\subset Q(\mathcal S)$, then $R\in\mathcal S$;
\item For each $Q\in\mathcal S$, $\angle(V_Q,V_{Q(\mathcal S)})\leq 1+\eta$;
\item If $Q\in\mathcal S,\, \mathcal C(Q)\subset\mathcal G$ and $\angle(V_R,V_{Q(\mathcal S)})\leq 1+\eta$ for all $R\in \mathcal C(Q)$, then $\mathcal C(Q)\subset \mathcal S$;
\item If $Q\in \mathcal S$, then $Q\in \min(\mathcal S)$ if and only if $\mathcal C(Q)\setminus\mathcal G\neq \emptyset$ (``$Q$ has at least one bad child'') or there is $R\in \mathcal C(Q)$ such that $\angle(V_R, V_{Q(\mathcal S)})>1+\eta$ (``$Q$ has at least one child $R$ with a large angle between $V_R$ and $V_{Q(\mathcal{S})}$'').  
\end{enumerate} The forest $\mathcal F$ is further subdivided as $\mathcal{F}=\mathcal{F}_1\cup \mathcal{F}_2\cup \mathcal{F}_3$, where the conditions for the definition of the $\mathcal{F}_i$'s are stated in terms of properties of the minimal (with respect to inclusion) cubes of the trees $\mathcal{S}\in \mathcal{F}_i$. 
In particular, for a given tree $ \mathcal S\in\mathcal F$ we set
\[\begin{split} &m_1(\mathcal S):=\{Q\in\min(\mathcal S): \mathcal C(Q)\setminus\mathcal G\neq\emptyset\},
\\ &m_2(\mathcal S):=\min(\mathcal S)\setminus m_1(\mathcal S).\end{split}\]
%Notice in particular that, if $Q\in m_2(\mathcal S)$, then there exists $R\in \mathcal C(Q)$ with $\angle(V_R, V_{Q(\mathcal S)})>1+\eta$.
Then we define the partition of the forest $\mathcal F$ as
\[\begin{split}\mathcal F_1&:=\left\{\mathcal S\in\mathcal F: \mu \left(\bigcup_{Q\in m_1(\mathcal S)} Q\right)\geq \mu(Q(\mathcal S))/4\right\},\\ \mathcal F_2&:=\left\{\mathcal S\in\mathcal F: \mu \left(Q(\mathcal S)\setminus\bigcup_{Q\in \rm{min}(\mathcal S)} Q\right)\geq \mu(Q(\mathcal S))/4\right\},\\\mathcal F_3&:=\left\{\mathcal S\in\mathcal F: \mu \left(\bigcup_{Q\in m_2(\mathcal S)} Q\right)\geq \mu(Q(\mathcal S))/2\right\}.\end{split}\]
Although the definitions are different, this is similar in spirit to the definitions used in the corresponding Euclidean proof \cite[p.39]{David1}. This is because, if $Q\in m_2(\mathcal S)$, then there exists $R\in \mathcal C(Q)$ with $\angle(V_R, V_{Q(\mathcal S)})>1+\eta$.
The maximal cubes $Q(\mathcal{S})$ for $\mathcal{S}\in \mathcal{F}_1 \cup \mathcal{F}_2$ are easily seen to satisfy a Carleson packing condition. This is the statement of \cite[Lemma 4.2]{Hahlomaa}, which has a purely axiomatic proof that applies verbatim in our situation, 
see also \cite[p.39, Lemma 7.4]{David1}. Precisely, the result reads as follows:
\begin{lemma}
 There exists a constant $C=C(\varepsilon, K)$ such that
    \[\sum_{\mathcal S\in\mathcal F_1\cup\mathcal F_2,\,Q(\mathcal S)\subset R}\mu(Q(\mathcal S))\leq C\mu(R),\quad R\in\mathcal D.\]    
\end{lemma}
The crux of the matter is to prove the corresponding Carleson packing condition for the remaining trees $\mathcal{S}\in \mathcal{F}_3$. It is in the proof of this statement, which corresponds to 
\cite[(19)]{Hahlomaa}, that our assumption ``$E\in \mathrm{GLem}(\beta_{1,\pi,A(2n,k)},2)$'' will finally become relevant.
This is the context of the Section \ref{ss:VerifP-C}, which forms the core of the proof. More precisely, the main deviation of our proof from Hahlomaa's will appear in Section \ref{ss:Step3}.

\subsection{Verifying the corona decomposition by horizontal planes (P-C)}\label{ss:VerifP-C}

We conclude that $E$ has a corona decomposition by horizontal planes as follows:
\begin{itemize}
    \item[\ref{sss:Step1}] Step 1: Projections are roughly bi-Lipschitz on top cubes
    \item[\ref{sss:Step 2}] Step 2: Associating a Lipschitz function $g:\mathbb{R}^k \to \mathbb{R}^{2n-k}$ to a tree $\mathcal{S}\in \mathcal{F}$
   \item[\ref{ss:Step3}] Step 3: Pushing estimates on $\beta_{1,\pi}$ from $E$ down to the graph of $g$
   %Integral estimate for $\gamma$-numbers associated to the Lipschitz function $g$
   \item[\ref{sss:Step4}] Step 4: Bounding $\mu(Q(\mathcal{S}))$ for $\mathcal{S}\in \mathcal{F}_3$ from above in terms of the $\beta_{1,\pi,A(2n,k)}$-numbers
   \item[\ref{sss:Step5}]Step 5: % The final step in the verification of (P-C): 
   A Carleson packing condition for top cubes of trees $\mathcal S\in \mathcal F_3$
\end{itemize}

\subsubsection{Step 1: Projections are roughly bi-Lipschitz on top cubes}\label{sss:Step1}\phantom{}\vspace{0.1cm}\\
In this step, we verify property \eqref{eq:P-C} in the definition of Corona decomposition by horizontal planes (P-C). If $\mathcal S\in\mathcal F$ is constructed as in Section \ref{ss:BuildingCoronization}, recall that
\begin{equation}\label{maph}h_\mathcal S(x):=\inf\{d(x,Q)+\diam(Q):Q\in\mathcal S\}.\end{equation}
The main core of this step is the following result (see \cite[Lemma 8.4]{David1} for its Euclidean counterpart).
\begin{lemma}\label{l:HahlLem4.3} Choosing $K$ large enough (depending on $K_0$ and on $\eta$) and choosing $\varepsilon$ small enough (depending on $\eta$) one has the following property for the trees constructed in Section \ref{ss:BuildingCoronization}.
If $\mathcal{S}\in \mathcal{F}$ and $x,y\in K_0 Q(\mathcal{S})$ with $d(x,y)>\eta\, \min\{h_{\mathcal{S}}(x),h_{\mathcal{S}}(y)\}$, then 
\begin{displaymath}
d(x,y)\leq (1+2\eta) d(P_{V_{Q(\mathcal{S})}}(x),P_{V_{Q(\mathcal{S})}}(y)).
\end{displaymath}
\end{lemma}

\begin{proof}
    %\textcolor{magenta}{CHECK}
    We follow the proof given in \cite[Lemma 4.3]{Hahlomaa} with $\alpha=1/\varrho$, paying attention to our weaker assumption on $d(x,y)$. Let $\mathcal S\in\mathcal F$, let $x,y\in K_0Q(\mathcal S)$ and assume for instance that $d(x,y)>\eta\,h_\mathcal S(x)$. Then there exists $Q\in\mathcal S$ such that \[d(x,y)>\eta(d(x,Q)+\diam(Q)).\]
    Let $R\in\mathcal S$ be the minimal cube such that $Q\subset R$ and $\diam (K_0R)\geq d(x,y)$. Notice that such a cube exists, since by assumption $\diam(K_0 Q(\mathcal S))\geq d(x,y)$. Then we have the following properties:
    \begin{align}\label{1prop}d(y,R)\leq d(x,y)+d(x,R)&\leq d(x,y)+d(x,Q)<d(x,y)+\eta^{-1}d(x,y);\\ \label{2prop}\diam(R)&\leq (D^2\rho^{-1}+\eta^{-1})d(x,y).\end{align}
    Property \eqref{2prop} can be deduced in the following way: if $Q=R$, then \[\diam(R)=\diam(Q)<\eta^{-1} d(x,y)<(D^2\rho^{-1}+\eta^{-1})\,d(x,y).\] If instead $R\neq Q$, then there exists $\widehat R$ which is a child of $R$ and contains $Q$. Then, by properties of dyadic cubes, see \eqref{eq:H_(14)}, $\diam(R)\leq D^2\rho^{-1}\, \diam(\widehat R)$. Therefore, using the minimality assumption on $R$, we get
    \[\diam(R)\leq D^2\rho^{-1} \,\diam(\widehat R)\leq D^2\rho^{-1} \,\diam(K_0\widehat R)\leq D^2\rho^{-1} \,d(x,y)\leq (D^2\rho^{-1}+\eta^{-1})\,d(x,y).\]
    %\textcolor{blue}{(Here I used the notation of \cite{Hahlomaa}, in particular the constants $D$ and $\alpha$ and property (14) of his paper. If we want to use the definition of dyadic cubes given in \cite{Bate}, then we have probably to adapt this part.)}\\
    By \eqref{1prop} and the fact that $\diam(K_0 R)\leq 2K_0\diam(R)$, we get
    \[\begin{split}d(y,R)&\leq (1+\eta^{-1})\,d(x,y)\leq (1+\eta^{-1})\diam(K_0R)\leq 2K_0(1+\eta^{-1})\diam(R),\\ d(x,R)&\leq d(x,Q)\leq \eta^{-1}d(x,y)\leq \dots\leq 2K_0(1+\eta^{-1})\diam(R). \end{split}\] Hence, choosing $K$ big enough depending only on $K_0$ and $\eta$ (in particular $K\geq2K_0(1+\eta^{-1})+1)$, we get $x,y\in KR$. Recall that $R$ is well approximated by the horizontal plane $V_R\in \mathcal V_k$, so that $x,y\in V_R(\varepsilon^2 \diam(R))$. %\footnote{\textcolor{magenta}{Would we even have this with $\varepsilon^2$ by Hahlomaa's definition of goodness? (Don't need, just for clarity)}}
    Let $z,w\in V_R$ be the points realizing $d(x,z)=d(x,V_R)$ and $d(y,w)=d(y,V_R)$ and let us  denote $P:=P_{V_{Q(\mathcal S)}}$. From the previous observation we get $d(x,z)\leq \varepsilon^2 \diam(R)$ and $d(y,w)\leq \varepsilon^2 \diam(R)$. Moreover, by property (3) in the construction of the forest $\mathcal F$, we can assume $d(z,w)\leq (1+\eta)\,d(P(z),P(w))$.
    Combining the previous estimates, using the triangular inequality and the 1-Lipschitz continuity of $P$,
    \[\begin{split}d(P(x),P(y))&\geq d(P(z),P(w))-d(P(x),P(z))-d(P(y),P(w))\\ &\geq d(P(z),P(w))-d(x,z)-d(y,w) \\ &\geq (1+\eta)^{-1}\,d(z,w)-2\varepsilon^2 \diam(R)\\ &\geq (1+\eta)^{-1}(d(x,y)-2\varepsilon^2 \diam(R))-2\varepsilon^2 \diam(R)\\&\stackrel{\eqref{2prop}}\geq((1+\eta)^{-1}-4\varepsilon^2(D^2\rho^{-1}+\eta^{-1}))d(x,y)\\ &\geq (1+2\eta)^{-1}d(x,y),\end{split}\]
    where the last estimate holds if we choose $\varepsilon>0$ small enough depending on $\eta$.
\end{proof}
Having proved Lemma \ref{l:HahlLem4.3}, the last step to get a corona decomposition by horizontal planes is a Carleson packing condition for the trees $\mathcal S\in\mathcal F_3$. This is the content of the remaining steps.
%\textcolor{blue}{One needs to substitute $D^{-2}$ with $\eta$ in the proof of \cite[Lemma 4.3]{Hahlomaa} and then choose $K>0$ big enough depending on $\eta$, as well as $\delta,\epsilon$ small enough} \textcolor{magenta}{to derive our version of (P-C).}

\subsubsection{Step 2: Associating a Lipschitz function  $g:\mathbb{R}^k \to \mathbb{R}^{2n-k}$ to a tree $\mathcal{S}\in \mathcal{F}$}\label{sss:Step 2}
%\textcolor{magenta} {This should go exactly as in \cite[Section 5]{Hahlomaa}}
\phantom{}\vspace{0.1cm}\\
In this step, we follow \cite[Section 5]{Hahlomaa}: there is no modification in the argument, but we reproduce it here for the convenience of the reader and to fix the notation that will be used in the next steps. This construction relies on the original Euclidean arguments contained in \cite[Section 8]{David1}. We emphasize the perhaps surprising fact that for the construction of the corona decomposition, it suffices to construct good Lipschitz functions $g: \mathbb{R}^k \to \mathbb{R}^{2n-k}$, rather than Lipschitz functions $g:\mathbb{R}^k \to \mathbb{H}^n$. This is the fundamental reason why in Theorem \ref{t:FromWeakAssToCoronaByNormed} we only need to assume a geometric lemma for the \emph{projection} $\beta$-numbers.

Fix $\mathcal S\in\mathcal F$ and assume, up to a translation and a rotation of $\H^n$ (which are isometries), that $V_{Q(\mathcal S)}=\{x\in\H^n:x_i=0\,\text{ for all }i>k\}$.
Here, \emph{rotation} is defined as in Definition \ref{d:rot}.
Our goal is to construct a (Euclidean) Lipschitz function $g:\R^k\to\R^{2n-k}$, %\footnote{\textcolor{magenta}{Is it necessary to introduce the notation $X_k$ (it reminds of a vector field)? Or can we just use $V_{Q(\mathcal{S})}$ and $\mathbb{R}^k$)?}}
 with Lipschitz constant controlled by $\sqrt{\eta}$, whose graph approximates $\pi(K_0Q(\mathcal S))$ well  at  the  scales associated to cubes $Q\in\mathcal S$; see Lemma \ref{lem_propertiesofg} for the precise statement. \medskip\\ %(see, in particular, \cite[Lemma 5.3 and Lemma 5.6]{Hahlomaa}). This construction only depends on the compatibility properties of the tree $\mathcal S$ and the observation in Section \ref{ss:CoronaStep1} concerning the small angle between two planes well approximating a common cube $Q$; thus it can be verbatim carried out here starting from the assumption $E\in\mathrm{WGL}(\beta_{\infty,\mathcal{V}_k})$.% Then one defines new flatness coefficients associated to the function $g$, called $\gamma$-numbers (see \cite[Section 6]{Hahlomaa}). Our goal is now to estimate these quantities by means of the projection $\beta$-numbers associated to $E$.
For this construction, we let $P:\H^n\to \R^k$ be the projection on the first $k$-coordinates, that is, $P(x):=(x_1,\dots,x_k)$. Identifying $V_{Q(\mathcal S)}\equiv\R^k$, we can regard $P$ as the horizontal projection $P_{V_{Q(\mathcal S)}}$. We also denote by $P^\perp:\H^n\to\R^{2n-k}$ to be the orthogonal projection %the Euclidean projection on the orthogonal subspace of $X_k$ inside $\R^{2n}\equiv\{t=0\}\subset\H^1$, namely 
defined as $P^\perp(x):=(x_{k+1},\dots,x_{2n}).$
We then define $H:\R^k\to\R$ by setting
\[H(p):=\inf\{h_\mathcal S(x): x\in P^{-1}(\{p\})\},\] where the map $h_\mathcal S$ is defined in \eqref{maph}. It is easy to see that $H$ can be equivalently computed by
\begin{equation}\label{def_H}H(p)=\inf\{d_{\mathrm{Eucl}}(p,P(Q))+\diam(Q): Q\in\mathcal S\}.\end{equation}
Hence it follows that $H$ is a 1-Lipschitz map, with respect to the Euclidean distance on $\R^k$.
%Thus, $H$ measures the distance of $p$ from projections of small cubes.
We let $Z:=\{x\in E: h_\mathcal S(x)=0\}$ and it follows from the Bolzano-Weierstrass theorem that $p\in P(Z)$ if and only if $H(p)=0$.

The idea is then to construct $g$ by setting $g(P(z))=P^{\bot}(z)$ for $z\in Z$ (using that $P|_Z$ is injective with Lipschitz inverse by Lemma \ref{l:HahlLem4.3}) and apply Whitney-type arguments to extend $g$ to  $\mathbb{R}^k$ (or more precisely, to a ball $U_0 \subset \mathbb{R}^k$).
To this end, one partitions $\R^k\setminus P(Z)$ into a countable union of disjoint (up to edges) standard dyadic cubes $\{R_i:i\in \N\}$ of $\R^k$ such that the following Whitney-type condition holds
\begin{equation}\label{eq:choiceofRi}20 \,\diam(R_i)\leq H(p)\quad\text{for every }p\in R_i.\end{equation} We also require each $R_i$ to be maximal (with respect to inclusion) among all cubes satisfying \eqref{eq:choiceofRi}.
Fix $x_0\in Q(\mathcal S)$ and set 
\begin{equation}\label{eq:Uj}
U_j:=B^k\big(P(x_0), 2^{-j}K_0\diam(Q(\mathcal S))\big),
\end{equation}
where $B^k$ denotes the Euclidean ball in $\R^k$.
We then associate a good cube $Q_i\in\mathcal S$ to each dyadic cube $R_i$ intersecting $U_0$. Precisely, for each $i\in I_0:=\{j\in \N:R_j\cap U_0\neq\emptyset \}$, there exists $Q_i\in\mathcal S$ with diameter comparable to $R_i$, in the sense that 
\[C^{-1}\diam(R_i)\leq \diam(Q_i)\leq C\diam(R_i)\]
for a suitable constant $C$ depending only on $K_0$. The choice of $Q_i$ can be performed in the following way: by maximality of $R_i$, it follows that $H(p)\leq 42\, \diam(R_i)$ for every $p\in R_i$. Fix $p\in R_i\cap U_0$. By \eqref{def_H}, there exists $Q\in\mathcal S$ such that $d(p,P(Q))+\diam(Q)<90\,\diam(R_i)$. If $d(p,P(Q))\leq K_0\,\diam(Q)$, then we set $Q_i=Q$. Otherwise, we replace $Q$ with a suitable ancestor $Q_i$ such that $K_0^{-1}d(p, P(Q_i))\leq \diam(Q_i)\leq K_0\, d(p,P(Q_i))$. The existence of such an ancestor is guaranteed by the fact that $d(p, P(Q(\mathcal S)))\leq K_0\,\diam (Q(\mathcal S))$, since $p\in U_0$. \medskip\\%\footnote{\textcolor{magenta}{I guess the existence of such an ancestor also uses that for $R_j \cap U_0\neq \emptyset$, we have $\diam(R_i) \lesssim \diam(Q(\mathcal{S}))$ by the definition of $H$? This could perhaps be mentioned.}} \medskip\\
We are now going to construct the desired map $g$. First, for each $i\in I_0$, let $A_i:\R^k\to\R^{2n-k}$ be the Euclidean affine function whose graph is $\pi(V_{Q_i})$. We also pick $\widetilde \phi_i:\R^k\to[0,1]$ to be a $C^2$ cutoff function assuming value 1 on $2R_i$, value 0 on $\R^k\setminus 3R_i$ and such that \[|\partial_j\widetilde \phi_i|\leq C\diam(R_i)^{-1},\quad |\partial_j\partial_m\widetilde \phi_i|\leq C\diam(R_i)^{-2}\quad\text{ for all }j,m\in\{1,\dots,k\}.\] These functions will be used for a Whitney-type construction on $U_0\setminus P(Z)$.
On the other hand, for each $p\in P(Z)$ there exists a unique point $x(p)\in E$ such that $P^{-1}(\{p\})\cap K_0Q(\mathcal S)=\{x(p)\}$: recall that this 
is an immediate consequence of Lemma \ref{l:HahlLem4.3}.
Hence we can finally define the function $g:U_0\to\R^{2n-k}$ by setting
\begin{equation}\label{eq:Defg}g(p):=\begin{cases}
   \displaystyle\frac{\sum_{i\in I_0} \widetilde \phi_i(p)A_i(p)}{\sum_{i\in I_0}\widetilde \phi_i(p)} & \text{ if } p\in U_0\setminus P(Z)\\ P^\perp(x(p)) &\text{ if }p\in P(Z)
\end{cases}\end{equation}
The main features of the map $g$ are summarized in the following Lemma, which mimics the corresponding statement \cite[Proposition 8.2]{David1} in the construction of the classical corona decomposition.
\begin{lemma}\label{lem_propertiesofg} Choosing $\eta\leq 1$, $K$ large enough depending on $K_0$ and $\eta$, and choosing $\varepsilon$ small enough depending on $\eta$ and $K$, the following property holds: if $g:U_0\subset\R^k\to\R^{2n-k}$ is the function constructed as above, then there is a positive constant $C$ such that
    \begin{enumerate}
        \item[(i)] $g$ is $C\sqrt\eta$-Lipschitz;
        \item[(ii)] $|P^\perp (x)-g(P(x))|\leq C\sqrt\varepsilon h_\mathcal S(x)$ for all $x\in K_0Q(\mathcal S)$;
        \item [(iii)] $|g(p)|\leq CK_0\sqrt\eta\, \diam(Q(\mathcal S))$ for all $p\in U_0$.
    \end{enumerate}
\end{lemma}
\begin{proof}
Properties (i)-(iii) are proven in \cite[Lemma 5.3, Lemma 5.6, and Lemma 5.9]{Hahlomaa}, using, in particular, the properties of the planes $V_Q$ stated in Lemma \ref{l:HahLem3.2} and Lemma \ref{l:HahlLem4.3}, as well as property (3) of the tree $\mathcal{S}$ in Section \ref{ss:BuildingCoronization}. The definition \eqref{eq:Defg} of $g$ with the ($2\sqrt{\eta}$-Lipschitz) affine functions $A_i$ allows one to deduce the $C\sqrt{\eta}$-Lipschitz continuity of $g$, the approximation of $\pi(K_0 Q(\mathcal{S}))$ by the graph of $g$, and the estimate for $|g(p)|$ from this geometric information.\end{proof}
\begin{remark}
The exponents of $\eta$ and $\varepsilon$ in Lemma \ref{lem_propertiesofg} look different from the corresponding ones in the Euclidean predecessor \cite[Proposition 8.2]{David1}. However, the reader should keep in mind the different conventions used for the definition of angles (Remark \ref{r:angle}) and pay attention to the exponents in Lemma \ref{l:HahLem3.2}.
%, which is used in the proof of Lemma \ref{lem_propertiesofg}.
What is important for the construction of the corona decomposition, is that we obtain estimate (ii) with $\varepsilon$ (and not with $\eta$). This will allow us to prove Lemma \ref{l:HLem7.1} by choosing $\varepsilon$ small enough depending on $\eta$.
\end{remark}

%\textcolor{red}{According to what we need, we can modify the previous Lemma and add details in the proof.}
\subsubsection{Step 3: 
%Integral estimate for $\gamma$-numbers associated to the Lipschitz function $g$
Pushing estimates on $\beta_{1,\pi}$ from $E$ down to the graph of $g$
}\label{ss:Step3}\phantom{}\vspace{0.1cm}\\
The proof in this section follows the outline in \cite[Section 6]{Hahlomaa}, which in turn is modeled after the Euclidean predecessor \cite[Section 13]{David1}. However, there is a crucial difference in that the upper bound in Lemma \ref{l:HLem6.1} is expressed in terms of the \emph{projection} $\beta$-numbers, rather than the potentially larger ``full'' $\beta$-numbers.

As in \cite[Section 6]{Hahlomaa}, we use Lemma \ref{lem_propertiesofg} (i),(iii) to extend the map $g$ constructed in the previous step to a $C\sqrt\eta$-Lipschitz function defined on $\R^k$ with support in $U_{-1}$, where the latter is defined as in \eqref{eq:Uj} for $j=-1$. We then introduce new flatness coefficients associated to $g$: for $p\in\R^k,\, t>0$, we set
\[\gamma(p,t):=t^{-k-1}\inf_a\int_{B^k(p,t)}|g(u)-a(u)|\,du,\]
where the infimum is taken over all affine functions $a:\R^k\to\R^{2n-k}$, and integration is with respect to the Euclidean $k$-dimensional Hausdorff measure on $\mathbb{R}^k$. By Lemma \ref{lem_propertiesofg} (i), if $\eta$ is chosen small enough, then 
\begin{equation}\label{estimateforgamma}
\gamma(p,t)\leq 2t^{-k-1}\inf_M\int_{B^k(p,t)}d_{\rm Eucl}((u,g(u)),M)\,du
\end{equation}
being the infimum computed among all $k$-planes $M$ in $\R^{2n}$.\medskip\\
The next lemma, which improves \cite[Lemma 6.1]{Hahlomaa}, shows that actually it is possible to bound the $\gamma$-numbers by means of the projection $\beta$-numbers associated to $E$.
\begin{lemma}\label{l:HLem6.1}
    Let $T=K_0 \diam (Q(\mathcal{S}))/2$. Let $\{U_i\}_{i\in \N}$ be as Section \ref{sss:Step 2}, \eqref{eq:Uj}.
    There exists a positive constant $C=C(K_0)>0$ such that
\begin{equation}\label{eq_step3}
\int_0^T \int_{U_1} \gamma(p,t)^2\, dp\frac{dt}{t}\leq C\varepsilon \mu(Q(\mathcal{S})) + C\varepsilon^{-6k}\int_{K_0 Q(\mathcal{S})}\int_{h_{\mathcal{S}}(x)/K_0}^T \beta_{1,\pi,A(2n,k)}(x, K_0 t)^2 \frac{dt}{t}d\mu(x).
\end{equation}
\end{lemma}

%\textcolor{magenta}{(Remember to discuss the case $\beta=0$)} \textcolor{blue}{Consider this case separately or argue by adding $\delta$ and letting $\delta\to 0$.}
\begin{proof}
For any $p\in U_1$ and $H(p)/60<t\leq T$, by \eqref{def_H} there exists a point $z_{p,t}\in Q(\mathcal S)$ such that $|p-P(z_{p,t})|\leq 60 t$. For any $z\in B(z_{p,t},t)\cap E$ and $\delta>0$, we find $W=W_{z,p,t,\delta}\in A(2n,k)$ such that
%The proof follows verbatim the one in \cite[Lemma 6.1]{Hahlomaa}.
%, upon observing that in estimate $(31)$ one can choose the plane $V=V_{p,t}\in \mathcal{V}_k$ such that
\begin{equation}\label{eq_Vz,p,t}
\int_{B(z,K_0t)} d_{\rm Eucl}(\pi(x), W)\, d\mu(x)\leq (K_0 t)^{k+1}[\beta_{1,\pi,A(2n,k)}(z, K_0t)+\delta],
\end{equation}
by definition of $\beta_{1,\pi,A(2n,k)}$-numbers. Thanks to \eqref{estimateforgamma}, we get that
\begin{align}\label{estimateforgamma2}\gamma(p,t)&\leq 2t^{-k-1}\int_{B^k(p,t)}d_{\rm Eucl}((u,g(u)),W)\,du%\\ &=2t^{-k-1}\left(\int_{B^k(p,t)\cap P(Z)}d_{\rm Eucl}((u,g(u)),W)\,du+\int_{B^k(p,t)\setminus P(Z)}d_{\rm Eucl}((u,g(u)),W)\,du\right)
.\end{align}
The core of the proof is now to estimate the integral 
with respect to $\mathcal{H}^k$
on the Euclidean ball $B^k(p,t)$ by the corresponding one with respect to $\mu=\mathcal{H}^k|_E$ on  $B(z, K_0t)$, in order to combine \eqref{eq_Vz,p,t} and \eqref{estimateforgamma2} and get a pointwise estimate for $\gamma(p,t)$ in terms of $\beta_{1,\pi, A(2n,k)}(z, K_0t)$. If $u\in B^k(p,t)\cap P(Z)$, then there exists a (unique, by Lemma \ref{l:HahlLem4.3}) $x\in K_0Q(\mathcal S)$ such that $P(x)=u$. Since $x\in Z$, then $h(x)=0$ and $\pi(x)=(u,g(u))\in\R^{2n}$.
It then follows by Lemma \ref{l:HahlLem4.3} and the triangular inequality that $d(x,z)\leq K_0t$, provided that $K_0$ is chosen big enough. By 1-Lipschitz continuity of $P$, we have $\mathcal{H}_{\mathrm{Eucl}}^k(\Omega)\leq  
P_{\sharp} \mu(\Omega) $ for $\Omega\subset B^k(p,t)\cap P(Z)$, and combining all this with \eqref{eq_Vz,p,t}, we find 
\begin{equation}\label{estimateinP(Z)}\begin{aligned}\int_{B^k(p,t)\cap P(Z)}d_{\rm Eucl}((u,g(u)),W)\,du&\leq\int_{B^k(p,t)\cap P(Z)}d_{\rm Eucl}((u,g(u)),W)\,dP_{\sharp} \mu(u) \\&\leq \int_{P(B(z,K_0t))}d_{\rm Eucl}((u,g(u)),W)\,dP_{\sharp} \mu(u)\\ &\leq \int_{B(z,K_0t)}d_{\rm Eucl}(\pi(x),W)\,d\mu(x)\\ &\leq (K_0t)^{k+1}[\beta_{1,\pi,A(2n,k)}(z,K_0t)+\delta].\end{aligned}\end{equation}
On the other hand, we need to estimate $d_{\rm Eucl}((u,g(u)),W)$ for $u\in B^k(p,t)\setminus P(Z)$. This corresponds to \cite[Lemma 6.3, inequality (34), Lemma 6.4 and Lemma 6.5]{Hahlomaa}: here we limit ourselves to collect the intermediate estimates together. Let $I(p,t):=\{i\in I_0:R_i\cap B^k(p,t)\neq \emptyset\}$ and notice that $B^k(p,t)\setminus P(Z)=\bigcup_{i\in I(p,t)}(B^k(p,t)\cap R_i)$, %recall that $\R^k\setminus P(Z)=\bigcup_{i\in\N} R_i$, as well as 
since $B^k(p,t)\subset U_0$. Combining this observation with the aforementioned estimates in \cite{Hahlomaa} one gets
\[\begin{split}&\int_{B^k(p,t)\setminus P(Z)}d_{\rm Eucl}((u,g(u)),W)\,du
%
%\sum_{i\in I(p,t)}\int_{B^k(p,t)\cap R_i}d_{\rm Eucl}((u,g(u)),W)\,du\\&\stackrel{\text{Lemma 6.3}}\leq \sum_{i\in I(p,t)}\int_{B^k(p,t)\cap R_i}d_{\rm Eucl}((u,g(u)),\pi(V_{Q_i}))\,du + \sum_{i\in I(p,t)}\mathcal L^k(R_i)\sup_{w\in V_{Q_i}\cap \,Q_i(C\diam(Q_i))}d_{\rm Eucl}(\pi(w),W)\\ &\stackrel{(34)}\lesssim \sum_{i\in I(p,t)} \sqrt\varepsilon \diam(R_i)^{k+1}+\sum_{i\in I(p,t)}\mathcal L^k(R_i)\sup_{w\in V_{Q_i}\cap \,Q_i(C\diam(Q_i))}d_{\rm Eucl}(\pi(w),W)\\&\stackrel{\text{Lemma 6.4}}\lesssim \sum_{i\in I(p,t)} \sqrt\varepsilon \diam(R_i)^{k+1}+\sum_{i\in I(p,t)}\varepsilon \diam(R_i)\mathcal  L^k(R_i)+\sum_{i\in I(p,t)} \varepsilon^{-3k}\mathcal L^k(R_i)\left(\fint_{2Q_i}d_{\rm Eucl}(\pi(x),W)^{\frac{1}{3}}\,d\mu\right)^3\\
%&\stackrel{\text{Lemma 6.5}}
\\&\lesssim \sum_{i\in I(p,t)} \sqrt\varepsilon \diam(R_i)^{k+1}+\sum_{i\in I(p,t)}\varepsilon \diam(R_i)\mathcal  L^k(R_i)+\varepsilon^{-3k}\int_{B(z,K_0 t)}d_{\rm Eucl}(\pi(x),W)\,d\mu.\end{split}\]
We can now use our refined estimate \eqref{eq_Vz,p,t} to get
\begin{equation}\label{estimateoutP(Z)}
    \int_{B^k(p,t)\setminus P(Z)}d_{\rm Eucl}((u,g(u)),W)\,du\lesssim \sum_{i\in I(p,t)}\sqrt\varepsilon \diam(R_i)^{k+1}+\varepsilon^{-3k}\,t^{k+1}[\beta_{1,\pi,A(2n,k)}(z, K_0t)+\delta].
\end{equation}
Combining \eqref{estimateforgamma2}, \eqref{estimateinP(Z)} and \eqref{estimateoutP(Z)}, using the arbitrariness of $\delta>0$, we conclude that
\begin{equation}
   \label{eq_estimateforgamma} 
\gamma(p,t)\lesssim \varepsilon^{-3k}\beta_{1,\pi,A(2n,k)}(z, K_0t)+\sqrt\varepsilon\, t^{-k-1}\sum_{i\in I(p,t)} \diam(R_i)^{k+1} .\end{equation}
The final part of the proof, which consists in deducing the integral estimate \eqref{eq_step3} from the pointwise inequality \eqref{eq_estimateforgamma}, can be verbatim carried out as in Hahlomaa's paper, starting from equation \cite[(38)]{Hahlomaa} onward, simply replacing the horizontal $\beta$-numbers $\beta_{1,\mathcal V_k}$ with the projection $\beta$-numbers $\beta_{1,\pi,A(2n,k)}$.
\end{proof}

\subsubsection{Step 4: Bounding $\mu(Q(\mathcal{S}))$ for $\mathcal{S}\in \mathcal{F}_3$ from above in terms of the $\beta_{1,\pi,A(2n,k)}$-numbers}\label{sss:Step4}
In order to prove a packing condition for trees in $\mathcal F_3$ defined at the beginning of Section \ref{ss:BuildingCoronization}, we estimate the measure of the corresponding top cubes in terms of the $\beta_{1,\pi,A(2n,k)}$-numbers.  This will allow us in Step 5 to make use of our assumption $E\in \mathrm{GLem}(\beta_{1,\pi,A(2n,k)},2)$ to obtain the conclusion.
\begin{lemma}\label{l:HLem7.1}
   Let $\mathcal{S}\in \mathcal{F}_3$. Then
    \begin{equation}\label{eq:lemma7.1}
\mu(Q(\mathcal{S}))<{\varepsilon}^{-6k-1}  \int_{{K_0}Q(\mathcal{S})}\int_{h_{\mathcal{S}}(x)/{K_0}}^{{K_0}\mathrm{diam}(Q(\mathcal{S}))}\beta_{1,\pi,A(2n,k)}(x,{K_0}t)^2\frac{dt}{t}\,d\mu(x).   
    \end{equation}
\end{lemma}
\begin{proof}
With Lemma \ref{l:HLem6.1} in place, the proof (by contradiction) reduces to a statement about the $\gamma$-numbers associated to $g$. This corresponds to \cite[Section 7]{Hahlomaa} and \cite[Sections 11 and 14]{David1}. In fact, let $\mathcal S\in\mathcal F_3$ and notice that we can assume $V_{Q(\mathcal S)}=\{x\in\H^n:x_i=0\,\text{ for all }i>k\}$ since all quantities in Lemma \ref{l:HLem7.1} are invariant under group translations and rotations, including $\beta_{1,\pi,A(2n,k)}$, cf. Lemmas \ref{lem_translations} and \ref{l:RotInv}. By contradiction, if \eqref{eq:lemma7.1} does not holds, then we would get by Lemma \ref{l:HLem6.1} 
\begin{equation}\label{eq:H41}\int_0^{K_0 \diam(Q(\mathcal S))/2} \int_{U_1} \gamma(p,t)^2\, dp\frac{dt}{t}\leq C(K_0)\varepsilon \mu(Q(\mathcal{S})).\end{equation}
This is exactly \cite[equation (41)]{Hahlomaa}; thus the rest of the proof works verbatim the same as in Hahlomaa's paper. Arguments in this vein have appeared for instance also in \cite[Section 5]{MR1709304} or \cite[Section 14]{David1}.  One uses \eqref{eq:H41} to estimate the oscillation of the function $g$ in a way that is incompatible with the assumption $\mathcal S\in\mathcal F_3$, recalling that $\mathcal S\in\mathcal F_3$ implies that a large (in terms of $\mu$-measure) part of $Q(\mathcal{S})$ is covered by cubes $Q\in m_2(\mathcal{S})$, which have at least one child $R$ with $\angle(V_R, V_{Q(\mathcal S)})>1+\eta$. Indeed, \eqref{eq:H41}, together with an application of Calder\'{o}n's reproducing formula (see for instance \cite[Chapter 5]{MR2359017}) for the function $g$, yields
\begin{displaymath}
    \mu\left(\bigcup_{Q\in m_2(\mathcal{S})}Q\right) \leq C(K_0) \varepsilon^{1/3}\mu(Q(\mathcal{S})), 
\end{displaymath}
provided that the parameters $K_0, \eta, \varepsilon, K$ are chosen appropriately. This leads to a contradiction with the definition of $\mathcal F_3$, by selecting $\varepsilon$ small enough depending on $K_0$.
%\textcolor{magenta}{CHECK}.
\end{proof}

\subsubsection{Step 5: A Carleson packing condition for top cubes of trees $\mathcal S\in \mathcal F_3$}\label{sss:Step5}
%\textcolor{magenta}{It is in this step that our weaker assumption plays a role and we have to verify that they suffice to run the argument in the proof of \cite[Lemma 6.1]{Hahlomaa}}
In this step we prove the final ingredient for the construction of the corona decomposition by horizontal planes. We explain how to derive the Carleson packing condition for the set $\mathcal F_3$, which allows us to deduce property \eqref{carlesonpackingcond} in the definition of coronization.
This can be seen as a variant of arguments in \cite[Section 12]{David1}.
Here we make use of our new assumption $E \in \mathrm{GLem}(\beta_{1,\pi,A(2n,k)}, 2)$. 
%More precisely
\begin{lemma}\label{l:HahLem8.1}
There is a constant $C=C(\varepsilon, K_0)>0$ such that
\begin{displaymath}
    \sum_{\mathcal{S}\in \mathcal{F}_3,Q(\mathcal{S})\subseteq R}\mu(Q(\mathcal{S}))\leq C \mu(R),\quad R\in \mathcal{D}.
\end{displaymath}
\end{lemma}
\begin{proof}
Proceeding exactly as in \cite[Lemma 8.1]{Hahlomaa}, the Ahlfors regularity of $E$ gives a bound $N=N(K_0)\in\N$ such that, for every $x\in E$ and any $t\in\R $, there are less than $N$ trees $\mathcal S\in\mathcal F$ satisfying $x\in K_0 Q(\mathcal S)$ and $h_\mathcal S(x)/K_0<t<K_0\,\diam(Q(\mathcal S))$. Combining Lemma \ref{l:HLem7.1} of the previous step with this observation, one estimates
\[\begin{split}\sum_{\mathcal{S}\in \mathcal{F}_3,Q(\mathcal{S})\subseteq R}\mu(Q(\mathcal{S}))&\leq\sum_{\mathcal{S}\in \mathcal{F}_3,Q(\mathcal{S})\subseteq R}{\varepsilon}^{-6k-1}  \int_{{K_0}Q(\mathcal{S})}\int_{h_{\mathcal{S}}(x)/{K_0}}^{{K_0}\mathrm{diam}(Q(\mathcal{S}))}\beta_{1,\pi,A(2n,k)}(x,{K_0}t)^2\frac{dt}{t}\,d\mu(x)\\ &\leq N\varepsilon^{-6k-1}\int_{K_0R}\int_{0}^{K_0\diam(R)}\beta_{1,\pi,A(2n,k)}(x,{K_0}t)^2\frac{dt}{t}\,d\mu(x).\end{split}\]
By \eqref{eq:H_(14)} and by our main assumption $E \in \mathrm{GLem}(\beta_{1,\pi,A(2n,k)}, 2)$, we get the desired final bound.
\end{proof}
\printbibliography
\end{document}